\documentclass[10pt]{amsart}
\usepackage{amssymb}
\usepackage{bm}
\usepackage{graphicx}
\usepackage[centertags]{amsmath}
\usepackage{amsfonts}
\usepackage{amsthm}
\newtheorem{thm}{Theorem}
\newtheorem{cor}[thm]{Corollary}
\newtheorem{lem}[thm]{Lemma}
\newtheorem{prop}[thm]{Proposition}
\newtheorem{fact}[thm]{Fact}
\newtheorem{defn}[thm]{Definition}
\newtheorem{ithm}{Theorem}
\newtheorem{icor}[ithm]{Corollary}
\newtheorem{idefn}[ithm]{Definition}
\theoremstyle{definition}
\newtheorem{rem}{Remark}

\newtheorem{notation}{Notation}
\newcommand{\rr}{\mathbb{R}}
\newcommand{\nn}{\mathbb{N}}

\newcommand{\ee}{\varepsilon}
\newcommand{\con}{\smallfrown}
\newcommand{\kk}{\mathcal{K}}

\newcommand{\bt}{\mathbb{N}^{<\mathbb{N}}}
\newcommand{\ct}{2^{<\mathbb{N}}}

\newcommand{\sg}{\sigma}

\newcommand{\incr}{\mathrm{Incr}}
\newcommand{\decr}{\mathrm{Decr}}
\newcommand{\llf}{\mathcal{L}_{\mathbf{f},f}}
\newcommand{\lbf}{\mathcal{L}_{\mathbf{f}}}

\newcommand{\splp}{\hat{S}_+(2^\nn)}
\newcommand{\splm}{\hat{S}_-(2^\nn)}
\newcommand{\alex}{\hat{A}(2^\nn)}
\newcommand{\dcantor}{\hat{D}(2^\nn)}
\newcommand{\dsplit}{\hat{D}\big(S(2^\nn)\big)}
\newcommand{\aaa}{\mathcal{A}}
\newcommand{\bbb}{\mathcal{B}}
\newcommand{\ccc}{\mathcal{C}}

\newcommand{\ppp}{\mathcal{P}}

\newcommand{\iii}{\mathcal{I}}
\newcommand{\jjj}{\mathcal{J}}
\newcommand{\fff}{\mathcal{F}}
\newcommand{\kkk}{\mathcal{K}}

\newcommand{\bseq}{\mathbf{f}=\{f_n\}_n}
\newcommand{\lbfc}{\mathcal{L}_{\mathbf{f},\mathcal{C}}}

\begin{document}

\title[A classification of separable Rosenthal compacta]{A
classification of separable Rosenthal compacta and
its applications}
\author{S. A. Argyros, P. Dodos and V. Kanellopoulos}
\address{National Technical University of Athens, Faculty of Applied
Sciences, Department of Mathematics, Zografou Campus, 157 80,
Athens, Greece} \email{sargyros@math.ntua.gr, pdodos@math.ntua.gr,
bkanel@math.ntua.gr}


\maketitle

\tableofcontents





\section{Introduction}

The theory of Rosenthal compacta, namely of compact subsets of the
first Baire class on a Polish space $X$, was initiated with the
pioneering work of H. P. Rosenthal \cite{Ro1}. Significant
contribution of many researchers coming from divergent areas has
revealed the deep structural properties of this class. Our aim is
to study some aspects of separable Rosenthal compacta, as well as,
to present some of their applications.

The present work consists of three parts. In the first one we
determine the prototypes of separable Rosenthal compacta and we
provide a classification theorem. The second part concerns an
extension of a theorem of S. Todor\v{c}evi\'{c} included in his
profound study of Rosenthal compacta \cite{To1}. The last one
is devoted to applications.

Our results, concerning the first part, are mainly included in
Theorems \ref{A} and \ref{B} below. Roughly speaking, we assert
that there exist seven separable Rosenthal compacta such that
every $\kk$ in the same class contains one of them in a very
canonical way. We start with the following.
\begin{idefn}
\label{id1} $\mathrm{(a)}$ Let $I$ be a countable set and $X, Y$ be Polish
spaces. Let $\{f_i\}_{i\in I}$ and $\{g_i\}_{i\in I}$ be two
pointwise bounded families of real-valued functions on $X$ and $Y$
respectively, indexed by the set $I$. We say that $\{f_i\}_{i\in
I}$ and $\{g_i\}_{i\in I}$ are equivalent if the natural map
$f_i\mapsto g_i$ is extended to a topological homeomorphism
between $\overline{\{f_i\}}^p_{i\in I}$ and
$\overline{\{g_i\}}^p_{i\in I}$. \\
$\mathrm{(b)}$ Let $X$ be a Polish space and $\{f_t\}_{t\in \ct}$ be
relatively compact in $\mathcal{B}_1(X)$. We say that
$\{f_t\}_{t\in \ct}$ is minimal if for every dyadic
subtree $S=(s_t)_{t\in\ct}$ of the Cantor tree $\ct$, the families
$\{f_t\}_{t\in\ct}$ and $\{f_{s_t}\}_{t\in\ct}$ are equivalent.
\end{idefn}
Related to the above notions, the following is proved.
\begin{ithm}
\label{A} $\mathrm{(a)}$ Up to equivalence, there are exactly seven minimal
families.\\
$\mathrm{(b)}$ For every family $\{f_t\}_{t\in\ct}$ relatively compact in
$\mathcal{B}_1(X)$, with $X$ Polish, there exists a regular dyadic
subtree $S=(s_t)_{t\in\ct}$ of $\ct$ such that
$\{f_{s_t}\}_{t\in\ct}$ is equivalent to one of the seven minimal
families.
\end{ithm}
For any of the seven minimal families the corresponding pointwise
closure is a separable Rosenthal compact containing the family as
a discrete set. We denote them as follows
\[ A(\ct), \ 2^{\leqslant\nn}, \ \splp, \ \splm, \ \alex, \
\dcantor, \text{ and } \dsplit. \] The precise description of the
families and the corresponding compacta is given in \S 4.3. The
first two in the above list are metrizable spaces. The next two
are hereditarily separable, non-metrizable and mutually
homeomorphic (thus, the above defined notion of equivalence of
families is stronger than saying that the corresponding closures
are homeomorphic). The space $\splp$, and so the space $\splm$
as well, can be realized as a closed subspace of the split interval $S(I)$.
Following \cite{E}, we shall denote by $A(2^\nn)$ the one point
compactification of the Cantor set $2^\nn$. The space $\alex$ is
the standard separable extension of $A(2^\nn)$ (see \cite{Pol2},
\cite{Ma}). This is the only not first countable space from the
above list. The space $\dcantor$ is the separable extension of the
Alexandroff duplicate of the Cantor set $D(2^\nn)$, as it was
described in \cite{To1}. Finally, the space $\dsplit$ can be
realized as a closed subspace of the Helly space. Its accumulation
points is the closure of the standard uncountable discrete subset
of the Helly space.

Theorem \ref{A} is essentially a success of the infinite-dimensional
Ramsey Theory for trees and perfect sets. There is a long history
on the interaction between Ramsey Theory and Rosenthal compacta,
which can be traced back to the classical J. Farahat's proof
\cite{F} of H. P. Rosenthal's $\ell_1$ Theorem \cite{Ro-ell1} and
its tree extension due to J. Stern \cite{S}. This interaction
was further expanded by S. Todor\v{c}evi\'{c} in \cite{To1}
with the use of the parameterized Ramsey Theory for
perfect sets.

The new Ramsey theoretic ingredient in the proof of Theorem \ref{A}
is a result concerning partitions of two classes of antichains
of the Cantor tree, which we call \textit{increasing} and
\textit{decreasing}. We will briefly comment on the proof
of Theorem \ref{A} and the critical role of this result.
One starts with a family $\{f_t\}_{t\in\ct}$ relatively
compact in $\bbb_1(X)$. A first topological
reduction shows that in order to understand the closure of
$\{f_t\}_{t\in\ct}$ in $\rr^X$ it is enough to determine all
subsets of the Cantor tree for which the corresponding subsequence
of $\{f_t\}_{t\in\ct}$ is pointwise convergent. A second reduction
shows that it is enough to determine only a cofinal subset of
convergent subsequences. One is then led to analyze which classes
of subsets of the Cantor tree are Ramsey and cofinal. First, we observe
that every infinite subset of $\ct$ either contains an infinite chain,
or an infinite antichain. It is well-known, and goes back to Stern,
that chains are Ramsey. On the other hand, the set of all
antichains is not. However, the classes of increasing
and decreasing antichains are Ramsey and, moreover,
they are cofinal in the set of all antichains. Using the
above properties of chains and of increasing and decreasing
antichains we are able to have a satisfactory
control over the convergent subsequences of $\{f_t\}_{t\in\ct}$.
Finally, repeated applications of F. Galvin's theorem on
partitions of doubletons of perfect sets of reals
permit us to fully canonicalize the topological
behavior of $\{f_t\}_{t\in\ct}$ yielding the
proof of Theorem \ref{A}.

A direct consequence of Theorem \ref{A}(b) is that for every
separable Rosenthal compact and for every countable dense subset
$\{f_t\}_{t\in\ct}$ of it, there exists a regular dyadic subtree
$S=(s_t)_{t\in\ct}$ such that the pointwise closure of
$\{f_{s_t}\}_{t\in\ct}$ is homeomorphic to one of the above
described compacta. In general, for a given countable dense subset
$\{f_n\}_n$ of a separable Rosenthal compact $\kk$, we say that
one of the minimal families canonically embeds into $\kk$ with
respect to $\{f_n\}_n$ if there exists an increasing injection
$\phi:\ct\to\nn$ such that the family $\{f_{\phi(t)}\}_{t\in\ct}$
is equivalent to it. The next theorem is a supplement of Theorem
\ref{A}, showing that the minimal families can be chosen to
characterize certain topological properties of $\kk$.
\begin{ithm}
\label{B} Let $\kk$ be a separable Rosenthal compact and $\{f_n\}_n$
a countable dense subset of $\kk$.
\begin{enumerate}
\item[(a)] If $\kk$ consists of bounded functions in
$\mathcal{B}_1(X)$, is metrizable and non-separable in the
supremum norm, then $2^{\leqslant\nn}$ canonically embeds into
$\kk$ with respect to $\{f_n\}_n$ such that its image is norm
non-separable.
\item[(b)] If $\kk$ is non-metrizable and hereditarily separable,
then either $\splp$, or $\splm$ canonically embeds into $\kk$
with respect to $\{f_n\}_n$.
\item[(c)] If $\kk$ is not hereditarily separable and first
countable, then either $\dcantor$, or $\dsplit$ canonically embeds into
$\kk$ with respect to $\{f_n\}_n$.
\item[(d)] If $\kk$ is not first countable, then $\alex$ canonically
embeds into $\kk$ with respect to $\{f_n\}_n$.
\end{enumerate}
In particular, if $\kk$ is non-metrizable, then one of the
non-metrizable prototypes canonically embeds into $\kk$ with
respect to any dense subset of $\kk$.
\end{ithm}
Part (a) is an extension of the classical Ch. Stegall's result
\cite{St}, which led to the characterization of the Radon-Nikodym
property in dual Banach spaces. We mention that Todor\v{c}evi\'{c}
\cite{To1} has shown that in case (b) above the split interval
$S(I)$ embeds into $\kk$. It is an immediate consequence of the
above theorem that every not hereditarily separable $\kk$ contains
an uncountable discrete subspace of the size of the continuum, a
result due to R. Pol \cite{Pol}. The proofs of parts (a), (b) and
(c) use variants of Stegall's fundamental construction, similar in
spirit as in the work of G. Godefroy and M. Talagrand \cite{GT}.
Part (d) is a consequence of a more general structural result
concerning non-$G_\delta$ points which we are about to describe.
To this end, we start with the following.
\begin{idefn}
\label{idan1} Let $\kk$ be a separable Rosenthal compact on a
Polish space $X$ and $\ccc$ a closed subspace of $\kk$. We say
that $\ccc$ is an analytic subspace if there exist a countable
dense subset $\{f_n\}_n$ of $\kk$ and an analytic subset $A$ of
$[\nn]$ such that the following are satisfied.
\begin{enumerate}
\item[(1)] For every $L\in A$ the accumulation points of the set
$\{f_n:n\in L\}$ in $\rr^X$ is a subset of $\ccc$. \item[(2)] For
every $g\in\ccc$ which is an accumulation point of $\kk$ there
exists $L\in A$ with $g\in\overline{\{f_n\}}^p_{n\in L}$.
\end{enumerate}
\end{idefn}
Observe that every separable Rosenthal compact $\kk$ is an analytic
subspace of itself with respect to any countable dense set.
Let us point out that while the class of analytic subspaces
is strictly wider than the class of separable ones, it shares all
the structural properties of the separable Rosenthal compacta.
This will become clear in the sequel.

A natural question raised by the above definition is whether the
concept of an analytic subspace depends on the choice of the
countable dense subset of $\kk$. We believe that it is
independent. This is supported by the fact that it is indeed the
case for analytic subspaces of separable Rosenthal compacta in
$\bbb_1(X)$ with $X$ compact metrizable.

To state our results concerning analytic subspaces, we also need the
following.
\begin{idefn}
\label{idan2} Let $\kk$ be a separable Rosenthal compact,
$\{f_n\}_n$ a countable dense subset of $\kk$ and $\ccc$ a closed
subspace of $\kk$. We say that one of the prototypes $\kk_i$
$(1\leq i\leq 7)$ canonically embeds into $\kk$ with respect to
$\{f_n\}_n$ and $\ccc$, if there exists a subfamily
$\{f_t\}_{t\in\ct}$ of $\{f_n\}_n$ which is equivalent to the
canonical dense family of $\kk_i$ and such that all accumulation
points of $\{f_t\}_{t\in\ct}$ are in $\ccc$.
\end{idefn}
The following theorem describes the structure of not first
countable analytic subspaces.
\begin{ithm}
\label{ialex} Let $\kk$ be a separable Rosenthal compact,
$\ccc$ an analytic subspace of $\kk$ and $\{f_n\}_n$ a
countable dense subset of $\kk$ witnessing the analyticity
of $\ccc$. Let also $f\in\ccc$ be a non-$G_\delta$ point of
$\ccc$. Then $\alex$ canonically embeds into $\kk$ with
respect to $\{f_n\}_n$ and $\ccc$ and such that $f$ is
the unique non-$G_\delta$ point of its image.
\end{ithm}
Theorem \ref{ialex} is the last step of a series of results
initiated by a fruitful problem concerning the character
of points in separable Rosenthal compacta, posed by
R. Pol \cite{Pol}. The first decisive step towards the solution
of this problem was made by A. Krawczyk \cite{Kra}. He
proved that a point $f\in\kk$ is non-$G_\delta$ if and only if
the set
\[ \llf=\{ L\in [\nn]: (f_n)_{n\in L} \text{ is pointwise convergent to } f \} \]
is co-analytic non-Borel. His analysis revealed a
fundamental construction, which we call \textit{Krawczyk tree}
($K$-tree) with respect to the given point $f$ and any countable
dense subset $\mathbf{f}=\{f_n\}_n$ of $\kk$. He actually showed
that there exists a subfamily  $\{f_t\}_{t\in\bt}$ of $\{f_n\}_n$
such that the following are fulfilled.
\begin{enumerate}
\item[(P1)] For every $\sg\in \nn^\nn$, $f\notin \overline{
\{f_{\sg|n}\}}^p_{n}$. \item[(P2)] If $A\subseteq \bt$ is such
that $f\notin \overline{ \{f_t\}}^p_{t\in A}$, then for $n\in\nn$
there exist $t_0, ..., t_k\in \nn^n$ such that $A$ is almost
included in the set of the successors of the $t_i$'s.
\end{enumerate}
Using $K$-trees, the second named author has shown that the set
\[ \lbf=\{ L\in[\nn]: (f_n)_{n\in L} \text{ is pointwise convergent}\} \]
is complete co-analytic if there exists a non-$G_\delta$ point $f\in\kk$
(\cite{D}). Let also point out that the deep effective version of
G. Debs' theorem \cite{De} yields that for any separable Rosenthal compact
the set $\lbf$ contains a Borel cofinal subset.

There are strong evidences, as Debs' theorem mentioned above, that
separable Rosenthal compacta are definable objects, hence, they
are naturally connected to descriptive set theory (see also
\cite{ADK}, \cite{B}, \cite{D}). One of the first results
illustrating this connection was proved in the late 70's by G.
Godefroy \cite{Go}, asserting that a separable compact $\kk$ is
Rosenthal if and only if $C(\kk)$ is an analytic subset of $\rr^D$
for every countable dense subset $D$ of $\kk$. Related to this, R.
Pol has conjectured that a separable Rosenthal compact $\kk$
embeds into $\mathcal{B}_1(2^\nn)$ if and only if $C(\kk)$ is a
Borel subset of $\rr^D$ (see \cite{Ma} and \cite{Pol2}). It is worth
mentioning that for a separable $\kk$ in $\mathcal{B}_1(2^\nn)$, for
every countable dense subset $\{f_n\}_n$ of $\kk$ and every
$f\in\kk$, there exists a Borel cofinal subset of the
corresponding set $\llf$, a property not shared by all separable
Rosenthal compacta.

The final step to the solution of Pol's problem was made by S.
Todor\v{c}evi\'{c} \cite{To1}. He proved that if $f$ is a
non-$G_\delta$ point of $\kk$, then the space $A(2^\nn)$ is
homeomorphic to a closed subset of $\kk$ with $f$ as the unique
limit point. His remarkable proof involves metamathematical
arguments like forcing method and absoluteness.

Let us proceed to a discussion on the proof of Theorem
\ref{ialex}. The first decisive step is the following theorem,
concerning the existence of $K$-trees.
\begin{ithm}
\label{iktrees} Let $\kk$, $\ccc$, $\{f_n\}_n$ and $f\in\ccc$ be as
in Theorem \ref{ialex}. Then there exists a $K$-tree
$\{f_t\}_{t\in\bt}$ with respect to the point $f$ and the dense
sequence $\{f_n\}_n$ such that for every $\sg\in \nn^\nn$ all
accumulation points of the set $\{f_{\sg|n}:n\in\nn\}$ are in
$\ccc$.
\end{ithm}
The proof of the above result is a rather direct extension of the
results of A. Krawczyk from \cite{Kra} and is based on the key
property of bi-sequentiality, established for separable Rosenthal
compacta by R. Pol \cite{Pol3}. We will briefly comment on some
further properties of the $K$-tree $\{f_t\}_{t\in\bt}$ obtained by
Theorem \ref{iktrees}. To this end, let us call an antichain
$\{t_n\}_n$ of $\bt$ a fan if there exist $s\in\bt$ and a strictly
increasing sequence $(m_n)_n$ in $\nn$ such that
$s^{\con}m_n\sqsubseteq t_n$ for every $n\in\nn$. Let us also say
that an antichain $\{t_n\}_n$ converges to $\sg\in \nn^\nn$ if for
every $k\in\nn$ the set $\{t_n\}_n$ is almost contained in the set
of the successors of $\sg|k$. Property (P2) of $K$-trees implies
that for every fan $\{t_n\}_n$ of $\bt$ the sequence $(f_{t_n})_n$
must be pointwise convergent to $f$. This fact combined with the
bi-sequentiality of separable Rosenthal compacta yields the
following.
\begin{enumerate}
\item[(P3)] For every $\sg\in \nn^\nn$ there exists an antichain
$\{t_n\}_n$ of $\bt$ which converges to $\sg$ and such that the
sequence $(f_{t_n})_n$ is pointwise convergent to $f$.
\end{enumerate}
In the second crucial step, we use the infinite dimensional
extension of Hindman's theorem, due to K. Milliken \cite{Mil}, to
pass to an infinitely splitting subtree $T$ of $\bt$ such that for
every $\sg\in [T]$ the corresponding antichain $\{t_n\}_n$,
described in property (P3), is found in a canonical way. We should
point out that, although Milliken's theorem is a result concerning
partitions of block sequences, it can be also considered as a
partition theorem for a certain class of infinitely splitting
subtrees of $\bt$. This fact was first realized by W. Henson, in
his alternative proof of Stern's theorem (see \cite{Od}), and it
is used in the proof of Theorem \ref{ialex} in a similar spirit.
The proof of Theorem \ref{ialex} is completed by choosing an
appropriate dyadic subtree $S$ of $T$ and applying the
canonicalization method (Theorem \ref{A}) to the family
$\{f_s\}_{s\in S}$.

The following consequence of Theorem \ref{ialex} describes the
universal property of $\alex$ among all fundamental prototypes.
\begin{icor}
\label{ikmk} Let $\kk$ be a non-metrizable separable Rosenthal
compact and $D=\{f_n\}_n$ a countable dense subset of $\kk$. Then
the space $\alex$ canonically embeds into $\kk-\kk$ with respect
to $D-D$ and with the constant function $0$ as the unique
non-$G_\delta$ point.
\end{icor}
We notice that the above corollary remains valid within the class
of analytic subspaces.

The embedding of $\alex$ in an analytic subspace $\ccc$ of a
separable Rosenthal compact $\kk$ yields unconditional families of
elements of $\ccc$ as follows.
\begin{ithm}
\label{iuncC} Let $\kk$ be a separable Rosenthal compact on a
Polish space $X$ consisting of bounded functions. Let also $\ccc$
be an analytic subspace of $\kk$ having the constant function $0$
as a non-$G_\delta$ point. Then there exists a family
$\{f_\sg:\sg\in 2^\nn\}$ in $\ccc$ which is 1-unconditional
in the supremum norm, pointwise discrete and having $0$ as
unique accumulation point.
\end{ithm}
The proof of Theorem \ref{iuncC} follows by Theorem \ref{ialex}
and the ``perfect unconditionality theorem" form \cite{ADKbanach}.

A second application concerns representable Banach spaces, a class
introduced in \cite{GT} and closely related to separable Rosenthal
compacta.
\begin{ithm}
\label{repres} Let $X$ be a non-separable representable Banach
space. Then $X^*$ contains an unconditional family of size
$|X^*|$.
\end{ithm}
We also introduce the concept of spreading and level unconditional
tree bases. This notion is implicitly contained in
\cite{ADKbanach} where their existence was established in every
separable Banach space not containing $\ell_1$ and with
non-separable dual. We present some extensions of this result in
the framework of separable Rosenthal compacta.

We proceed to discuss how this work is organized. In \S 2, we set
up our notations concerning trees and we present the Ramsey
theoretic preliminaries needed in the rest of the paper. In the
next section we define and study the classes of increasing and
decreasing antichains. The main result in \S 3 is Theorem
\ref{ap1} which establishes the Ramsey properties of these
classes. Section 4 is exclusively devoted to the proof of Theorem
\ref{A}. It consists of four subsections. In the first one, we
prove a theorem (Theorem \ref{rt1} in the main text) which is the
first step towards the proof of Theorem \ref{A}. Theorem \ref{rt1}
is a consequence of the Ramsey and structural properties of chains
and of increasing and decreasing antichains. In \S 4.2, we
introduce the notion of equivalence of families of functions and
we provide a criterion for establishing it. As we have already
mentioned, in \S 4.3 we describe the seven minimal families. The
proof of Theorem \ref{A} is completed in \S 4.4.

In \S 5.1, we introduce the class of analytic subspaces of
separable Rosenthal compacta and we present some of their
properties, while in \S 5.2 we study separable Rosenthal compacta
in $\bbb_1(2^\nn)$. In \S 6, we present parts (a), (b) and (c) of
Theorem \ref{B}. Actually, Theorem \ref{B} is proved for the wider
class of analytic subspaces and within the context of Definition
\ref{idan2}. The precise statement is as follows.
\begin{ithm}
\label{C+} Let $\kk$ be a separable Rosenthal compact, $\ccc$ an
analytic subspace of $\kk$ and $\{f_n\}_n$ a countable dense
subset of $\kk$ witnessing the analyticity of $\ccc$.
\begin{enumerate}
\item[(a)] If $\ccc$ is metrizable in the pointwise topology,
consists of bounded functions and it is non-separable in the
supremum norm of $\mathcal{B}_1(X)$, then $2^{\leqslant\nn}$
canonically embeds into $\kk$ with respect to $\{f_n\}_n$ and
$\ccc$, such that its image is norm non-separable. \item[(b)] If
$\ccc$ is hereditarily separable and non-metrizable, then either
$\splp$, or $\splm$ canonically embeds into $\kk$ with respect to
$\{f_n\}_n$ and $\ccc$. \item[(c)] If $\ccc$ is not hereditarily
separable and first countable, then either $\dcantor$, or
$\dsplit$ canonically embeds into $\kk$ with respect to
$\{f_n\}_n$ and $\ccc$.
\end{enumerate}
\end{ithm}
Section 7 is devoted to the study of not first countable analytic
subspaces. In \S 7.1 we prove Theorem \ref{iktrees}, while \S 7.2
is devoted to the proof of Theorem \ref{ialex}. The final section
is devoted to applications and in particular to the proofs of
Theorem \ref{iuncC} and Theorem \ref{repres}.

We thank Stevo Todor\v{c}evi\'{c} for his valuable remarks and
comments.


\section{Ramsey properties of perfect sets and of subtrees
of the Cantor tree}

The aim of this section is to present the Ramsey theoretic
preliminaries needed in the rest of the paper, as well as, to set
up our notation concerning trees.

Ramsey Theory for trees was initiated with the fundamental
Halpern-L\"{a}uchli Partition Theorem \cite{HL}. The original
proof was using metamathematical arguments. The proof avoiding
metamathematics was given in \cite{AFK}. Partition theorems
related to the ones presented in this section can be found in the
work of K. Milliken \cite{Mil2}, A. Blass \cite{Bl} and A.
Louveau, S. Shelah and B. Veli\v{c}kovi\'{c} \cite{LSV}.


\subsection{Notations}

We let $\nn=\{0,1,2,...\}$. By $[\nn]$ we denote the set of all
infinite subsets of $\nn$, while for every $L\in[\nn]$ by $[L]$ we
denote the set of all infinite subsets of $L$. If $k\geq 1$ and
$L\in[\nn]$, then $[L]^k$ stands for the set of all finite subsets
of $L$ of cardinality $k$.\\
\textbf{A.} By $\ct$ we denote the set of all finite sequences of
$0$'s and $1$'s (the empty sequence is included).
We view $\ct$ as a tree equipped with the (strict) partial order
$\sqsubset$ of extension. If $t\in\ct$, then the length $|t|$ of
$t$ is defined to be the cardinality of the set $\{s:s\sqsubset
t\}$. If $s,t\in \ct$, then by $s^\con t$ we denote their
concatenation. Two nodes $s, t$ are said to be \textit{comparable}
if either $s\sqsubseteq t$ or $t\sqsubseteq s$; otherwise are
said to be \textit{incomparable}. A subset of $\ct$
consisting of pairwise comparable nodes is said to be
a \textit{chain} while a subset of $\ct$ consisting of
pairwise incomparable nodes is said to be an
\textit{antichain}. For every $x\in 2^\nn$ and every $n\geq 1$ we
set $x|n=\big( x(0),..., x(n-1)\big)\in \ct$ while
$x|0=(\varnothing)$. For $x,y\in (\ct\cup 2^\nn)$ with $x\neq y$
we denote by $x\wedge y$ the $\sqsubset$-maximal node $t$ of $\ct$
with $t\sqsubseteq x$ and $t\sqsubseteq y$. Moreover, we write
$x\prec y$ if $w^{\con}0\sqsubseteq x$ and $w^{\con}1\sqsubseteq
y$, where $w=x\wedge y$. The ordering $\prec$ restricted on
$2^\nn$ is the usual lexicographical ordering of the Cantor set.\\
\textbf{B.} We view every subset of $\ct$ as a \textit{subtree}
with the induced partial ordering. A subtree $T$ of $\ct$
is said to be \textit{pruned} if for every $t\in T$ there exists
$s\in T$ with $t\sqsubset s$. It is said to be \textit{downwards
closed} if for every $t\in T$ and every $s\sqsubset t$ we have
that $s\in T$. For a subtree $T$ of $\ct$ (not necessarily
downwards closed) we set $\hat{T}=\{ s:\exists t\in T \text{ with
} s\sqsubseteq t\}$. If $T$ is downwards closed, then the \textit{body}
$[T]$ of $T$ is the set $\{ x\in 2^\nn: x|n\in T \ \forall n\}$.\\
\textbf{C.} Let $T$ be a (not necessarily downwards closed) subtree of
$\ct$. For every $t\in T$ by $|t|_T$ we denote the cardinality of
the set $\{ s\in T: s\sqsubset t\}$ and for every $n\in\nn$ we set
$T(n)=\{ t\in T: |t|_T=n\}$. Moreover, for every $t_1, t_2\in T$
by $t_1\wedge_T t_2$ we denote the $\sqsubset$-maximal node $w$ of
$T$ such that $w\sqsubseteq t_1$ and $w\sqsubseteq t_2$. Notice
that $t_1\wedge_T t_2\sqsubseteq t_1\wedge t_2$. Given two
subtrees $S$ and $T$ of $\ct$, we say that $S$ is a
\textit{regular} subtree of $T$ if $S\subseteq T$ and for every
$n\in\nn$ there exists $m\in\nn$ such that $S(n)\subseteq T(m)$.
For a regular subtree $T$ of $\ct$, the \textit{level set} $L_T$
of $T$ is the set $\{ l_n: T(n)\subseteq 2^{l_n}\}\subseteq \nn$.
Notice that for every $x\in [\hat{T}]$ and every $m\in\nn$
we have that $x|m\in T$ if and only if $m\in L_T$. Hence,
the chains of $T$ are naturally identified with the
product $[\hat{T}]\times [L_T]$. A pruned subtree $T$ of $\ct$ is
said to be \textit{skew} if for every $n\in\nn$ there exists at
most one splitting node of $T$ in $T(n)$ with exactly two immediate
successors in $T$; it is said to be \textit{dyadic} if every
$t\in T$ has exactly two immediate successors in $T$.
We observe that a subtree $T$ of the Cantor
tree is regular dyadic if there exists a (necessarily unique)
bijection $i_T:\ct\to T$ such that the following are satisfied.
\begin{enumerate}
\item[(1)] For all $t_1, t_2\in \ct$ we have $|t_1|=|t_2|$ if and
only if $|i_T(t_1)|_T=|i_T(t_2)|_T$. \item[(2)] For all $t_1,
t_2\in \ct$ we have $t_1\sqsubset t_2$ (respectively $t_1\prec
t_2$) if and only if $i_T(t_1)\sqsubset i_T(t_2)$ (respectively
$i_T(t_1)\prec i_T(t_2)$).
\end{enumerate}
When we write $T=(s_t)_{t\in\ct}$, where $T$ is a regular dyadic
subtree of $\ct$, we mean that $s_t=i_T(t)$ for all $t\in\ct$.
Finally we notice the following. If $T$ is a regular dyadic
subtree of $\ct$ and $R$ is a regular dyadic subtree of $T$, then
$R$ is a regular dyadic subtree of $\ct$ too.


\subsection{Partitions of trees}

We begin by recalling the following notion from \cite{Ka}.
\begin{defn}
Let $T$ be a skew subtree of $\ct$. We define
$f_T:\nn\to\{1,2\}^{<\nn}$ as follows. For every $n\in\nn$, let
$T(n)=\{s_0\prec ...\prec s_{m-1}\}$ be the $\prec$-increasing
enumeration of $T(n)$. We set $f_T(n)=(e_0,...,e_{m-1})\in
\{1,2\}^m$, where for every $i\in\{0,...,m-1\}$, $e_i$ is the
cardinality of the set of the immediate successors of $s_i$ in
$T$. The function $f_T$ will be called the code of the tree $T$.
If $f:\nn\to \{1,2\}^{<\nn}$ is a function such that there exists
a skew tree $T$ with $f=f_T$, then $f$ will be called a skew tree
code.
\end{defn}
For instance, if $f_T(n)=(1)$ for all $n\in\nn$, then the tree $T$
is a chain. Also, if $f_T(0)=(2)$ and $f_T(n)=(1,1)$ for all
$n\geq 1$, then $T$ consists of two chains. Moreover, observe that
if $T$ and $S$ are two skew subtrees of $\ct$ with $f_T= f_S$,
then $T$ and $S$ are isomorphic with respect to both $\prec$ and
$\sqsubset$. If $f$ is a skew tree code and $T$ is a regular
dyadic subtree of $\ct$, then by $[T]_f$ we denote the set of all
regular skew subtrees of $T$ of code $f$. It is easy to see that
the set $[T]_f$ is a Polish subspace of $2^T$. Also observe that
if $R$ is a regular dyadic tree of $T$, then $[R]_f= [T]_f\cap
2^R$. We will need the following theorem, which is a consequence
of Theorem 46 in \cite{Ka}.
\begin{thm}
\label{tka1} Let $T$ be a regular dyadic subtree of $\ct$, $f$ a
skew tree code and $A$ be an analytic subset of $[T]_f$. Then
there exists a regular dyadic subtree $R$ of $T$ such that either
$[R]_f\subseteq A$, or $[R]_f\cap A=\varnothing$.
\end{thm}
For a regular dyadic subtree $T$ of $\ct$, denote by
$[T]_{\mathrm{chains}}$ the set of all infinite chains of $T$.
Theorem \ref{tka1} includes the following result due to J. Stern
\cite{S}, A.W. Miller, S. Todor\v{c}evi\'{c} \cite{Mi} and J.
Pawlikowski \cite{Pa}.
\begin{thm}
\label{chains} Let $T$ be a regular dyadic subtree of $\ct$ and
$A$ be an analytic subset of $[T]_{\mathrm{chains}}$. Then there
exists a regular dyadic subtree $R$ of $T$ such that either
$[R]_{\mathrm{chains}}\subseteq A$, or $[R]_{\mathrm{chains}} \cap
A =\varnothing$.
\end{thm}
Theorem \ref{tka1} will essentially be applied to the following
classes of skew subtrees.
\begin{defn}
\label{ad2} Let $T$ be a regular dyadic subtree of $\ct$. A
subtree $S$ of $T$ will be called increasing (respectively
decreasing) if the following are satisfied.
\begin{enumerate}
\item[(a)] $S$ is uniquely rooted, regular, skew and pruned.
\item[(b)] For every $n\in\nn$, there exists a splitting node of
$S$ in $S(n)$, which is the $\prec$-maximum (respectively
$\prec$-minimum) node of $S(n)$ and it has two immediate
successors in $S$.
\end{enumerate}
The class of increasing (respectively decreasing) subtrees of $T$
will be denoted by $[T]_\incr$ (respectively $[T]_\decr$).
\end{defn}
It is easy to see that every increasing (respectively decreasing)
subtree is of fixed code. Thus Theorem \ref{tka1} can be applied
to give the following.
\begin{cor}
\label{cka1} Let $T$ be a regular dyadic subtree of $\ct$ and $A$
be an analytic subset of $[T]_{\incr}$. Then there exists a
regular dyadic subtree $R$ of $T$ such that either
$[R]_\incr\subseteq A$, or $[R]_\incr\cap A=\varnothing$.
Similarly for the case of $[T]_\decr$.
\end{cor}
The above corollary may be considered as a parameterized version
of the Louveau-Shelah-Veli\v{c}kovi\'{c} theorem \cite{LSV}.


\subsection{Partitions of perfect sets}

For every subset $X$ of $2^\nn$, by $[X]^2$ we denote the set of
all doubletons of $X$. We identify $[X]^2$ with the set of all
$(\sg,\tau)\in X^2$ with $\sg\prec \tau$. We will need
the following partition theorem due to F. Galvin
(see \cite{Kechris}, Theorem 19.7).
\begin{thm}
\label{galvin2} Let $P$ be a perfect subset of $2^\nn$. If $A$ is a
subset of $[P]^2$ with the Baire property, then there exists
a perfect subset $Q$ of $P$ such that either $[Q]^2\subseteq A$, or
$[Q]^2\cap A=\varnothing$.
\end{thm}


\section{Increasing and decreasing antichains of a regular dyadic tree}


In this section we define the increasing and decreasing antichains
and we establish their fundamental Ramsey properties.

As we have already seen in \S 2 the class of infinite chains of
the Cantor tree is Ramsey. On the other hand an analogue of
Theorem \ref{chains} for infinite antichains is not valid. For instance,
color an antichain $(t_n)_n$ of $\ct$ red if $t_0\prec t_1$; otherwise color
it blue. It is easy to see that this is an open partition, yet there is no
dyadic subtree of $\ct$ all of whose antichains are monochromatic. So, it is
necessary, in order to have a Ramsey result for antichains, to restrict
our attention to those which are monotone with respect to $\prec$.
Still, however, this is not enough. To see this, consider the set of
all $\prec$-increasing antichains and color such an antichain $(t_n)_n$
red if $|t_0|\leq |t_1\wedge t_2|$; otherwise color it blue. Again we
see that this is an open partition which is not Ramsey.

The following definition incorporates all the restrictions
indicated by the above discussion and which are, as we shall see,
essentially the only obstacles to a Ramsey result for antichains.
\begin{defn}
\label{ad1} Let $T$ be a regular dyadic subtree of the Cantor tree
$\ct$. An infinite antichain $(t_n)_n$ of $T$ will be called
increasing if the following conditions are satisfied.
\begin{enumerate}
\item[(1)] For all $n,m\in\nn$ with $n<m$, $|t_n|_T<|t_m|_T$.
\item[(2)] For all $n,m,l\in\nn$ with $n<m<l$, $|t_n|_T\leq
|t_m\wedge_T t_l|_T$. \item[(3I)] For all $n,m\in\nn$ with $n<m$,
$t_n\prec t_m$.
\end{enumerate}
The set of all increasing antichains of $T$ will be denoted by
$\incr(T)$. Similarly, an infinite antichain $(t_n)_n$ of $T$ will
be called decreasing if (1) and (2) above are satisfied and (3I)
is replaced by the following.
\begin{enumerate}
\item[(3D)] For all $n,m\in\nn$ with $n<m$, $t_m\prec t_n$.
\end{enumerate}
The set of all decreasing antichains of $T$ will be denoted by
$\decr(T)$.
\end{defn}
The classes of increasing and decreasing antichains of $T$ have
the following crucial stability properties.
\begin{lem}
\label{al1} Let $T$ be a regular dyadic subtree of $\ct$. Then the
following hold.
\begin{enumerate}
\item[(1)] (Hereditariness) Let $(t_n)_n\in \incr(T)$ and
$L=\{l_0<l_1<...\}$ be an infinite subset of $\nn$. Then
$(t_{l_n})_n\in\incr(T)$. Similarly, if $(t_n)_n\in\decr(T)$, then
$(t_{l_n})_n\in\decr(T)$. \item[(2)] (Cofinality) Let $(t_n)_n$ be
an infinite antichain of $T$. Then there exists
$L=\{l_0<l_1<...\}\in[\nn]$ such that either
$(t_{l_n})_n\in\incr(T)$ or $(t_{l_n})_n\in\decr(T)$. \item[(3)]
(Coherence) We have $\incr(T)=\incr(\ct)\cap 2^T$ and similarly
for the decreasing antichains.
\end{enumerate}
\end{lem}
\begin{proof}
(1) It is straightforward.\\
(2) The point is that all three properties in the definition of
increasing and decreasing antichains are cofinal in the set of all
antichains of $T$. Indeed, let $(t_n)_n$ be an infinite antichain
of $T$. Clearly there exists $N\in [\nn]$ such that the sequence
$\big( |t_n|_T\big)_{n\in N}$ is strictly increasing. Moreover, by
Ramsey's theorem, there exists $M\in [N]$ such that the sequence
$(t_n)_{n\in M}$ is either $\prec$-increasing or
$\prec$-decreasing. Finally, to see that condition (2) in
Definition \ref{ad1} is cofinal, let
\[ A=\big\{ (n,m,l)\in [M]^3: |t_n|_T\leq |t_m\wedge_T t_l|_T \big\} \]
By Ramsey's Theorem again, there exists $L\in [M]$ such that
either $[L]^3\subseteq A$ or $[L]^3\cap A=\varnothing$. We claim
that $[L]^3\subseteq A$, which clearly completes the proof. Assume
not, i.e. $[L]^3\cap A=\varnothing$. Let $n=\min L$ and
$L'=L\setminus\{n\}\in [L]$. Let also $k=|t_n|_T$. Then for every
$(m,l)\in [L']^2$ we have that $|t_m\wedge_T t_l|_T<k$. The set
$\{ t\in T: |t|_T<k\}$ is finite. Hence, by another application of
Ramsey's theorem, there exist $s\in T$ with $|s|_T<k$ and
$L''\in[L']$ such that for every $(m,l)\in [L'']^2$ we have that
$s=t_m\wedge_T t_l$. But this is clearly impossible as the tree
$T$ is dyadic.\\
(3) First we observe the following. As the tree $T$ is regular,
for every $t, s\in T$ we have $|t|_T< |s|_T$ (respectively
$|t|_T=|s|_T$) if and only if $|t|<|s|$ (respectively $|t|=|s|$).

Now, let $(t_n)_n\in \incr(T)$. In order to show that
$(t_n)_n\in\incr(\ct)\cap 2^T$ it is enough to prove
that for every $n<m<l$ we have $|t_n|\leq |t_m\wedge t_l|$.
By the above remarks, we have that $|t_n|\leq |t_m\wedge_T t_l|$.
As $t_m\wedge_T t_l\sqsubseteq t_m\wedge t_l$, we are done.

Conversely assume that $(t_n)_n\in\incr(\ct)\cap 2^T$. Again it is
enough to check that condition (2) in Definition \ref{ad1} is
satisfied. So let $n<m<l$. There exist $s_m,s_l\in T$ with
$|s_m|_T=|s_l|_T=|t_n|_T$, $s_m\sqsubseteq t_m$ and $s_l\sqsubseteq
t_l$. We claim that $s_m=s_l$. Indeed, if not, then $|t_m\wedge
t_l|=|s_m\wedge s_l|<|t_n|$ contradicting the fact that the
antichain $(t_n)_n$ is increasing in $\ct$. It follows that
$t_m\wedge_T t_l \sqsupseteq s_m$, and so,
$|s_m|_T=|t_n|_T\leq |t_m\wedge_T t_l|_T$, as desired.
The proof for the decreasing antichains is identical.
\end{proof}
A corollary of property (3) of Lemma \ref{al1} is the following.
\begin{cor}
\label{ac1} Let $T$ be a regular dyadic subtree of $\ct$ and $R$ a
regular dyadic subtree of $T$. Then $\incr(R)=\incr(T)\cap 2^R$
and $\decr(R)=\decr(T)\cap 2^R$.
\end{cor}
We notice that for every regular dyadic subtree $T$ of the Cantor
tree $\ct$ the sets $\incr(T)$ and $\decr(T)$ are Polish subspaces
of $2^T$. The main result of this section is the following.
\begin{thm}
\label{ap1} Let $T$ be a regular dyadic subtree of $\ct$ and $A$
be an analytic subset of $\incr(T)$ (respectively of $\decr(T)$).
Then there exists a regular dyadic subtree $R$ of $T$ such that
either $\incr(R)\subseteq A$, or $\incr(R)\cap A=\varnothing$
(respectively, either $\decr(R)\subseteq A$, or $\decr(R)\cap
A=\varnothing$).
\end{thm}
We notice that, after a first draft of the present paper, S.
Todor\v{c}evi\'{c} informed us that he is also aware of the above
result with a proof based on K. Milliken's theorem for strong
subtrees (\cite{To2}).

The proof of Theorem \ref{ap1} is based on Corollary \ref{cka1}.
The method is to reduce the coloring of $\incr(T)$ (respectively
of $\decr(T)$) in Theorem \ref{ap1}, to a coloring of the class
$[T]_{\incr}$ (respectively $[T]_{\decr}$) of increasing
(respectively decreasing) regular subtrees of $T$ (see Definition
\ref{ad2}). To this end, we need the following easy fact
concerning the classes $[T]_\incr$ and $[T]_\decr$.
\begin{fact}
\label{af1} Let $T$ be a regular dyadic subtree of $\ct$. If $S\in
[T]_\incr$ or $S\in [T]_\decr$, then for every $n\in\nn$ we have
$|S(n)|=n+1$.
\end{fact}
As we have indicated, the crucial fact in the present setting is
that there is a canonical correspondence between $[T]_\incr$ and
$\incr(T)$ (and similarly for the decreasing antichains) which we
are about to describe. For every $S\in [\ct]_\incr$ or $S\in
[\ct]_\decr$ and every $n\in\nn$, let $\{s^n_0\prec...\prec
s^n_n\}$ be the $\prec$-increasing enumeration of $S(n)$. Define
$\Phi: [\ct]_\incr\to \incr(\ct)$ by
\[ \Phi(S)= (s^{n+1}_{n})_n. \]
It is easy to see that $\Phi$ is a well-defined continuous map.
Respectively, define $\Psi:[\ct]_\decr\to \decr(\ct)$ by
$\Psi(S)=(s^{n+1}_1)_n$. Again it is easy to see that $\Psi$ is
well-defined and continuous.
\begin{lem}
\label{al2} Let $T$ be a regular dyadic subtree of $\ct$. Then
$\Phi\big( [T]_\incr\big)=\incr(T)$ and
$\Psi\big([T]_\decr\big)=\decr(T)$.
\end{lem}
\begin{proof}
We shall give the proof only for the case of increasing subtrees.
The proof of the other case is similar. First, we notice that for
every $S\in[T]_\incr$ we have $\Phi(S)\in \incr(\ct)\cap 2^T$, and
so, by Lemma \ref{al1}(3) we get that $\Phi\big( [T]_\incr\big)
\subseteq\incr(T)$. Conversely, let $(t_n)_n\in\incr(T)$.
\bigskip

\noindent \textsc{Claim 1.} \textit{For every $n<m<l$ we have
$t_n\wedge_T t_m=t_n\wedge_T t_l$.}
\bigskip

\noindent \textit{Proof of the claim.} Let $n<m<l$. By condition
(2) in Definition \ref{ad1}, there exists $s\in T$ with
$|s|_T=|t_n|_T$ and such that $s\sqsubseteq t_m\wedge_T t_l$.
Moreover, observe that $t_n\prec s$, as $t_n\prec t_m$. It follows
that $t_n\wedge_T t_m = t_n\wedge_T s = t_n\wedge_T t_l$, as
claimed. \hfill $\lozenge$
\bigskip

\noindent For every $n\in\nn$, we set $c_n=t_n\wedge_T t_{n+1}$.
\bigskip

\noindent \textsc{Claim 2.} \textit{For every $n<m$ we have
$c_n\sqsubset c_m$. That is, the sequence $(c_n)_n$ is an infinite
chain of $T$. }
\bigskip

\noindent \textit{Proof of the claim.} Let $n<m$. By Claim 1, we
get that $c_n$ and $c_m$ are compatible, since $c_n=t_n\wedge_T
t_m$ and, by definition, $c_m= t_m\wedge_T t_{m+1}$. Now notice
that $|c_n|_T< |t_n|_T\leq |t_m\wedge_T t_{m+1}|_T =|c_m|_T$.
\hfill $\lozenge$
\bigskip

\noindent For every $n\geq 1$, let $c'_n$ be the unique node of
$T$ such that $c'_n\sqsubseteq c_n$ and $|c'_n|_T=|t_{n-1}|_T$. We
define recursively $S\in [T]_\incr$ as follows. We set $S(0)=\{
c_0\}$ and $S(1)=\{ t_0, c'_1\}$. Assume that $S(n)=\{
s^n_0\prec...\prec s^n_n\}$ has been defined so as
$s^n_{n-1}=t_{n-1}$ and $s^n_n=c'_n$. For every $0\leq i\leq n-1$,
we chose nodes $s^{n+1}_i$ such that $s^n_i\sqsubset s^{n+1}_i$
and $|s^{n+1}_i|_T=|t_n|_T$. We set $S(n+1)=\{
s^{n+1}_0\prec...\prec s^{n+1}_{n-1}\prec t_n \prec c'_{n+1}\}$.
It is easy to check that $S\in [T]_\incr$ and that
$\Phi(S)=(t_n)_n$. The proof is completed.
\end{proof}
We are ready to give the proof of Theorem \ref{ap1}.
\begin{proof}[Proof of Theorem \ref{ap1}]
Let $A$ be an analytic subset of $\incr(T)$. By Lemma \ref{al2},
the set $B=\Phi^{-1}(A)\cap [T]_\incr$ is an analytic subset of
$[T]_\incr$. By Corollary \ref{cka1}, there exists a regular
dyadic subtree $R$ of $T$ such that either $[R]_\incr\subseteq B$
or $[R]_\incr\cap B=\varnothing$. By Lemma \ref{al2}, the first
case implies that $\incr(R)=\Phi\big([R]_\incr\big)\subseteq
\Phi(B)\subseteq A$, while the second that $\incr(R)\cap
A=\Phi\big([R]_\incr\big)\cap A=\varnothing$. The proof for the
case of decreasing antichains is similar.
\end{proof}


\section{Canonicalizing sequential compactness of trees of functions}


The present section consists of four subsections. In the first
one, using the Ramsey properties of chains and of increasing
and decreasing antichains, we prove a strengthening of a result
of J. Stern \cite{S}. In the second one, we introduce the notion
of equivalence of families of functions and we provide a criterion
for establishing it. In the third subsection, we define the
seven minimal families. The last subsection is devoted to the
proof of the main result of the section, concerning the canonical
embedding in any separable Rosenthal compact of one of the
minimal families.

\subsection{Sequential compactness of trees of functions}

We start with the following definition.
\begin{defn}
\label{rd1} Let $L$ be an infinite subset of $\ct$ and $\sg\in
2^\nn$. We say that $L$ converges to $\sg$ if for every $k\in\nn$
the set $L$ is almost included in the set
$\{t\in\ct:\sg|k\sqsubseteq t\}$. The element $\sg$ will be called
the limit of the set $L$. We write $L\to\sg$ to denote that $L$
converges to $\sg$.
\end{defn}
It is clear that the limit of a subset $L$ of $\ct$ is unique, if
it exists.
\begin{fact}
\label{rf11} Let $(t_n)_n$ be an increasing (respectively
decreasing) antichain of $\ct$. Then $(t_n)_n$ converges to $\sg$,
where $\sg$ is the unique element of $2^\nn$ determined by the
chain $(c_n)_n$ with $c_n=t_n\wedge t_{n+1}$ (see the proof of
Lemma \ref{al2}).
\end{fact}
We will also need the following notations.
\begin{notation}
For every $L\subseteq\ct$ infinite and every $\sg\in 2^\nn$ we
write $L\prec^* \sg$ if the set $L$ is almost included in the
set $\{t:t\prec \sg\}$. Respectively, we write $L\preceq^* \sg$
if $L$ is almost included in the set $\{t: t\prec \sg\}\cup
\{ \sg|n:n\in\nn\}$. The notations $\sg\prec^* L$ and $\sg\preceq^*
L$ have the obvious meaning. We also write $L\subseteq^* \sg$ if for
all but finitely many $t\in L$ we have $t\sqsubset\sg$, while by
$L\perp \sg$ we mean that the set $L\cap \{\sg|n: n\in\nn\}$ is
finite.
\end{notation}
The following fact is essentially a consequence of Lemma \ref{al1}(2).
\begin{fact}
\label{factnew} If $L$ is an infinite subset of $\ct$ and $\sg\in
2^\nn$ are such that $L\to \sg$ and $L\prec^* \sg$ (respectively
$\sg\prec^* L$), then every infinite subset of $L$ contains an
increasing (respectively decreasing) antichain converging to
$\sg$.
\end{fact}
The aim of this subsection is to give a proof of the following
result.
\begin{thm}
\label{rt1} Let $X$ be a Polish space and $\{f_t\}_{t\in\ct}$ be a
family relatively compact in $\mathcal{B}_1(X)$. Then there exist
a regular dyadic subtree $T$ of $\ct$ and a family $\{g^0_\sg,
g^+_\sg, g^-_\sg: \sg\in P\}$, where $P=[\hat{T}]$, such that for
every $\sg\in P$ the following are satisfied.
\begin{enumerate}
\item[(1)] The sequence $(f_{\sg|n})_{n\in L_T}$ converges
pointwise to $g^0_{\sg}$ (recall that $L_T$ stands for the level
set of $T$). \item[(2)] For every sequence $(\sg_n)_n$ in $P$
converging to $\sg$ such that $\sg_n\prec \sg$ for all $n\in\nn$,
the sequence $(g^{\ee_n}_{\sg_n})_n$ converges pointwise to
$g^+_{\sg}$ for any choice of $\ee_n\in\{0,+,-\}$. If such a
sequence $(\sg_n)_n$ does not exist, then $g^+_\sg=g^0_\sg$.
\item[(3)] For every sequence $(\sg_n)_n$ in $P$ converging to
$\sg$ such that $\sg\prec \sg_n$ for all $n\in\nn$, the sequence
$(g^{\ee_n}_{\sg_n})_n$ converges pointwise to $g^-_{\sg}$ for any
choice of $\ee_n\in\{0,+,-\}$. If such a sequence $(\sg_n)_n$ does
not exist, then $g^-_\sg=g^0_\sg$. \item[(4)] For every infinite
subset $L$ of $T$ converging to $\sg$ with $L\prec^* \sg$, the
sequence $(f_{t})_{t\in L}$ converges pointwise to $g^+_{\sg}$.
\item[(5)] For every infinite subset $L$ of $T$ converging to
$\sg$ with $\sg\prec^* L$, the sequence $(f_t)_{t\in L}$ converges
pointwise to $g^-_{\sg}$.
\end{enumerate}
Moreover, the functions $0,+,-:P\times X\to\rr$ defined by
\[ 0(\sg,x)=g^0_\sg(x), \ +(\sg,x)=g^+_\sg(x), \ -(\sg,x)=g^-_\sg(x) \]
are all Borel.
\end{thm}
Before we proceed to the proof of Theorem \ref{rt1} we notice the
following fact (the proof of which is left to the reader).
\begin{fact}
\label{rf1} $\mathrm{(1)}$ Let $A_1=(t^1_n)_n$ and $A_2=(t^2_n)_n$
be two increasing (respectively decreasing) antichains of $\ct$
converging to the same $\sg\in 2^\nn$. Then there exists an
increasing (respectively decreasing) antichain $(t_n)_n$ of $\ct$
converging to $\sg$ such that $t_{2n}\in A_1$ and $t_{2n+1}\in A_2$
for every $n\in\nn$.\\
$\mathrm{(2)}$ Let $(\sg_n)_n$ be a sequence in $2^\nn$ converging
to $\sg\in 2^\nn$. For every $n\in\nn$, let $N_n=(t^n_k)_k$ be a
sequence in $\ct$ converging to $\sg_n$. If $\sg_n\prec \sg$
(respectively $\sg_n\succ\sg$) for all $n$, then there exist an increasing
(respectively decreasing) antichain $(t_m)_m$ and $L=\{n_m:m\in\nn\}$
such that $(t_m)_m$ converges to $\sg$ and $t_m\in N_{n_m}$ for
every $m\in\nn$.
\end{fact}
\begin{proof}[Proof of Theorem \ref{rt1}]
Our hypotheses imply that for every sequence $(g_n)_n$ belonging
to the closure of $\{f_t\}_{t\in\ct}$ in $\rr^X$, there exists a
subsequence of $(g_n)_n$ which is pointwise convergent. Consider
the following subset $\Pi_1$ of $[\ct]_{\mathrm{chains}}$ defined
by
\[ \Pi_1=\big\{ c\in [\ct]_{\mathrm{chains}}: \text{the sequence }
(f_{t})_{t\in c} \text{ is pointwise convergent}\big\}. \] Then
$\Pi_1$ is a co-analytic subset of $[\ct]_{\mathrm{chains}}$ (see
\cite{S}). Applying Theorem \ref{chains} and invoking our
hypotheses, we get a regular dyadic subtree $T_1$ of $\ct$ such
that $[T_1]_{\mathrm{chains}}\subseteq \Pi_1$. Now consider the
subset $\Pi_2$ of $\incr(T_1)$, defined by
\[ \Pi_2= \big\{ (t_n)_n\in\incr(T_1): \text{the sequence }
(f_{t_n})_n \text{ is pointwise convergent}\big\}. \] Again
$\Pi_2$ is co-analytic (this can be checked with similar arguments
as in \cite{S}). Applying Theorem \ref{ap1}, we get a regular
dyadic subtree $T_2$ of $T_1$ such that $\incr(T_2)\subseteq
\Pi_2$. Finally, applying Theorem \ref{ap1} for the decreasing
antichains of $T_2$ and the color
\[ \Pi_3= \big\{ (t_n)_n\in\decr(T_2): \text{the sequence }
(f_{t_n})_n \text{ is pointwise convergent}\big\}, \]
we obtain a regular dyadic subtree $T$ of $T_2$ such that,
setting $P=[\hat{T}]$, the following are satisfied.
\begin{enumerate}
\item[(i)] For every increasing antichain $(t_n)_n$ of $T$, the
sequence $(f_{t_n})_n$ is pointwise convergent. \item[(ii)] For
every decreasing antichain $(t_n)_n$ of $T$, the sequence
$(f_{t_n})_n$ is pointwise convergent. \item[(iii)] For every
$\sg\in P$, the sequence $(f_{\sg|n})_{n\in L_T}$ is pointwise
convergent to a function $g^0_\sg$.
\end{enumerate}
We notice the following. By Fact \ref{rf1}(1), if $(t^1_n)_n$ and
$(t^2_n)_n$ are two increasing (respectively decreasing)
antichains of $T$ converging to the same $\sg$, then
$(f_{t^1_n})_n$ and $(f_{t^2_n})_n$ are both pointwise convergent
to the same function. For every $\sg\in P$, we define $g^+_\sg$ as
follows. If there exists an increasing antichain $(t_n)_n$ of $T$
converging to $\sg$, then we set $g^+_\sg$ to be the pointwise
limit of $(f_{t_n})_n$ (by the above remarks $g^+_\sg$ is
independent of the choice of $(t_n)_n$). Otherwise we set
$g^+_\sg=g^0_\sg$. Similarly we define $g^-_\sg$ to be the
pointwise limit of $(f_{t_n})_n$, with $(t_n)_n$ a decreasing
antichain of $T$ converging to $\sg$, if such an antichain exists.
Otherwise we set $g^-_\sg=g^0_\sg$. By Fact \ref{factnew} and the
above discussion, properties (i) and (ii) can be strengthened as
follows.
\begin{enumerate}
\item[(iv)] For every $\sg\in P$ and every infinite $L\subseteq T$
converging to $\sg$ with $L\prec^*\sg$, the sequence $(f_t)_{t\in
L}$ is pointwise convergent to $g^+_\sg$. \item[(v)] For every
$\sg\in P$ and every infinite $L\subseteq T$ converging to $\sg$
with $\sg\prec^*L$, the sequence $(f_t)_{t\in L}$ is pointwise
convergent to $g^-_\sg$.
\end{enumerate}
We claim that the tree $T$ and the family $\{g^0_\sg,g^+_\sg,
g^-_\sg:\sg\in P\}$ are as desired. First we check that properties
(1)-(5) are satisfied. Clearly we only have to check (2) and (3).
We will prove only property (2) (the argument is symmetric). We
argue by contradiction. So, assume that there exist a sequence
$(\sg_n)_n$ in $P$, $\sg\in P$ and $\ee_n\in\{0,+,-\}$ such that
$\sg_n\prec \sg$, $(\sg_n)_n$ converges to $\sg$ while the
sequence $(g^{\ee_n}_{\sg_n})_n$ does not converge pointwise to
$g^+_\sg$. Hence there exist $L\in [\nn]$ and an open neighborhood
$V$ of $g^+_\sg$ in $\rr^X$ such that $g^{\ee_n}_{\sg_n}\notin
\overline{V}$ for all $n\in L$. By definition, for every $n\in L$
we may select a sequence $(t^n_k)_k$ in $T$ such that for every
$n\in L$ the following hold.
\begin{enumerate}
\item[(a)] The sequence $N_n=(t^n_k)_k$ converges to $\sg_n$.
\item[(b)] The sequence $(f_{t^n_k})_k$ converges pointwise to
$g^{\ee_n}_{\sg_n}$. \item[(c)] For all $k\in \nn$, we have
$f_{t^n_k}\notin \overline{V}$. \item[(d)] The sequence
$(\sg_n)_{n\in L}$ converges to $\sg$ and $\sg_n\prec \sg$.
\end{enumerate}
By Fact \ref{rf1}(2), there exist a diagonal increasing antichain
$(t_m)_m$ converging to $\sg$. By (c) above, we see that
$(f_{t_m})_m$ is not pointwise convergent to $g^+_\sg$. This leads
to a contradiction by the definition of $g^+_\sg$.

Now we will check the Borelness of the maps $0, +$ and $-$.
Let $L_T= \{l_0<l_1<...\}$ be the increasing enumeration of the
level set $L_T$ of $T$. For every $n\in \nn$ define
$h_n:P\times X\to \rr$ by $h_n(\sg,x)=f_{\sg|l_n}(x)$.
Clearly $h_n$ is Borel. As for all $(\sg,x)\in P\times X$
we have
\[ 0(\sg,x)=g^0_\sg(x)=\lim_{n\in\nn} h_n(\sg,x) \]
the Borelness of $0$ is clear. We will only check the Borelness of
the function $+$ (the argument for the map $-$ is symmetric).
For every $n\in\nn$ and every $\sg\in P$, let $l_n(\sg)$ be the
lexicographically minimum of the closed set $\{\tau\in P:
\sg|l_n\sqsubset \tau\}$. The function $P\ni\sg\mapsto l_n(\sg)\in
P$ is clearly continuous. Invoking the definition of $g^+_\sg$ and
property (2) in the statement of the theorem we see that for all
$(\sg,x)\in P\times X$ we have
\[ +(\sg,x) = g^+_\sg(x) = \lim_{n\in\nn} g^0_{l_n(\sg)}(x) = \lim_{n\in\nn}
0\big( l_n(\sg),x\big). \]
Thus $+$ is Borel too and the proof is completed.
\end{proof}
\begin{rem}
We would like to point out that in order to apply the Ramsey
theory for trees in the present setting one has to know that all
the colors are sufficiently definable. This is also the reason why
the Borelness of the functions $0, +$ and $-$ is emphasized in
Theorem \ref{rt1}. As a matter of fact, we will need the full
strength of the Ramsey theory for trees and perfect sets, in the
sense that in certain situations the color will belong to the
$\sg$-algebra generated by the analytic sets. It should be noted
that this is in contrast with the classical Silver's theorem
\cite{Si} for which, most applications, involve Borel partitions.
\end{rem}


\subsection{Equivalence of families of functions}

Let us give the following definition.
\begin{defn}
\label{equiv} Let $I$ be a countable set and $X, Y$ be Polish
spaces. Let also $\{f_i\}_{i\in I}$ and $\{g_i\}_{i\in I}$ be two
pointwise bounded families of real-valued functions on $X$ and $Y$
respectively, indexed by the set $I$. We say that $\{f_i\}_{i\in
I}$ is equivalent to $\{g_i\}_{i\in I}$ if the map
\[ f_i\mapsto g_i \]
is extended to a topological homeomorphism between
$\overline{\{f_i\}}^p_{i\in I}$ and $\overline{\{g_i\}}^p_{i\in
I}$.
\end{defn}
The equivalence of the families $\{f_i\}_{i\in I}$ and
$\{g_i\}_{i\in I}$ is stronger than saying that
$\overline{\{f_i\}}^p_{i\in I}$ is homeomorphic to
$\overline{\{g_i\}}^p_{i\in I}$ (such an example will be given in
the next subsection). The crucial point in Definition \ref{equiv}
is that the equivalence of $\{f_i\}_{i\in I}$ and $\{g_i\}_{i\in
I}$ gives a natural homeomorphism between their closures.

The following lemma provides an efficient criterion for checking the
equivalence of families of Borel functions. We mention that in its
proof we will often make use of the Bourgain-Fremlin-Talagrand
theorem \cite{BFT} without making an explicit reference. From the
context it will be clear that this is what we use.
\begin{lem}
\label{l1} Let $I$ be a countable set and $X,Y$ be Polish spaces.
Let $\kk_1$ and $\kk_2$ be two separable Rosenthal compacta on $X$ and
$Y$ respectively. Let $\{f_i\}_{i\in I}$ and $\{g_i\}_{i\in I}$ be two
dense families of $\kk_1$ and $\kk_2$ respectively. Assume that
for every $i\in I$ the functions $f_i$ and $g_i$ are isolated in $\kk_1$
and $\kk_2$ respectively. Then the following are equivalent.
\begin{enumerate}
\item[(1)] The families $\{f_i\}_{i\in I}$ and $\{g_i\}_{i\in I}$
are equivalent. \item[(2)] For every $L\subseteq I$ infinite, the
sequence $(f_i)_{i\in L}$ converges pointwise if and only if the
sequence $(g_i)_{i\in L}$ does.
\end{enumerate}
\end{lem}
\begin{proof}
The direction $(1)\Rightarrow (2)$ is obvious. What remains is to
prove the converse. So assume that $(2)$ holds. Let $M\subseteq I$
infinite. We set $\kk_1^M=\overline{ \{f_i\}}^p_{i\in M}$ and
$\kk_2^M=\overline{\{g_i\}}^p_{i\in M}$. Notice that both
$\kk_1^M$ and $\kk_2^M$ are separable Rosenthal compacta. Our
assumptions imply that the isolated points of $\kk_1^M$ is
precisely the set $\{f_i:i\in M\}$ and similarly for $\kk_2^M$.
Define $\Phi_M:\kk_1^M\to \kk_2^M$ as follows. First, for every
$i\in M$ we set $\Phi_M(f_i)=g_i$. If $h\in \kk_1^M$ with
$h\notin\{ f_i:i\in M\}$, then there exists $L\subseteq M$
infinite such that $h$ is the pointwise limit of the sequence
$(f_i)_{i\in L}$. Define $\Phi_M(h)$ to be the pointwise limit of
the sequence $(g_i)_{i\in L}$ (by our assumptions this limit
exists). To simplify notation, let $\Phi=\Phi_I$.
\bigskip

\noindent \textsc{Claim.} \textit{Let $M\subseteq I$ infinite.
Then the following hold.
\begin{enumerate}
\item[(1)] The map $\Phi_M$ is well-defined, 1-1 and onto.
\item[(2)] We have $\Phi|_{\kk_1^M}=\Phi_M$.
\end{enumerate} }
\bigskip

\noindent \textit{Proof of the claim.} (1) Fix $M\subseteq I$
infinite. To see that $\Phi_M$ is well-defined, notice that for
every $h\in \kk_1^M$ with $h\notin\{f_i:i\in M\}$ and every $L_1,
L_2\subseteq M$ infinite with $h=\lim_{i\in L_1} f_i=\lim_{i\in
L_2} f_i$ it holds that $\lim_{i\in L_1} g_i=\lim_{i\in L_2} g_i$.
For if not, we would have that the sequence $(f_i)_{i\in L_1\cup
L_2}$ converges pointwise while the sequence $(g_i)_{i\in L_1\cup
L_2}$ does not, contradicting our assumptions.

We observe the following consequence of our assumptions and the
definition of the map $\Phi_M$. For every $h\in\kk^M_1$,
the point $h$ is isolated in $\kk^M_1$ if and only if $\Phi_M(h)$
is isolated in $\kk^M_2$. Using this we will show that $\Phi_M$ is
1-1. Indeed, let $h_1, h_2\in \kk_1^M$ with $\Phi_M(h_1)=\Phi_M(h_2)$.
Then, either $\Phi_M(h_1)$ is isolated in $\kk^M_2$ or not.
In the first case, there exists an $i_0\in M$ with
$\Phi_M(h_1)=g_{i_0}=\Phi_M(h_2)$. Thus, $h_1=f_{i_0}=h_2$.
So, assume that $\Phi_M(h_1)$ is not isolated in $\kk^m_2$. Hence,
neither $\Phi_M(h_2)$ is. It follows that both $h_1$ and
$h_2$ are not isolated points of $\kk_1^M$. Pick $L_1,
L_2\subseteq M$ infinite with $h_1=\lim_{i\in L_1} f_i$ and
$h_2=\lim_{i\in L_2} f_i$. As the sequence $(g_i)_{i\in L_1\cup
L_2}$ is pointwise convergent to $\Phi_M(h_1)=\Phi_M(h_2)$, our
assumptions yield that
\[ h_1=\lim_{i\in L_1} f_i =\lim_{i\in L_1\cup L_2} f_i = \lim_{i\in L_2} f_i=h_2 \]
which proves that $\Phi_M$ is 1-1. Finally, to see that $\Phi_M$
is onto, let $w\in \kk_2^M$ with $w\notin\{ g_i: i\in M\}$. Let
$L\subseteq M$ infinite with $w=\lim_{i\in L} g_i$. By our
assumptions, the sequence $(f_i)_{i\in L}$
converges pointwise to an $h\in \kk_1^M$ and clearly $\Phi_M(h)=w$.\\
(2) By similar arguments as in (1). \hfill $\lozenge$
\bigskip

\noindent By the above claim, it is enough to show that the map
$\Phi$ is continuous. Notice that it is enough to show that if
$(h_n)_n$ is a sequence in $\kk_1$ that converges pointwise to an
$h\in\kk_1$, then the sequence $\big(\Phi(h_n)\big)_n$ converges
to $\Phi(h)$. Assume on the contrary. Hence, there exist a
sequence $(h_n)_n$ in $\kk_1$, $h\in\kk_1$ and $w\in \kk_2$ such
that $h=\lim_n h_n$, $w=\lim_n \Phi(h_n)$ and $w\neq \Phi(h)$. As
the map $\Phi$ is onto, there exists $z\in \kk_1$ such that $z\neq
h$ and $\Phi(z)=w$. Pick $x\in X$ and $\ee>0$ such that
$|h(x)-z(x)|>\ee$. As the sequence $(h_n)_n$ converges pointwise
to $h$ we may assume that for all $n\in\nn$ we have
$|h_n(x)-z(x)|>\ee$. Let
\[ M=\big\{i\in I: |f_i(x)-z(x)|\geq\frac{\ee}{2}\big\}. \]
Observe the following.
\begin{enumerate}
\item[(O1)] For all $n\in\nn$, $h_n\in \kk_1^M$.
\item[(O2)] $z\notin \kk_1^M$.
\end{enumerate}
By part (2) of the above claim and (O1), we get that
$\Phi(h_n)=\Phi_M(h_n)\in\kk_2^M$ for all $n\in\nn$ and so $w\in
\kk_2^M$. As $\Phi_M$ is onto, there exists $h'\in\kk_1^M$ such
that $\Phi_M(h')=w$. Hence by (O2) and invoking the claim once
more, we have that $z\neq h'$ while $\Phi_M(h')=\Phi(h')=\Phi(z)$,
contradicting that $\Phi$ is 1-1. The proof of the lemma is
completed.
\end{proof}


\subsection{Seven families of functions}

The aim of this subsection is to describe seven families
\[ \{d^i_t:t\in \ct\} \ \ (1\leq i\leq 7)\]
of functions indexed by the Cantor tree. For every
$i\in\{1,...,7\}$, the closure of the family $\{d^i_t:t\in\ct\}$
in the pointwise topology is a separable Rosenthal compact
$\kk_i$. Each one of them is \textit{minimal}, namely, for every
dyadic (not necessarily regular) subtree $S=(s_t)_{t\in\ct}$ of
$\ct$ and every $i\in \{1,...,7\}$ the families
$\{d^i_t\}_{t\in\ct}$ and $\{d^i_{s_t}\}_{t\in\ct}$ are equivalent
in the sense of Definition \ref{equiv}. Although the families are
mutually non-equivalent, the corresponding compacta might be
homeomorphic. In all cases, the family $\{d^i_t:t\in \ct\}$ will
be discrete in its closure. For any of the corresponding compacta
$\kk_i$ $(1\leq i\leq 7)$, by $\mathcal{L}(\kk_i)$ we shall denote
the set of all infinite subsets $L$ of $\ct$ for which the
sequence $(d^i_t)_{t\in L}$ is pointwise convergent. We will name
the corresponding compacta (all of them are homeomorphic to closed
subspaces of well-known compacta -- see \cite{AU}, \cite{E}) and
we will refer to the families of functions as the canonical dense
sequences of them. We will use the following notations.

If $\sg\in 2^\nn$, then $\delta_\sg$ is the Dirac function at
$\sg$. By $x_\sg^+$ we denote the characteristic function of the
set $\{ \tau\in 2^\nn: \sg\preceq \tau\}$, while by $x_\sg^-$
the characteristic function of the set $\{ \tau\in 2^\nn:
\sg\prec \tau\}$. Notice that if $t\in\ct$, then $t^{\con} 0^\infty\in
2^\nn$, and so, the function $x^+_{t^{\con} 0^\infty}$ is
well-defined. It is useful at this point to isolate the
following property of the functions $x^+_\sg$ and $x^-_\sg$ which
will justify the notation $g^+_\sg$ and $g^-_\sg$ in Theorem
\ref{rt1}. If $(\sg_n)_n$ is a sequence in $2^\nn$ converging to
$\sg$ with $\sg_n\prec \sg$ (respectively $\sg\prec \sg_n$) for all
$n\in\nn$, then sequence $(x^{\ee_n}_{\sg_n})_n$ converges
pointwise to $x^+_\sg$ (respectively to $x^-_\sg$) for any choice
of $\ee_n\in \{+,-\}$.

By identifying the Cantor set with a subset of the unit interval,
we will identify every $\sg\in 2^\nn$ with the real-valued function
on $2^\nn$ which is equal everywhere with $\sg$. Notice that for
every $t\in\ct$, we have $t^{\con}0^\infty\in 2^\nn$, and so, the
function $t^{\con}0^\infty$ is well-defined. For every $t\in\ct$,
$v_t$ stands for the characteristic function of the clopen set
$V_t=\{ \sg\in 2^\nn: t\sqsubset \sg\}$. By $0$ we denote the
constant function on $2^\nn$ which is equal everywhere with zero.
We will also need to deal with functions on $2^\nn\oplus 2^\nn$.
In this case when we write, for instance, $(\delta_\sg,x^+_\sg)$
we mean that this function is the function $\delta_\sg$ on the
first copy of $2^\nn$ while it is the function $x^+_\sg$ on the
second copy.

We also fix a regular dyadic subtree $R=(s_t)_{t\in\ct}$ of $\ct$
with the following property.
\begin{enumerate}
\item[(Q)] For every $s,s'\in R$, we have that $s^{\con}
0^\infty\neq s'^{\con} 0^\infty$ and $s^{\con} 1^\infty\neq
s'^{\con} 1^\infty$. Hence, the set $[\hat{R}]$ does not contain
the eventually constant sequences.
\end{enumerate}
In what follows by $P$ we shall denote the perfect set
$[\hat{R}]$. By $P^+$ we shall denote the subset of $P$ consisting
of all $\sg$'s for which there exists an increasing antichain
$(s_n)_n$ of $R$ converging to $\sg$ in the sense of Definition
\ref{rd1}. Respectively, by $P^-$ we shall denote the subset of
$P$ consisting of all $\sg$'s for which there exists a decreasing
antichain $(s_n)_n$ of $R$ converging to $\sg$.

\subsubsection{The Alexandroff compactification of the
Cantor tree $A(\ct)$} It is the pointwise closure of the family
\[ \Big\{ \frac{1}{|t|+1} v_t: t\in\ct\Big\}. \]
Clearly the space $A(\ct)$ is countable compact, as the whole
family accumulates to $0$. Setting $d^1_t=\frac{1}{|t|+1}v_t$ for
all $t\in\ct$, we see that the family $\{d^1_t:t\in\ct\}$ is a
dense discrete subset of $A(\ct)$. In this case the description of
$\mathcal{L}\big(A(\ct)\big)$ is trivial as
\[ L\in \mathcal{L}\big(A(\ct)\big)\Leftrightarrow L\subseteq \ct. \]

\subsubsection{The space $2^{\leqslant \nn}$}
It is the pointwise closure of the family
\[ \{ s^{\con} 0^\infty: s\in R\}. \]
The accumulation points of $2^{\leqslant\nn}$ is the set
\[ \{\sg: \sg\in P\} \]
which is clearly homeomorphic to $2^{\nn}$. Thus, the space
$2^{\leqslant\nn}$ is uncountable compact metrizable. Setting
$d^2_t=s_t^{\con} 0^\infty$ for all $t\in\ct$ and invoking
property (Q) above, we see that the family $\{d^2_t:t\in\ct\}$ is
a dense discrete subset of $2^{\leqslant\nn}$. The description
of $\mathcal{L}\big( 2^{\leqslant\nn}\big)$ is given by
\[ L\in\mathcal{L}\big( 2^{\leqslant\nn}\big) \Leftrightarrow
\exists \sg\in 2^\nn \text{ with } L\to \sg. \]

\subsubsection{The extended split Cantor set $\splp$} It is the
pointwise closure of the family
\[ \{ x^+_{s^{\con} 0^\infty}:s\in R\}. \]
Notice that $\splp$ can be realized as a closed subspace of the
split interval $S(I)$. Thus, it is hereditarily separable. For
every $\sg\in P$, the function $x^+_\sg$ belongs to $\splp$.
However, for an element $\sg\in P$, the function $x^-_\sg$ belongs
to $\splp$ if and only if there exists a decreasing antichain
$(s_n)_n$ of $R$ converging to $\sg$. Finally observe that the
family $\{ x^+_{s^{\con} 0^\infty}:s\in R\}$ is a discrete subset
of $\splp$ (this is essentially a consequence of property (Q)
above). Hence, the accumulation points of $\splp$ is the set
\[ \{ x^+_\sg: \sg\in P\} \cup \{ x^-_\sg: \sg\in P^- \}. \]
Setting $d^3_t=x^+_{s_t^{\con} 0^\infty}$ for all $t\in \ct$, we
see that the family $\{d^3_t:t\in \ct\}$ is a dense discrete
subset of $\splp$. Moreover, we have the following description of
$\mathcal{L}\big(\splp\big)$
\[ L\in\mathcal{L}\big(\splp\big) \Leftrightarrow \exists \sg\in 2^\nn
\text{ with } L\to\sg \text{ and } (\text{either } L\preceq^* \sg
\text{ or } \sg\prec^* L). \]

\subsubsection{The mirror image $\splm$ of the extended split Cantor set}
The space $\splp$ has a natural mirror image $\splm$ which is the
pointwise closure of the set
\[ \{ x^-_{s^{\con} 1^\infty}:s\in R\}. \]
The spaces $\splp$ and $\splm$ are homeomorphic. To see this, for
every $t\in\ct$ let $\bar{t}\in\ct$ be the finite sequence
obtained by reversing $0$ with $1$ and $1$ with $0$ in the finite
sequence $t$. Define $\phi:R\to R$ by $\phi(s_t)=s_{\bar{t}}$ for
all $t\in \ct$. Then it is easy to see that the map
\[ \splp\ni x^+_{s_t^{\con}0^\infty} \mapsto
x^-_{\phi(s_t)^{\con}1^{\infty}}\in \splm\]
is extended to a topological homeomorphism between $\splp$ and
$\splm$. However, the canonical dense sequences in them are
\textit{not} equivalent. Notice that for every $\sg\in P$ the
function $x^-_\sg$ belongs to $\splm$, while the function
$x^+_\sg$ belongs to $\splm$ if and only if there exists an
increasing antichain $(s_n)_n$ of $R$ converging to $\sg$. It
follows that the accumulation points of $\splm$ is the set
\[ \{ x^-_\sg: \sg\in P\} \cup \{ x^+_\sg: \sg\in P^+\}. \]
As before, setting $d^4_t=x^-_{s_t^{\con} 1^\infty}$ for all
$t\in \ct$, we see that the family $\{d^4_t: t\in\ct\}$ is a
dense discrete subset of $\mathcal{L}\big(\splm\big)$
and moreover
\[ L\in \mathcal{L}\big(\splm\big) \Leftrightarrow \exists
\sg\in 2^\nn\text{ with } L\to\sg \text{ and } (\text{either }
L\prec^* \sg \text{ or } \sg\preceq^* L). \]

\subsubsection{The extended Alexandroff compactification of the Cantor set $\alex$}
The space $\alex$ is the pointwise closure of the family
\[ \{ v_t:t\in\ct\}. \]
For every $\sg\in 2^\nn$ the function $\delta_\sg$
belongs in $\alex$, the family $\{\delta_\sg:\sg\in 2^\nn\}$ is
discrete and accumulates to $0$. The function $0$ is the only
non-$G_\delta$ point of $\alex$ and this is witnessed in the most
extreme way. The accumulation points of $\alex$ is the set
\[ \{ \delta_\sg: \sg\in 2^\nn\} \cup \{0\} \]
Setting $d^5_t=v_t$ for all $t\in \ct$, the family $\{d^5_t:t\in
\ct\}$ is a dense discrete subset of $\alex$ and
\[ L\in\mathcal{L}\big(\alex\big) \Leftrightarrow (\exists \sg\in 2^\nn
\text{ with } L\subseteq^* \sg) \text{ or } (\forall \sg\in 2^\nn \ L\perp \sg). \]

\subsubsection{The extended duplicate of the Cantor set $\dcantor$} The space
$\dcantor$ is the pointwise closure of the family
\[ \{ (v_t, t^{\con} 0^\infty): t\in \ct\}. \]
This is the separable extension of the space $D(2^\nn)$,
as it was described in \cite{To1}. The accumulation points of
$\dcantor$ is the set
\[ \{ (\delta_\sg,\sg): \sg\in 2^\nn\} \cup \{ (0,\sg):\sg\in 2^\nn\},\]
which is homeomorphic to the Alexandroff duplicate of the Cantor
set. Todor\v{c}evi\'{c} was the first to realize that this classical
construction can be represented as a compact subset of the first
Baire class. The space $\dcantor$ is not only first countable
but it is also pre-metric of degree at most two, in the sense of
\cite{To1}. As in the previous cases, setting
$d^6_t=(v_t,t^{\con} 0^\infty)$ for every $t\in\ct$, we
see that the family $\{d^6_t:t\in \ct\}$ is a dense discrete
subset of $\dcantor$ and
\[ L\in\mathcal{L}\big(\dcantor\big) \Leftrightarrow \exists \sg\in 2^\nn \text{ with }
L\to\sg \text{ and } (\text{either } L\subseteq^* \sg \text{ or } L\perp \sg). \]

\subsubsection{The extended duplicate of the split Cantor set $\dsplit$} It is
the pointwise closure of the family
\[ \{ (v_s, x^+_{s^{\con} 0^\infty}): s\in R\}. \]
The space $\dsplit$ is homeomorphic to a subspace of the Helly
space $\mathcal{H}$. To see this, let $\{(a_t, b_t):t\in\ct\}$
be a family in $[0,1]^2$ such that
\begin{enumerate}
\item[(i)] $a_t=a_{t^{\con}0}<b_{t^{\con}0}<a_{t^{\con}1}<b_{t^{\con}1}=b_t$, and
\item[(ii)] $b_t-a_t \leq \frac{1}{3^{|t|}}$
\end{enumerate}
for every $t\in\ct$. Define $h_t:[0,1]\to [0,1]$ by
\[ h_t(x)= \left\{ \begin{array} {r@{\quad:\quad}l}
    1 & b_t< x,\\
   \frac{1}{2} & a_t\leq x\leq b_t, \\
   0 & x< a_t. \end{array} \right. \]
It is easy to see that the map
\[ \dsplit \ni (v_{s_t}, x^+_{s_t^{\con}0^\infty})\mapsto
h_t\in \mathcal{H}\]
is extended to a homeomorphic embedding. It is follows that
the space $\dsplit$ is first countable. We notice, however,
that it is not pre-metric of degree at most two.

As in all previous cases, we will describe the accumulation
points of $\dsplit$. First we observe that if $(s_n)_n$ is a chain
of $R$ converging to $\sg\in P$, then the sequence $\big((v_{s_n},
x^+_{{s_n}^{\con} 0^\infty})\big)_n$ is pointwise convergent to
$(\delta_\sg,x^+_\sg)$. If $(s_n)_n$ is an increasing antichain of
$R$ converging to $\sg$, then the sequence $\big((v_{s_n},
x^+_{{s_n}^{\con} 0^\infty})\big)_n$ is pointwise convergent to
$(0,x^+_\sg)$, while if it is decreasing, then it is pointwise
convergent to $(0,x^-_\sg)$. Thus, the accumulation points of
$\dsplit$ is the set
\[ \{(\delta_\sg,x^+_\sg):\sg\in P\} \cup \{ (0,x^+_\sg): \sg\in P^+\}
\cup \{ (0,x^-_\sg):\sg\in P^-\}. \] Finally, setting
$d^7_t=(v_{s_t}, x^+_{{s_t}^{\con} 0^\infty})$ for all $t\in \ct$,
we see that the family $\{d^7_t:t\in \ct\}$ is a dense discrete
subset of $\dsplit$. The description of $\mathcal{L}\big(\dsplit\big)$
is given by
\[ L\in \mathcal{L}\big(\dsplit\big)\Leftrightarrow \exists \sg\in 2^\nn
\text{ with } L\to \sg \text{ and } (L\prec^* \sg \text{ or } L\subseteq^* \sg
\text{ or } \sg\prec^* L). \]
We close this subsection by noticing the following minimality
property of the above described families.
\begin{prop}
\label{stability} Let $\{d^i_t:t\in\ct\}$ with $i\in\{1,...,7\}$ be one
of the seven families of functions and let $S=(s_t)_{t\in\ct}$
be a dyadic (not necessarily regular) subtree of $\ct$.
Then the family $\{d^i_t:t\in\ct\}$ and the corresponding
family $\{d^i_{s_t}:t\in\ct\}$ determined by the tree $S$ are equivalent.
\end{prop}
We also observe that any two of the seven families are
\textit{not} equivalent. Moreover, beside the case of $\splp$ and
$\splm$, the corresponding compacta are not mutually homeomorphic
either.


\subsection{Canonicalization}

The main result of this section is the following.
\begin{thm}
\label{basis} Let $\{f_t\}_{t\in\ct}$ be a family of real-valued
functions on a Polish space $X$ which is relatively compact in
$\mathcal{B}_1(X)$. Let also $\{d^i_t\}_{t\in\ct} \ (1\leq i\leq
7)$ be the families described in the previous subsection. Then
there exist a regular dyadic subtree $S=(s_t)_{t\in\ct}$ of $\ct$
and $i_0\in\{1,...,7\}$ such that $\{f_{s_t}\}_{t\in\ct}$ is
equivalent to $\{d^{i_0}_t\}_{t\in\ct}$.
\end{thm}
\begin{proof}
The family $\{f_t\}_{t\in\ct}$ satisfies all hypotheses of Theorem
\ref{rt1}. Thus, there exist a regular dyadic subtree $T$ of $\ct$
and a family of functions $\{ g^0_\sg, g^+_\sg, g^-_\sg:\sg\in
P\}$, with $P=[\hat{T}]$, as described in Theorem \ref{rt1}. Let
also $0, +$ and $-$ be the corresponding Borel functions.
We recall that for every subset $X$ of $2^\nn$ we identify the
set $[X]^2$ of doubletons of $X$ with the set of all $(\sg,\tau)\in X^2$
with $\sg\prec \tau$. For every $\ee\in\{0,+,-\}$ let
\[ A_{\ee,\ee}=\{ (\sg_1,\sg_2)\in [P]^2: g^\ee_{\sg_1}\neq g^\ee_{\sg_2} \}.\]
Then $A_{\ee,\ee}$ is an analytic subset of $[P]^2$. To see this,
notice that
\begin{eqnarray*}
(\sg_1,\sg_2)\in A_{\ee,\ee} &\Leftrightarrow & \exists x\in X \text{ with }
g^\ee_{\sg_1}(x)\neq g^\ee_{\sg_2}(x) \\
& \Leftrightarrow & \exists x\in X \text{ with } \ee(\sg_1,x)\neq \ee(\sg_2,x).
\end{eqnarray*}
Invoking the Borelness of the functions $0,+,-$ we see that
$A_{\ee,\ee}$ is analytic, as desired. Notice that for every
$Q\subseteq P$ perfect and every $\ee\in\{0,+,-\}$, the set
$A_{\ee,\ee}\cap [Q]^2$ is analytic in $[Q]^2$. Thus, applying
Theorem \ref{galvin2} successively three times, we get a perfect
subset $Q_0$ of $P$ such that for all $\ee\in \{0,+,-\}$ we have
that
\[ \text{ either } [Q_0]^2\subseteq A_{\ee,\ee} \text{ or }
A_{\ee,\ee} \cap [Q_0]^2= \varnothing. \]

\noindent \textbf{Case 1.} $A_{0,0}\cap [Q_0]^2=\varnothing$. Is
this case, we have that $g^0_{\sg_1}=g^0_{\sg_2}$ for all
$(\sg_1,\sg_2)\in [Q_0]^2$. Thus, there exists a function $g$ such
that $g^0_\sg=g$ for all $\sg\in Q_0$. By properties (2) and (3)
in Theorem \ref{rt1} and the homogeneity of $Q_0$, we see that
$g^+_\sg=g^-_\sg=g^0_\sg=g$ for all $\sg\in Q_0$. Pick a regular
dyadic subtree $S=(s_t)_{t\in\ct}$ of $T$ such that $[\hat{S}]\subseteq Q_0$ and
$f_s\neq g$ for all $s\in S$. Invoking properties (1), (4) and (5)
of Theorem \ref{rt1} as well as Lemma \ref{al1}(2), we see that
for every infinite subset $A$ of $S$, the sequence $(f_t)_{t\in
A}$ accumulates to $g$. It follows that
$\overline{\{f_s\}}^p_{s\in S}=\{f_s\}_{s\in S}\cup \{g\}$, and
so, $\{f_{s_t}\}_{t\in\ct}$ is equivalent to the canonical dense
family of $A(\ct)$.
\bigskip

\noindent \textbf{Case 2.} $[Q_0]^2\subseteq A_{0,0}$. Then for
every $(\sg_1,\sg_2)\in [Q_0]^2$ we have that $g^0_{\sg_1}\neq
g^0_{\sg_2}$. By passing to a further perfect subset of $Q_0$ if
necessary, we may also assume that
\begin{enumerate}
\item[(P1)] $g^0_\sg\neq f_t$ for every $\sg\in Q_0$ and every $t\in T$.
\end{enumerate}
\bigskip

\noindent \textbf{Case 2.1.} Either $A_{+,+}\cap
[Q_0]^2=\varnothing$, or $A_{-,-}\cap [Q_0]^2=\varnothing$. Assume
first that $A_{+,+}\cap [Q_0]^2=\varnothing$. In this case we have
that there exists a function $g$ such that $g^+_\sg=g$ for all
$\sg\in Q_0$. By property (3) in Theorem \ref{rt1} and the
homogeneity of $Q_0$, we must also have that $g^-_\sg=g$ for all
$\sg\in Q_0$. This means that $A_{-,-}\cap [Q_0]^2=\varnothing$.
Thus, by symmetry, this case is equivalent to say that
$A_{+,+}\cap [Q_0]^2=\varnothing$ and $A_{-,-}\cap
[Q_0]^2=\varnothing$. It follows that there exists a function $g$
such that $g^+_\sg=g^-_\sg=g$ for all $\sg\in Q_0$. By passing to
a further perfect subset of $Q_0$ if necessary, we may also assume
that $g^0_\sg\neq g$ for all $\sg\in Q_0$. We select a regular
dyadic subtree $S=(s_t)_{t\in\ct}$ of $T$ such that $[\hat{S}]\subseteq Q_0$ and
$f_s\neq g$ for all $s\in S$. This property combined with (P1)
implies that for every $s\in S$ the function $f_s$ is isolated in
$\overline{\{f_s\}}^p_{s\in S}$.

We claim that $\{f_{s_t}\}_{t\in\ct}$ is equivalent to the
canonical dense family of $\alex$. We will give a detailed
argument which will serve as a prototype for the other
cases as well. First, we notice that, by Lemma \ref{l1} and the
description of $\mathcal{L}\big(\alex\big)$, it is enough to show
that for a subset $A$ of $S$, the sequence $(f_s)_{s\in A}$
converges pointwise if and only if either $A$ is almost included
in a chain, or $A$ does not contain an infinite chain. For the if part, we
observe that if $A$ is almost contained in a chain, then by
property (1) of Theorem \ref{rt1}, the sequence $(f_s)_{s\in A}$
is pointwise convergent. Assume that $A$ does not contain an
infinite chain. Since $g^+_\sg=g^-_\sg=g$ for all $\sg\in Q_0$,
we see that for every increasing and every decreasing antichain
$(s_n)_n$ of $S$, the sequence $(f_{s_n})_n$ converges pointwise to $g$. Thus,
$(f_s)_{s\in A}$ is pointwise convergent to $g$. For the only if
part we argue by contradiction. If there exist $\sg_1\neq \sg_2$
contained in $[\hat{S}]$ with $A\cap \{\sg_1|n:n\in\nn\}$ and
$A\cap \{\sg_2|n:n\in\nn\}$ infinite, then the fact that
$g^0_{\sg_1}\neq g^0_{\sg_2}$ implies that the sequence
$(f_s)_{s\in A}$ is not pointwise convergent. Finally, if $A$
contains an infinite chain and an infinite antichain, then the
fact that $g^0_\sg\neq g$ for all $\sg\in [\hat{S}]$ implies that
$(f_s)_{s\in A}$ is not pointwise convergent too.
\bigskip

\noindent \textbf{Case 2.2.} $[Q_0]^2 \subseteq A_{+,+}$ and
$[Q_0]^2\subseteq A_{-,-}$. In this case we have that
\begin{enumerate}
\item[(P2)] $g^\ee_{\sg_1}\neq g^\ee_{\sg_2}$ for all
$(\sg_1, \sg_2)\in [Q_0]^2$ and $\ee\in\{0,+,-\}$.
\end{enumerate}
Moreover, by passing to a further perfect subset of $Q_0$ if
necessary, we may strengthen (P1) to
\begin{enumerate}
\item[(P3)] $g^\ee_\sg\neq f_t$ for all $\sg\in Q_0$, $\ee\in\{ 0,+,-\}$
and $t\in T$.
\end{enumerate}
Observe that (P3) implies the following. For every regular dyadic
subtree $S$ of $T$ with $[\hat{S}]\subseteq Q_0$ and every $s\in
S$, the function $f_s$ is isolated in the closure of
$\{f_s\}_{s\in S}$ in $\rr^X$. Thus, as in Case 2.1, in what
follows Lemma \ref{l1} will be applicable.

For every $\ee_1,\ee_2\in\{0,+,-\}$ with $\ee_1\neq\ee_2$ let
\[ A_{\ee_1,\ee_2}=\{ (\sg_1,\sg_2)\in [Q_0]^2:
g^{\ee_1}_{\sg_1}\neq g^{\ee_2}_{\sg_2} \}.\]
Then $A_{\ee_1,\ee_2}$ is an analytic subset of $[Q_0]^2$. Applying
Theorem \ref{galvin2} successively six times, we find
$Q_1\subseteq Q_0$ perfect such that for all $\ee_1,\ee_2\in
\{0,+,-\}$ with $\ee_1\neq \ee_2$ we have that
\[ \text{either } [Q_1]^2\subseteq A_{\ee_1,\ee_2} \text{ or }
A_{\ee_1,\ee_2} \cap [Q_1]^2=\varnothing. \] We claim that for
each pair $\ee_1, \ee_2$ the first alternative must occur. Assume
on the contrary that there exist $\ee_1,\ee_2$ with $\ee_1\neq
\ee_2$ and such that $A_{\ee_1,\ee_2}\cap [Q_1]^2=\varnothing$.
Let $\tau$ be the lexicographical minimum of $Q_1$. Then for every
$\sg,\sg'\in Q_1$ with $\tau\prec \sg\prec \sg'$ we
have $g^{\ee_2}_{\sg}=g^{\ee_1}_{\tau}=g^{\ee_2}_{\sg'}$ which
contradicts (P2). Summing up, by passing to $Q_1$, we have
strengthen (P2) to
\begin{enumerate}
\item[(P4)] $g^{\ee_1}_{\sg_1}\neq g^{\ee_2}_{\sg_2}$ for all
$(\sg_1,\sg_2)\in [Q_1]^2$ and $\ee_1,\ee_2\in \{0,+,-\}$.
\end{enumerate}
For every $\ee\in\{+,-\}$, define $B_{0,\ee}\subseteq Q_1$ by
\[ B_{0,\ee}=\{ \sg\in Q_1: g^0_\sg\neq g^\ee_{\sg}\}. \]
It is easy to see that $B_{0,\ee}$ is an analytic subset of $Q_1$.
Thus, by the classical perfect set theorem, we find $Q_2\subseteq
Q_1$ perfect such that for every $\ee\in\{+,-\}$ we have
\[ \text{either } Q_2\subseteq B_{0,\ee} \text{ or } B_{0,\ee}\cap Q_2=\varnothing.\]

\noindent \textbf{Case 2.2.a.} $B_{0,+}\cap Q_2=\varnothing$ and
$B_{0,-}\cap Q_2= \varnothing$. In this case, for every $\sg\in
Q_2$ there exists a function $g_\sg$ such that
$g_\sg=g^0_\sg=g^+_\sg=g^-_\sg$. Moreover, $g_{\sg_1}\neq
g_{\sg_2}$ for all $\sg_1\neq \sg_2$ in $Q_2$, as $Q_2\subseteq
Q_1$. Invoking properties (2) and (3) in Theorem \ref{rt1}, we see
that the set $\{g_\sg:\sg\in Q_2\}$ is homeomorphic to $Q_2$. We
select a regular dyadic subtree $S=(s_t)_{t\in\ct}$ of $T$ such that
$[\hat{S}]\subseteq Q_2\subseteq Q_0$. It follows that
$\overline{\{f_s\}}^p_{s\in S}=\{f_s\}_{s\in S}\cup \{g_\sg:\sg\in
[\hat{S}]\}$, and so, the family $\{f_{s_t}\}_{t\in\ct}$ is
equivalent to the canonical dense family of $2^{\leqslant\nn}$.
\bigskip

\noindent \textbf{Case 2.2.b.} $B_{0,+}\cap Q_2=\varnothing$ and
$Q_2\subseteq B_{0,-}$. This means that $g^0_\sg=g^+_\sg$ and
$g^0_\sg\neq g^-_\sg$ for all $\sg\in Q_2$. Let $S=(s_t)_{t\in\ct}$ be a regular
dyadic subtree of $T$ such that $[\hat{S}]\subseteq Q_2\subseteq
Q_0$. Invoking (P3) and the remarks following it, the description
of $\mathcal{L}\big(\splp\big)$ and Lemma \ref{l1}, arguing
precisely as in Case 2.1, we see that $\{f_{s_t}\}_{t\in\ct}$ is
equivalent to the canonical dense family of $\splp$.
\bigskip

\noindent \textbf{Case 2.2.c.} $Q_2\subseteq B_{0,+}$ and
$B_{0,-}\cap Q_2=\varnothing$. This means that $g^0_\sg=g^-_\sg$
and $g^0_\sg\neq g^+_\sg$ for all $\sg\in Q_2$. As in the previous
case, let $S=(s_t)_{t\in\ct}$ be a regular dyadic subtree of $T$
such that $[\hat{S}]\subseteq Q_2\subseteq Q_0$. In this case
$\{f_{s_t}\}_{t\in\ct}$ is equivalent to canonical dense family of
the mirror image $\splm$ of the extended split Cantor set (the
argument is as in Case 2.1).
\bigskip

\noindent \textbf{Case 2.2.d.} $Q_2\subseteq B_{0,+}$ and
$Q_2\subseteq B_{0,-}$. In this case we have
\begin{enumerate}
\item[(P5)] $g^0_\sg\neq g^+_\sg$ and $g^0_\sg\neq g^-_\sg$
for all $\sg\in Q_2$.
\end{enumerate}
Let
\[ B_{+,-}=\{ \sg\in Q_2: g^+_\sg\neq g^-_\sg\} \]
Again, $B_{+,-}$ is an analytic subset of $Q_2$. Thus there exists
$Q_3\subseteq Q_2$ perfect such that either $Q_3\subseteq B_{+,-}$
or $Q_3\cap B_{+,-}=\varnothing$.
\bigskip

\noindent \textbf{Case 2.2.d.I.} $Q_3\cap B_{+,-}=\varnothing$.
This means that for every $\sg\in Q_3$ there exists a function
$g_\sg$ such that $g_\sg=g^+_\sg=g^-_\sg$ and $g_\sg\neq g^0_\sg$.
Moreover, by property (P4) above, we have that $g_{\sg_1}\neq g_{\sg_2}$ and
$g^0_{\sg_1}\neq g^0_{\sg_2}$ for all $(\sg_1,\sg_2)\in [Q_3]^2$,
as $Q_3\subseteq Q_2\subseteq Q_1$. Let $S=(s_t)_{t\in\ct}$ be a regular dyadic
subtree of $T$ such that $[\hat{S}]\subseteq Q_3\subseteq Q_0$. In
this case $\{f_{s_t}\}_{t\in\ct}$ is equivalent to the canonical
dense family of $\dcantor$. The verification is similar to the
previous cases.
\bigskip

\noindent \textbf{Case 2.2.d.II.} $Q_3\subseteq B_{+,-}$. This
means that $g^+_\sg\neq g^-_\sg$ for all $\sg\in Q_3$. Combining
this with (P4) and (P5), we see that $g^{\ee_1}_{\sg_1}\neq
g^{\ee_2}_{\sg_2}$ if either $\ee_1\neq\ee_2$ or $\sg_1\neq\sg_2$.
As before, let $S=(s_t)_{t\in\ct}$ be a regular dyadic subtree of
$T$ such that $[\hat{S}]\subseteq Q_3\subseteq Q_0$. Then
$\{f_{s_t}\}_{t\in\ct}$ is equivalent to the canonical dense
family of $\dsplit$.
\bigskip

\noindent All the above cases are exhaustive and the proof is
completed.
\end{proof}
By Theorem \ref{basis} and Proposition \ref{stability} we get the
following corollary.
\begin{cor}
\label{corbasis} Let $X$ be a Polish space and $\{f_t\}_{t\in\ct}$
be family of functions relatively compact in $\mathcal{B}_1(X)$.
Then for every regular dyadic subtree $T$ of $\ct$ there exist a
regular dyadic subtree $S$ of $T$ and $i_0\in\{1,...,7\}$ such
that for every regular dyadic subtree $R=(r_t)_{t\in\ct}$ of
$S$, the family $\{f_{r_t}\}_{t\in\ct}$ is equivalent to
$\{d^{i_0}_t\}_{t\in\ct}$.
\end{cor}


\section{Analytic subspaces of separable Rosenthal compacta}

In this section we introduce a class of subspaces of separable
Rosenthal compacta and we present some of their basic properties.

\subsection{Definitions and basic properties}

Let $\kk$ be a separable Rosenthal compact on a Polish space $X$.
For every subset $\fff$ of $\kk$ by $Acc(\fff)$ we denote the set
of accumulation points of $\fff$ in $\rr^X$. We start with the
following definition.
\begin{defn}
\label{nad1} Let $\kk$ be a separable Rosenthal compact on
a Polish space $X$ and $\ccc$ a closed subspace of $\kk$.
We say that $\ccc$ is an analytic subspace of $\kk$ if there
exist a countable dense subset $\{f_n\}_n$ of $\kk$ and
an analytic subset $A$ of $[\nn]$ such that the following
are satisfied.
\begin{enumerate}
\item[(1)] For every $L\in A$ we have that $Acc\big(\{f_n: n\in L\}\big)
\subseteq \ccc$.
\item[(2)] For every $g\in \ccc\cap Acc(\kk)$ there exists $L\in A$
with $g\in \overline{\{f_n\}}^p_{n\in L}$.
\end{enumerate}
\end{defn}
Let us make some remarks concerning the above notion.
First we notice that the analytic set $A$ witnessing the
analyticity of $\ccc$ can always be assumed to be hereditary.
We also observe that an analytic subspace of $\kk$ is not
necessarily separable. For instance, if $\kk=\alex$
and $\ccc=A(2^\nn)$, then it is easy to see that $\ccc$
is an analytic subspace of $\kkk$. The following
proposition gives some examples of analytic subspaces.
\begin{prop}
\label{nap1} Let $\kk$ be a separable Rosenthal compact.
Then the following hold.
\begin{enumerate}
\item[(1)] $\kk$ is analytic with respect to any countable
dense subset $\{f_n\}_n$ of $\kk$.
\item[(2)] Every closed $G_\delta$ subspace $\ccc$ of $\kk$ is analytic.
\item[(3)] Every closed separable subspace $\ccc$ of $\kk$ is analytic.
\end{enumerate}
\end{prop}
\begin{proof}
(1) Take $A=[\nn]$. \\
(2) Let $(U_k)_k$ be a sequence of open subsets of $\kk$ with
$\overline{U}_{k+1}\subseteq U_k$ for all $k\in\nn$ and such that
$\ccc=\bigcap_k U_k$. Let also $\{f_n\}_n$ be a countable dense subset of $\kkk$.
For every $k\in\nn$, let $M_k=\{n\in\nn: f_n\in U_k\}$. Notice that the
sequence $(M_k)_k$ is decreasing. Let $A\subseteq [\nn]$ be defined by
\[ L\in A \Leftrightarrow \forall k\in\nn \ (L\subseteq^* M_k). \]
Clearly $A$ is Borel. It is easy to see that $A$ satisfies condition
(1) of Definition \ref{nad1} for $\ccc$. To see that condition (2)
is also satisfied, let $g\in \ccc\cap Acc(\kk)$. By the Bourgain-Fremlin-Talagrand
theorem \cite{BFT}, there exists an infinite subset $L$ on $\nn$
such that $g$ is the pointwise limit of the sequence $(f_n)_{n\in L}$.
As $g\in U_k$ for all $k\in\nn$, we see that $L\subseteq^* M_k$ for all
$k$. Hence the set $A$ witness the analyticity of $\ccc$.\\
(3) Let $D_1$ be a countable dense subset of $\kk$ and $D_2$
a countable dense subset of $\ccc$. Let $\{f_n\}_n$ be an enumeration
of the set $D_1\cup D_2$ and set $L=\{n\in\nn: f_n\in D_2\}$.
Let also $M=\{k\in L: f_k\in Acc(\kk)\}$ and for every $k\in M$
select $L_k\in [\nn]$ such that $f_k$ is the pointwise limit
of the sequence $(f_n)_{n\in L_k}$. Define $A=[L] \cup \big(\bigcup_{k\in M} [L_k]\big)$.
The countable dense subset $\{f_n\}_n$ of $\kk$ and the set $A$ verify
the analyticity of $\ccc$.
\end{proof}
To proceed with our discussion on the properties of analytic subspaces
we need some pieces of notation. Let $\kk$ be a separable Rosenthal
compact and $\mathbf{f}=\{f_n\}_n$ a countable dense subset of $\kk$. We set
\[ \lbf=\{ L\in [\nn]: (f_n)_{n\in L} \text{ is pointwise convergent}\}. \]
Moreover, for every accumulation point $f$ of $\kk$ we let
\[ \llf=\{ L\in [\nn]: (f_n)_{n\in L} \text{ is pointwise convergent
to } f\}. \]
We notice that both $\lbf$ and $\llf$ are co-analytic.
The first result relating the topological behavior of a
point $f$ in $\kk$ with the descriptive set-theoretic properties
of the set $\llf$ is the result of A. Krawczyk from \cite{Kra}
asserting that a point $f\in\kk$ is $G_\delta$ if and only if
the set $\llf$ is Borel. Another important structural property
is the following consequence of the effective
version of the Bourgain-Fremlin-Talagrand theorem,
proved by G. Debs in \cite{De}.
\begin{thm}
\label{Debs} Let $\kk$ be a separable Rosenthal compact.
Then for every countable dense subset $\mathbf{f}=\{f_n\}_n$
of $\kk$, there exists a Borel, hereditary and cofinal subset
$C$ of $\lbf$.
\end{thm}
We refer the reader to \cite{D} for an explanation of how
Debs' theorem yields the above result.

Let $\kk$ and $\mathbf{f}=\{f_n\}_n$ as above. For every $A\subseteq
\lbf$ we set
\[ \kk_{A,\mathbf{f}}=\{g\in\kk: \exists L\in A \text{ with }
g=\lim_{n\in L} f_n\}.\]
We have the following characterization of analytic subspaces
which is essentially a consequence of Theorem \ref{Debs}.
\begin{prop}
\label{nap2} Let $\kk$ be a separable Rosenthal compact and
$\ccc$ a closed subspace of $\kk$. Then $\ccc$ is analytic if and only
if there exist a countable dense subset $\bseq$ of $\kk$ and
a hereditary and analytic subset $A'$ of $\lbf$ such that
$\kk_{A',\mathbf{f}}=\ccc\cap Acc(\kk)$.
\end{prop}
\begin{proof}
The direction $(\Leftarrow)$ is immediate. Conversely, assume
that $\ccc$ is analytic and let $\bseq$ and $A\subseteq [\nn]$
verifying its analyticity. As we have already remarked, we may assume
that $A$ is hereditary. By Theorem \ref{Debs}, there exists a Borel,
hereditary and cofinal subset $C$ of $\lbf$. We set $A'=A\cap C$.
We claim that $A'$ is the desired set. Clearly $A'$ is a hereditary
and analytic subset of $\lbf$. Also observe that, by condition (1) of
Definition \ref{nad1}, for every $L\in A'$ the sequence $(f_n)_{n\in L}$
must be pointwise convergent to a function $g\in\ccc$. Hence
$\kk_{A',\mathbf{f}}\subseteq \ccc\cap Acc(\kk)$. Conversely
let $g\in \ccc\cap Acc(\kk)$. There exists $M\in A$ with
$g\in\overline{\{f_n\}}^p_{n\in M}$. By the Bourgain-Fremlin-Talagrand
theorem, there exists $N\in [M]$ such that $g$ is the pointwise
limit of the sequence $(f_n)_{n\in N}$. Clearly $N\in\lbf$. As $C$ is
cofinal in $\lbf$, there exists $L\in [N]$ with $L\in C$. As $A$
is hereditary, we see that $L\in A\cap C=A'$. The proof is completed.
\end{proof}

\subsection{Separable Rosenthal compacta in $\mathcal{B}_1(2^\nn)$}

Let $\kk$ be separable Rosenthal compact on a Polish space
$X$ and $\mathbf{f}=\{f_n\}_n$ a countable dense subset of $\kk$.
By Theorem \ref{Debs}, there exists a Borel cofinal subset
of $\lbf$. The following proposition shows that if $X$ is
compact metrizable, then the global property of $\lbf$ (namely
that it contains a Borel cofinal set) is also valid locally.
We notice that in the argument below we make use of the Arsenin-Kunugui
theorem in a spirit similar as in \cite{Pol2}.
\begin{prop}
\label{newpcof} Let $X$ be a compact metrizable space, $\kk$ a
separable Rosenthal compact on $X$ and $\mathbf{f}=\{f_n\}_n$
a countable dense subset of $\kkk$. Then for every $f\in\kkk$
there exists an analytic hereditary subset $B$ of $\llf$ which
is cofinal in $\llf$.
\end{prop}
\begin{proof}
We apply Theorem \ref{Debs} and we get a hereditary, Borel and
cofinal subset $C$ of $\lbf$. Consider the function $\Phi:C\times X\to \rr$
defined by $\Phi(L,x)=f_L(x)$, where by $f_L$ we denote the pointwise
limit of the sequence $(f_n)_{n\in L}$. Then $\Phi$ is Borel.
To see this, for every $n\in\nn$ let $\Phi_n:C\times X \to \rr$
be defined by $\Phi_n(L,x)=f_{l_n}(x)$, where $l_n$ is the $n^{th}$
element of the increasing enumeration of $L$. Clearly $\Phi_n$ is Borel.
As $\Phi(L,x)=\lim_{n} \Phi_n(L,x)$ for all $(L,x)\in C\times X$,
the Borelness of $\Phi$ is shown. For every $m\in\nn$ define
$P_m\subseteq C\times X$ by
\begin{eqnarray*}
(L,x)\in P_m & \Leftrightarrow & |f_L(x)-f(x)|>\frac{1}{m+1} \\
& \Leftrightarrow & (c,x)\in \Phi^{-1}\big( (-\infty,-\frac{1}{m+1})\cup
(\frac{1}{m+1},+\infty)\big).
\end{eqnarray*}
Clearly $P_m$ is Borel. For every $L\in C$ the function $x\mapsto
|f_L(x)-f(x)|$ is Baire-1. Hence, for every $L\in C$ the section
$(P_m)_L=\{x\in X:(c,x)\in P_m\}$ of $P_m$ at $L$ is $F_\sg$, and
as $X$ is compact metrizable, it is $K_\sg$. By the
Arsenin-Kunugui theorem (see \cite{Kechris}, Theorem 35.46), the
set
\[ G_m=\mathrm{proj}_C P_m \]
is Borel. It follows that the set $G=\bigcup_m G_m$ is a Borel
subset of $C$. Put $D=C\setminus G$. Now observe that for every
$L\in C$ we have that $L\in\llf$ if and only if $L\notin G$.
Hence, the set $D$ is a Borel subset of $\llf$, and as $C$ is
cofinal, we get that $D$ is cofinal in $\llf$. Hence,
setting $B$ to be the hereditary closure of $D$, we see
that $B$ is as desired.
\end{proof}
\begin{rem}
\label{r1} (1) We notice that Proposition \ref{newpcof} is not valid
for an arbitrary separable Rosenthal compact. A counterexample,
taken from \cite{Pol2} (see also \cite{Ma}), is the following.
Let $A$ be an analytic non-Borel subset of $2^\nn$
and denote by $\kk_A$ the separable Rosenthal compact
obtained by restrict every function of $\alex$ on $A$.
Clearly the function $0|_A$ belongs to $\kk_A$ and is
a non-$G_\delta$ point of $\kk_A$. It is easy to check that,
in this case, there does not exist a Borel cofinal subset
of $\mathcal{L}_{0|_A}$.\\
(2) We should point out that the hereditary and cofinal subset
$B$ of $\llf$, obtained by Proposition \ref{newpcof}, can be chosen
to be Borel. To see this, start with an analytic and cofinal
subset $A_0$ of $\llf$. Using Souslin's separation theorem
we construct two sequences $(B_n)_n$ and $(C_n)_n$ such that
$B_n$ is Borel, $C_n$ is the hereditary closure of $B_n$ and
$A_0\subseteq B_n\subseteq C_n\subseteq B_{n+1}\subseteq \llf$
for all $n\in\nn$. Setting $B=\bigcup_n B_n$, we see that
$B$ is as desired.
\end{rem}
The arguments in the proof of Proposition \ref{newpcof} can
be used to derive certain properties of analytic subspaces
of separable Rosenthal compacta. To state them we need one more
piece of notation. For a separable Rosenthal compact $\kk$
on a Polish space $X$, $\bseq$ a countable dense
subset of $\kk$ and $\ccc$ a closed subspace of $\kk$ we
set
\[ \lbfc=\{ L\in[\nn]: \exists g\in\ccc \text{ with }
g=\lim_{n\in L} f_n\}. \]
Clearly $\lbfc$ is a subset of $\lbf$. Also notice that if $\ccc=\{f\}$
for some $f\in\kk$, then $\lbfc=\llf$.

Part (1) of the following proposition extends Proposition \ref{newpcof}
for analytic subspaces. The second part shows that the notion
of an analytic subspace of $\kk$ is independent
of the choice of the dense sequence, for every separable
Rosenthal compact $\kk$ in $\bbb_1(2^\nn)$.
\begin{prop}
\label{analp1} Let $X$ be a compact metrizable space,
$\kk$ be a separable Rosenthal compact on $X$ and $\ccc$
and analytic subspace of $\kk$. Let $\bseq$ be a countable dense
subset of $\kk$ and $A\subseteq [\nn]$ witnessing the analyticity
of $\ccc$. Then the following hold.
\begin{enumerate}
\item[(1)] There exists an analytic cofinal subset
$A_1$ of $\lbfc$.
\item[(2)] For every countable dense subset $\mathbf{g}=\{g_n\}_n$
of $\kk$ there exists an analytic subset $A_2$ of $\mathcal{L}_{\mathbf{g}}$
such that $\kk_{A_2,\mathbf{g}}=\ccc\cap Acc(\kk)$.
\end{enumerate}
\end{prop}
\begin{proof}
(1) By Proposition \ref{nap2}, there exists a hereditary and
analytic subset $A'$ of $\lbf$ such that $\kk_{A',\mathbf{f}}=
\ccc\cap Acc(\kk)$. Applying Theorem \ref{Debs}, we get a Borel,
hereditary and cofinal subset $C$ of $\lbf$. As in Proposition
\ref{newpcof}, for every $L\in C$ by $f_L$ we denote the pointwise
limit of the sequence $(f_n)_{n\in L}$. Let $A''=A'\cap C$. Clearly
$A''$ is analytic and hereditary. Moreover, it is easy to see that
$\kk_{A'',\mathbf{f}}=\ccc\cap Acc(\kk)$ (i.e. the set $A''$
codes all function in $Acc(\kk)\cap\ccc$). Consider the following
equivalence relation $\sim$ on $C$, defined by
\[ L\sim M \Leftrightarrow f_L=f_M \Leftrightarrow \forall x\in X \
f_L(x)=f_M(x). \]
We claim that $\sim$ is Borel. To see this notice that the map
\[ C\times C\times X\ni (L,M,x)\mapsto |f_L(x)-f_M(x)|\]
is Borel (this can be easily checked arguing as in Proposition
\ref{newpcof}). Moreover, for every $(L,M)\in C\times C$, the
map $x\mapsto |f_L(x)-f_M(x)|$ is Baire-1. Observe that
\[ \neg (L\sim M) \Leftrightarrow \exists x\in X \ \exists \ee>0
\text{ with } |f_L(x)-f_M(x)|>\ee.\]
By the fact that $X$ is compact metrizable and by the Arsenin-Kunugui theorem
we see that $\sim$ is Borel. We set $A_1$ to be the $\sim$
saturation of $A''$, i.e.
\[ A_1=\{ M\in C: \exists L\in A'' \text{ with } M\sim L\}. \]
As $A''$ is analytic and $\sim$ is Borel, we get that $A_1$ is analytic.
As $C$ is cofinal, it is easy to check that $A_1$ is cofinal
in $\lbfc$. Thus, the set $A_1$ is the desired one.\\
(2) Let $C_1$ and $C_2$ be two hereditary, Borel subsets of
$\lbf$ and $\mathcal{L}_{\mathbf{g}}$ cofinal in $\lbf$
and $\mathcal{L}_{\mathbf{g}}$ respectively. By part (1),
there exists a hereditary and analytic subset $A_1$
of $\lbf$ which is cofinal in $\lbfc$. We set $A'_1=A_1\cap C_1$.
Consider the following subset $S$ of $C_1\times C_2$ defined by
\[ (L,M)\in S \Leftrightarrow f_L=g_M \Leftrightarrow
\forall x\in X \ f_L(x)=g_M(x) \]
where $f_L$ denotes the pointwise limit of the sequence $(f_n)_{n\in L}$
while $g_M$ denotes the pointwise limit of the sequence $(g_n)_{n\in M}$.
As $X$ is compact metrizable, arguing as in part (1), it is easy
to see that $S$ is Borel. We set
\[ A_2=\{ M\in C_2: \exists L\in A'_1 \text{ with } (L,M)\in S\}. \]
The set $A_2$ is the desired one.
\end{proof}
We close this subsection with the following proposition
which provides further examples of analytic subspaces.
\begin{prop}
\label{analp2} Let $\kk$ be a separable Rosenthal compact on a
Polish space $X$. Let also $F$ be a $K_\sg$ subset of $X$. Then
the subspace $\mathcal{C}_F=\{ f\in\kk: f|_F=0\}$ of $\kk$ is
analytic with respect to any countable dense subset
$\mathbf{f}=\{f_n\}_n$ of $\kk$.
\end{prop}
\begin{proof}
Let $C$ be a hereditary, Borel and cofinal subset of $\lbf$.
Let $Z$ be the subset of $C\times X$ defined by
\[ (L,x)\in Z \Leftrightarrow (x\in F) \text{ and }
(\exists \ee>0 \text{ with } |f_L(x)|>\ee).\]
The set $Z$ is Borel. As $F$ is $K_\sg$, we see that for every
$L\in C$ the section $Z_L=\{x\in X:(L,x)\in Z\}$ of $Z$ at $L$ is $K_\sg$.
Thus, setting $A=C\setminus \mathrm{proj}_C Z$ and invoking the
Arsenin-Kunugui theorem, we see that the set $A$ witnesses the
analyticity of $\ccc_F$ with respect to $\{f_n\}_n$.
\end{proof}
Related to the above propositions and the concept of an analytic
subspace of $\kk$, the following questions are open to us.
\bigskip

\noindent \textit{Problem 1.} Is it true that the concept
of an analytic subspace is independent of the choice of the
countable dense subset of $\kk$? More precisely, if $\ccc$
is an analytic subspace of a separable Rosenthal compact
$\kk$ on a Polish space $X$ and $\mathbf{f}=\{f_n\}_n$ is
an arbitrary countable dense subset of $\kk$, does there
exists $A\subseteq \lbf$ analytic with
$\kk_{A,\mathbf{f}}=\ccc\cap Acc(\kk)$?
\bigskip

\noindent \textit{Problem 2.} Let $\kk$ be a separable Rosenthal
compact on a Polish space $X$ and let $B\subseteq X$ Borel. Is
the subspace $\ccc_B=\{f\in\kk: f|_B=0\}$ analytic?


\section{Canonical embeddings in analytic subspaces}

This section is devoted to the canonical embedding of the most
representative prototype, among the seven minimal families, into a
given analytic subspace of a separable Rosenthal compact. The section
is divided into two subsections. The first subsection concerns
metrizable Rosenthal compacta and the second the non-metrizable ones.
We start with the following definitions.
\begin{defn}
\label{caninj} An injection $\phi:\ct\to\nn$ is said to be
canonical provided that $\phi(s)<\phi(t)$ if either $|s|<|t|$, or
$|s|=|t|$ and $s\prec t$. By $\phi_0$ we denote the unique
canonical bijection between $\ct$ and $\nn$.
\end{defn}
\begin{defn}
\label{canemb} Let $\kk$ be a separable Rosenthal compact,
$\{f_n\}_n$ a countable dense subset of $\kk$ and $\ccc$ a closed
subspace of $\kk$. Let also $\{d^i_t\}_{t\in\ct} \ (1\leq i\leq
7)$ be the canonical families described in \S 4.3 and let
$\kk_i \ (1\leq i\leq 7)$ be the corresponding separable Rosenthal
compacta. For every $i\in \{1,...,7\}$, we say that $\kk_i$
canonically embeds into $\kk$ with respect to $\{f_n\}_n$ and
$\ccc$ if there exists a canonical injection $\phi:\ct\to\nn$
such that the families $\{d^i_t\}_{t\in\ct}$ and
$\{f_{\phi(t)}\}_{t\in\ct}$ are equivalent, that is if the map
\[ \kk_i\ni d^i_t \mapsto f_{\phi(t)} \in\kk \]
is extended to a homeomorphism between $\kk_i$ and
$\overline{\{f_{\phi(t)}\}}^p_{t\in\ct}$, and moreover
\[ Acc\big(\{f_{\phi(t)}:t\in\ct\}\big)\subseteq \ccc.\]
If $\ccc=\kk$, then we simply say that $\kk_i$ canonical embeds
into $\kk$ with respect to $\{f_n\}_n$.
\end{defn}


\subsection{Metrizable Rosenthal compacta}

This subsection is devoted to the proof of the following theorem.
\begin{thm}
\label{metrnew} Let $\kk$ be a separable Rosenthal compact on a
Polish space $X$ consisting of bounded functions. Let also
$\{f_n\}_n$ be a countable dense subset of $\kk$. Assume that
$\kk$ is metrizable in the pointwise topology and non-separable
in the supremum norm of $\mathcal{B}_1(X)$. Then there exists
a canonical embedding of $2^{\leqslant\nn}$ into $\kk$ with
respect to $\{f_n\}_n$ whose accumulation points are
$\ee$-separated in the supremum norm for some $\ee>0$. In
particular, its image is non-separable in the supremum norm.
\end{thm}
\begin{proof}
Fix a compatible metric $\rho$ for the pointwise topology of
$\kk$. Our assumptions on $\kk$ yield that there exist $\ee>0$
and a family $\Gamma=\{f_\xi:\xi<\omega_1\}\subseteq \kk$ such
that $\Gamma$ is $\ee$-separated in the supremum norm and each
$f_\xi$ is a condensation point of the family $\Gamma$ in the
pointwise topology.

By recursion on the length of finite sequences in $\ct$ we shall
construct the following.
\begin{enumerate}
\item[(C1)] A family $(B_t)_{t\in\ct}$ of open subsets of $\kk$,
\item[(C2)] a family $(x_t)_{t\in\ct}$ in $X$, \item[(C3)] two
families $(r_t)_{t\in\ct}$, $(q_t)_{t\in\ct}$ of reals and
\item[(C4)] a canonical injection $\phi:\ct\to\nn$
\end{enumerate}
such that for every $t\in\ct$ the following are satisfied.
\begin{enumerate}
\item[(P1)] $\overline{B}_{t^{\con}0}\cap
\overline{B}_{t^{\con}1}= \varnothing$,
$\overline{B}_{t^{\con}0}\cup \overline{B}_{t^{\con}1} \subseteq
B_t$ and $\rho-\mathrm{diam}(B_t)\leq \frac{1}{|t|+1}$.
\item[(P2)] $|B_t\cap \Gamma|=\aleph_1$. \item[(P3)] $r_t<q_t$ and
$q_t-r_t>\ee$. \item[(P4)] For every $f\in
\overline{B}_{t^{\con}0}$, $f(x_t)<r_t$, while for every $f\in
\overline{B}_{t^{\con}1}$, $f(x_t)>q_t$. \item[(P5)]
$f_{\phi(t)}\in B_t$.
\end{enumerate}
We set $B_{(\varnothing)}=\kk$ and $\phi\big((\varnothing)\big)
=0$. We choose $f,g\in \Gamma$ and we pick $x\in X$ and
$r,q\in\rr$ such that $f(x)<r<q<g(x)$ and $q-r>\ee$. We set
$x_{(\varnothing)}=x$, $r_{(\varnothing)}=r$ and
$q_{(\varnothing)}=q$. We select $B_{(0)}, B_{(1)}$ open subsets
of $\kk$ such that $f\in B_{(0)}\subseteq \{h\in \rr^X:
h(x_{(\varnothing)})<r_{(\varnothing)}\}$, $g\in B_{(1)} \subseteq
\{h\in \rr^X: h(x_{(\varnothing)})> q_{(\varnothing)}\}$,
$\rho-\mathrm{diam}(B_{(0)})<\frac{1}{2}$ and
$\rho-\mathrm{diam}(B_{(1)}) <\frac{1}{2}$. Let us observe that
$x_{(\varnothing)}, r_{(\varnothing)}, q_{(\varnothing)}, B_{(0)}$
and $B_{(1)}$ satisfy properties (P1)-(P4) above. Notice also that
$B_{(0)}, B_{(1)}$ are uncountable, hence, they intersect the
dense set $\{f_n\}_n$ at an infinite set. So, we may select
$\phi\big((\varnothing)\big)<\phi\big((0)\big)< \phi\big((1)\big)$
satisfying (P5). The general inductive step proceeds in a similar
manner assuming that
\begin{enumerate}
\item[(a)] for each $t\in \ct$ with $|t|< n-1$, $x_t$, $r_t$ and
$q_t$ have been chosen, and \item[(b)] for each $t\in\ct$ with
$|t|<n$, $B_t$ and $\phi(t)$ have been chosen
\end{enumerate}
such that (P1)-(P5) are satisfied. This completes the recursive
construction.

Notice that for every $\sg\in 2^\nn$ we have $\bigcap_n
B_{\sg|n}=\{f_\sg\}$. The map $2^\nn\ni \sg\mapsto f_\sg\in
\kk$ is a homeomorphic embedding. Moreover, for every $\sg\in
2^\nn$, the sequence $(f_{\phi(\sg|n)})_n$ is pointwise convergent
to $f_\sg$. We also observe the following consequence of
properties (P3) and (P4). If $\sg<\tau\in 2^\nn$, then, setting
$t=\sg\wedge \tau$, we have that $f_\sg(x_t)\leq r_t<q_t\leq
f_\tau(x_t)$ and so $\|f_\sg-f_\tau\|_{\infty} >\ee$. As there are
at most countable many $\sg\in 2^\nn$ with $f_\sg\in \{f_n\}_n$,
by passing to a regular dyadic subtree of $\ct$ if necessary, we
may assume that for every $t\in \ct$ the function $f_{\phi(t)}$ is
isolated in $\overline{\{f_{\phi(t)}\}}^p_{t\in\ct}$. This easily
yields that the family $\{f_{\phi(t)}\}_{t\in\ct}$ is equivalent
to the canonical dense family of $2^{\leqslant\nn}$. The proof is
completed.
\end{proof}


\subsection{Non-metrizable separable Rosenthal compacta}

The main results of this subsection are the following.
\begin{thm}
\label{embt1} Let $\kk$ be a separable Rosenthal compact on a
Polish space $X$ and let $\ccc$ be an analytic subspace of $\kkk$.
Let also $\{f_n\}_n$ be a countable dense subset of $\kk$ and
$A\subseteq [\nn]$ analytic, witnessing the analyticity of $\ccc$.
Assume that $\ccc$ is not hereditarily separable. Then either
$\alex$, or $\dcantor$, or $\dsplit$ canonically embeds into
$\kk$ with respect to $\{f_n\}_n$ and $\ccc$.

In particular, if $\kk$ is first countable and not hereditarily
separable, then either $\dcantor$, or $\dsplit$ canonically embeds
into $\kk$ with respect to every countable dense subset $\{f_n\}_n$
of $\kk$.
\end{thm}
As it is shown in Corollary \ref{gdt1}, if $\kk$ is not first
countable, then $\alex$ canonically embeds into $\kk$.
\begin{thm}
\label{embt2} Let $\kk$ be a separable Rosenthal compact on a
Polish space $X$ and $\{f_n\}_n$ a countable dense subset of $\kk$.
Assume that $\kk$ is hereditarily separable and non-metrizable.
Then either $\splp$, or $\splm$ canonically embeds into $\kk$ with
respect to $\{f_n\}_n$.
\end{thm}


\subsubsection{Proof of Theorem \ref{embt1}}

The main goal is to prove the following.
\begin{prop}
\label{embp1} Let $\kk$, $\ccc$ and $\{f_n\}_n$ be as in
Theorem \ref{embt1}. Then there exists a canonical injection
$\psi:\ct\to\nn$ such that, setting
\[ \kk_\sg=\overline{ \{f_{\psi(\sg|n)}\}}^p_{n} \setminus
\{f_{\psi(\sg|n)}\}_n \]
for all $\sg\in 2^\nn$, there exists an open subset
$V_\sg\subseteq \rr^X$ with $\kk_\sg\subseteq V_\sg\cap \ccc$ and
such that $\kk_\tau\cap V_\sg=\varnothing$ for every $\tau\in 2^\nn$
with $\tau\neq \sg$.
\end{prop}
Granting Proposition \ref{embp1}, we complete the proof as follows.
Let $\psi$ be the canonical injection obtained by the above
proposition and define $f_t=f_{\psi(t)}$ for all $t\in\ct$. We
apply Theorem \ref{basis} and we get a regular dyadic subtree
$S=(s_t)_{t\in\ct}$ of $\ct$ and $i_0\in\{1,...,7\}$ such that
$\{f_{s_t}\}_{t\in\ct}$ is equivalent to
$\{d^{i_0}_t\}_{t\in\ct}$. By Proposition \ref{embp1}, we see that
the closure of $\{f_{s_t}\}_{t\in\ct}$ in $\rr^X$ contains an
uncountable discrete set. Thus $\{f_{s_t}\}_{t\in\ct}$ is
equivalent to the canonical dense family of either $\alex$, or
$\dcantor$, or $\dsplit$. Setting $\phi=\psi\circ i_S$ we see that
$\phi$ is a canonical injection imposing an embedding of either
$\alex$, or $\dcantor$, or $\dsplit$ into $\kk$ with respect to
$\{f_n\}_n$ and $\ccc$.

We proceed to the proof of Proposition \ref{embp1}. By enlarging the
topology on $X$ if necessary (see \cite{Kechris}), we may assume
that the functions $\{f_n\}_n$ are continuous. We may also assume
that the set $A$ is hereditary. It follows by condition (2) of
Definition \ref{nad1} and the Bourgain-Fremlin-Talagrand theorem,
that for every $g\in \ccc\cap Acc(\kk)$ there exists $L\in A$ such
that $g$ is the pointwise limit of the sequence $(f_n)_{n\in L}$.
We fix a continuous map $\Phi:\nn^\nn\to [\nn]$ with
$\Phi(\nn^\nn)=A$.

We will need the following notation. For every $m\in\nn$,
$y=(x_1,...,x_m)\in X^m$,
$\lambda=(\lambda_1,...,\lambda_m)\in\rr^m$ and $\ee>0$ we set
\[ V(y,\lambda,\ee)=\{ g\in \rr^X: \lambda_i-\ee < g(x_i)<\lambda_i+\ee \
\forall i=1,...,m\} \] By $\overline{V}(y,\lambda,\ee)$ we denote
the closure of $V(y,\lambda,\ee)$ in $\rr^X$.

Using the fact that $\ccc$ is not hereditarily separable, by
recursion on countable ordinals we get
\begin{enumerate}
\item[(1)] $m\in\nn$, $\lambda=(\lambda_1,..., \lambda_m)
\in\mathbb{Q}^m$ and positive rationals $\ee$ and $\delta$,
\item[(2)] a family $\Gamma=\{y_\xi=(x_1^\xi,...,x^\xi_m):
\xi<\omega_1\}\subseteq X^m$, \item[(3)] a family $\{f_\xi:
\xi<\omega_1\}\subseteq \ccc$, \item[(4)] a family
$\{M_\xi:\xi<\omega_1\}\subseteq [\nn]$, and \item[(5)] a family
$\{b_\xi:\xi<\omega_1\}\subseteq \nn^\nn$
\end{enumerate}
such that for every $\xi<\omega_1$ the following are satisfied.
\begin{enumerate}
\item[(i)] $f_\xi\in Acc(\kk)$. \item[(ii)] $f_\xi\in
V(y_\xi,\lambda,\ee)$, while for every $\zeta<\xi$ we have
$f_\zeta\notin \overline{V}(y_\xi,\lambda,\ee+\delta)$.
\item[(iii)] $y_\xi$ is a condensation point of $\Gamma$ in $X^m$.
\item[(iv)] $\Phi(b_\xi)=M_\xi$ and $f_\xi$ is the pointwise limit
of the sequence $(f_n)_{n\in M_\xi}$.
\end{enumerate}

Now, by induction on the length of the finite sequences in $\ct$
we shall construct the following.
\begin{enumerate}
\item[(C1)] A canonical injection $\psi:\ct\to\nn$. \item[(C2)] A
family $(B_t)_{t\in\ct}$ of open balls in $X^m$, taken with
respect to a compatible complete metric $\rho$ of $X^m$.
\item[(C3)] A family $(\Delta_t)_{t\in\ct}$ of uncountable subsets
of $\omega_1$.
\end{enumerate}
The construction is done so that for every $t\in\ct$ the following
are satisfied.
\begin{enumerate}
\item[(P1)] If $t\neq (\varnothing)$, then $f_{\psi(t)}\in
V(y,\lambda,\ee)$ for all $y\in B_t$. \item[(P2)] For all $t',
t\in\ct$ with $|t'|=|t|$ and $t'\neq t$ we have $f_{\psi(t)}\notin
\overline{V}(y,\lambda,\ee+\delta)$ for every $y\in B_{t'}$.
\item[(P3)] $\overline{B}_{t^{\con}0}\cap \overline{B}_{t^{\con}1}
= \varnothing$, $\overline{B}_{t^{\con}0}\cup
\overline{B}_{t^{\con}1}\subseteq B_t$ and
$\rho-\mathrm{diam}(B_t)\leq \frac{1}{|t|+1}$. \item[(P4)]
$\Delta_{t^{\con}0}\cap \Delta_{t^{\con}1}=\varnothing$ and
$\Delta_{t^{\con}0}\cup \Delta_{t^{\con}1}\subseteq \Delta_t$.
\item[(P5)] $\mathrm{diam}\big(\{b_\xi:\xi\in \Delta_t\}\big)\leq
\frac{1}{2^{|t|}}$. \item[(P6)] $\{y_\xi:\xi\in\Delta_t\}\subseteq
B_t$. \item[(P7)] If $t\neq (\varnothing)$, then $\psi(t)\in M_\xi$
for every $\xi\in \Delta_t$.
\end{enumerate}
Assume that the construction has been carried out. We set
$y_\sg=\bigcap_n B_{\sg|n}$ and
$V_\sg=V(y_\sg,\lambda,\ee+\frac{\delta}{2})$ for all $\sg\in
2^\nn$. Using (P1) and (P2), it is easy to see that
$\kk_\sg\subseteq V_\sg$ and $\kk_\sg\cap V_\tau=\varnothing$ if
$\sg\neq\tau$. We only need to check that $\kk_\sg\subseteq\ccc$
for every $\sg\in 2^\nn$. So, let $\sg\in 2^\nn$ arbitrary. We set
$M=\{\psi(\sg|n):n\geq 1\}\in[\nn]$. It is enough to show that
$M\in A$. For every $k\geq 1$ we select $\xi_k\in \Delta_{\sg|k}$.
By properties (P4), (P5) and (P7), the sequence
$(b_{\xi_k})_{k\geq 1}$ converges to a unique $b\in \nn^\nn$ and,
moreover, $\psi(\sg|n)\in M_{\xi_k}= \Phi(b_{\xi_k})$ for every
$1\leq n\leq k$. By the continuity of $\Phi$ we get that
$M_{\xi_k}\to\Phi(b)$, and so, $M\subseteq \Phi(b)$. As $A$ is
hereditary, we see that $M\in A$, as desired.

We proceed to the construction. We set $\psi\big(
(\varnothing)\big)=0$, $B_{(\varnothing)}=X^m$ and
$\Delta_{(\varnothing)}=\omega_1$. Assume that for some $n\geq 1$
and for all $t\in 2^{<n}$ the values $\psi(t)\in\nn$, the open
balls $B_t$ and the sets $\Delta_t$ have been constructed.
Refining if necessary, we may assume that for every $t\in 2^{<n}$
and every $\xi\in\Delta_t$ the point $y_\xi$ is a condensation
point of the set $\{y_\zeta:\zeta\in\Delta_t\}$.

Let $\{t_0\prec ...\prec t_{2^{n-1}-1}\}$ be the
$\prec$-increasing enumeration of $2^{n-1}$. For every $j\in
\{0,...,2^n-1\}$ we choose an open ball $B^{-1}_j$ in $X^m$ and an
uncountable subset $\Delta^{-1}_j$ of $\omega_1$ such that
$\rho-\mathrm{diam}(B^{-1}_j)<\frac{1}{n+1}$, $\{y_\xi:\xi\in
\Delta^{-1}_j\}\subseteq B^{-1}_j$ and
$\mathrm{diam}\{b_\xi:\xi\in \Delta^{-1}_j\}\leq \frac{1}{2^n}$.
Moreover, the selection is done so that for $j$ even we have
$\overline{B^{-1}_j}\cap\overline{B^{-1}_{j+1}}=\varnothing$,
$\overline{B^{-1}_j}\cup\overline{B^{-1}_{j+1}}\subseteq B_{t_{j/2}}$
and $\Delta^{-1}_j\cup \Delta^{-1}_{j+1}\subseteq \Delta_{t_{j/2}}$.
We set $m_{-1}=\max\{ \psi(t): t\in 2^{<n}\}$.

By finite recursion on $k\in \{0,...,2^n-1\}$, we will construct a
family $\{B^k_j: j=0,...,2^n-1\}$ of open balls of $X^m$,
a family $\{\Delta^k_j:j=0,..., 2^n-1\}$ of uncountable subsets of
$\omega_1$ and a positive integer $m_k\in\nn$ such that for every
$k\in\{0,...,2^n-1\}$ the following are satisfied.
\begin{enumerate}
\item[(a)] For every $j\in\{0,...,2^n-1\}$, $B^{k-1}_j\supseteq
B^k_j$, $\Delta_j^{k-1}\supseteq \Delta^k_j$ and $\{y_\xi:\xi\in
\Delta^k_j\}\subseteq B^k_j$. Moreover, for every $j$ and every
$\xi\in \Delta^k_j$ the point $y_\xi$ is a condensation point of
$\{y_\zeta:\zeta\in \Delta^k_j\}$. \item[(b)] $m_{k-1}< m_k$.
\item[(c)] For every $y\in B^k_k$ we have $f_{m_k}\in V(y,\lambda,\ee)$,
while for every $j\in\{0,..., 2^n-1\}$ with $j\neq k$ and
every $y\in B^k_j$ we have $f_{m_k}\notin \overline{V}(y,\lambda,\ee+\delta)$.
\item[(d)] $m_k\in M_\xi$ for all $\xi\in \Delta_k^k$.
\end{enumerate}
As the first step of this construction is identical to the general
one, we may assume that the construction has been carried out for
all $k'<k$, where $k\in\{0,...,2^n-1\}$. Fix a countable base
$\mathcal{B}$ of open balls of $X^m$. We first observe that for
each $\xi\in \Delta^{k-1}_k$ and every $j\in \{0, ..., 2^n-1\}$
there exist
\begin{enumerate}
\item[(e)] a basic open ball $B^{k,\xi}_j\subseteq B^{k-1}_j$, and
\item[(f)] a positive integer $m_\xi\in M_\xi$ with
$m_{k-1}<m_\xi$
\end{enumerate}
such that the following hold.
\begin{enumerate}
\item[(g)] For every $y\in B^{k,\xi}_k$ we have $f_{m_\xi}\in
V(y,\lambda,\ee)$, while for every $j\neq k$ and every $y\in B^{k,\xi}_j$
we have $f_{m_\xi}\notin \overline{V}(y,\lambda, \ee+\delta)$.
\end{enumerate}
To see that such choices are possible, fix $\xi\in \Delta^{k-1}_k$.
We can choose $\xi_0,..., \xi_{2^n-1}$ distinct countable ordinals
such that
\begin{enumerate}
\item[(h)] for all $j\in\{0,...,2^n-1\}$ we have $\xi_j\in
\Delta^{k-1}_j$, and \item[(k)] $\xi=\xi_k=\min\{ \xi_j: j=0,...,
2^n-1\}$.
\end{enumerate}
By (ii) and (k) above, we have that $f_{\xi_k}\in
V(y_{\xi_k},\lambda, \ee)$ while $f_{\xi_k}\notin
\overline{V}(y_{\xi_j}, \lambda,\ee+\delta)$ for all $j\neq k$. By
(iv), we can choose $m_\xi\in M_\xi$ with $m_{k-1}<m_\xi$ (thus
condition (f) above is satisfied) and such that $f_{m_\xi}\in V(y_{\xi},
\lambda,\ee)$ while $f_{m_\xi}\notin \overline{V}(y_{\xi_j},
\lambda,\ee+\delta)$ for all $j\neq k$. As $f_{m_\xi}$ is
continuous, for every $j\in\{0,...,2^n-1\}$ we can find a basic
open ball $B^{k,\xi}_j$ in $X^m$ containing $y_{\xi_j}$ such that
conditions (e) and (g) above are satisfied.

By cardinality arguments, we see that there exist
$\Delta^k_k\subseteq \Delta^{k-1}_k$ uncountable,
$m_k\in\nn$ and for every $j$ a ball $B^k_j$ such that $m_\xi=m_k$
and $B^{k,\xi}_j=B^k_j$ for all $\xi\in \Delta^k_k$. Setting
$\Delta^k_j=\{\xi: y_\xi\in B^k_j\}\cap \Delta^{k-1}_j$, the
recursive construction, described by clauses (a)-(d) above,
is completed.

Now let $\{t_0\prec ...\prec t_{2^n-1}\}$ be the
$\prec$-increasing enumeration of $2^n$. We set $\psi(t_k)=m_k$,
$B_{t_k}=B^{2^n-1}_k$ and $\Delta_{t_k}=\Delta^{2^n-1}_k$ for all
$k\in\{0,...,2^n-1\}$. It is easy to check that (P1)-(P7) are
satisfied. This completes the construction described in (C1), (C2)
and (C3). Having completed the proof of Proposition \ref{embp1},
the proof of Theorem \ref{embt1} follows.


\subsubsection{Proof of Theorem \ref{embt2}}

As in the proof of Theorem \ref{embt1}, the main goal is to prove
the following.
\begin{prop}
\label{embp2} Let $\{f_n\}_n$ be a family of continuous functions
relatively compact in $\mathcal{B}_1(X)$. If the closure $\kk$ of
$\{f_n\}_n$ in $\rr^X$ is non-metrizable, then there exist canonical
injections $\psi_1:\ct\to\nn$ and $\psi_2:\ct\to\nn$ such that
following are satisfied. Setting
\[ \kk^1_\sg=\overline{ \{f_{\psi_1(\sg|n)}\}}^p_{n} \setminus \{f_{\psi_(\sg|n)}\}_n
\text{ and } \kk^2_\sg=\overline{ \{f_{\psi_2(\sg|n)}\}}^p_{n}
\setminus \{f_{\psi_2(\sg|n)}\}_n \] for all $\sg\in 2^\nn$, there
exists an open subset $V_\sg$ of $\rr^X$ with $(\kk^1_\sg-\kk^2_\sg)
\subseteq V_\sg$ and $(\kk^1_\tau-\kk^2_\tau)\cap
V_\sg=\varnothing$ for every $\tau\in 2^\nn$ with $\tau\neq \sg$.
\end{prop}
Granting Proposition \ref{embp2}, we complete the proof
of Theorem \ref{embt2} as follows. Let $\{f_n\}_n$ be the
countable dense subset of $\kk$. As we have already remarked,
we may assume that the functions $\{f_n\}_n$ are actually
continuous. We apply Proposition \ref{embp2} to the family
$\{f_n\}_n$ and we get the canonical injections $\psi_1:\ct\to\nn$
and $\psi_2:\ct\to\nn$ as described above. We define
$g_t=f_{\psi_1(t)}$ and $h_t=f_{\psi_2(t)}$ for every $t\in\ct$.
Applying Corollary \ref{corbasis} successively two times,
we get a regular dyadic subtree $S=(s_t)_{t\in\ct}$ of
$\ct$ such that the families $\{g_{s_t}\}_{t\in\ct}$ and
$\{h_{s_t}\}_{t\in\ct}$ are canonicalized. The fact that $\kk$ is
hereditarily separable implies that each of the above families must
be equivalent to the canonical dense family of either
$A(\ct)$, or $2^{\leqslant\nn}$, or $\splp$, or $\splm$. By
Proposition \ref{embp2}, we see that it cannot be the case
that both $\overline{\{g_{s_t}\}}^p_{t\in\ct}$ and
$\overline{\{h_{s_t}\}}^p_{t\in\ct}$ are metrizable. Thus, at
least one of them is equivalent to either $\splp$, or $\splm$,
which clearly implies Theorem \ref{embt2}.

We proceed to the proof of Proposition \ref{embp2} which is similar
to the proof of Proposition \ref{embp1} and relies on the fact
that $\kk$ is metrizable if and only if there exists $D\subseteq
X$ countable such that the map $\mathrm{Acc}(\kk)\ni f\mapsto
f|_D\in \rr^D$ is 1-1. Thus, by our assumptions and by transfinite
recursion on countable ordinals, we get
\begin{enumerate}
\item[(1)] $p<q\in\mathbb{Q}$, \item[(2)] a set
$\Gamma=\{x_\xi:\xi<\omega_1\}\subseteq X$, and \item[(3)] two
families $\{g_\xi:\xi<\omega_1\}$ and $\{h_\xi:\xi<\omega_1\}$ in
$\mathrm{Acc}(\kk)$
\end{enumerate}
such that for every $\xi<\omega_1$ the following are satisfied.
\begin{enumerate}
\item[(i)] $g_\xi(x_\zeta)=h_\xi(x_\zeta)$ for all $\zeta<\xi$.
\item[(ii)] $g_\xi(x_\xi)<p<q< h_\xi(x_\xi)$. \item[(iii)] $x_\xi$
is a condensation point of $\Gamma$.
\end{enumerate}

By recursion on the length of the finite sequences in $\ct$ we
shall construct the following.
\begin{enumerate}
\item[(C1)] Two canonical injections $\psi_1:\ct\to\nn$ and
$\psi_2:\ct\to\nn$. \item[(C2)] A family $(B_t)_{t\in\ct}$ of open
balls in $X$, taken with respect to a compatible complete metric
$\rho$ of $X$.
\end{enumerate}
The construction is done so that for every $t\in\ct$ the following
are satisfied.
\begin{enumerate}
\item[(P1)] If $t\neq (\varnothing)$, then $f_{\psi_1(t)}(x)<
p<q<f_{\psi_2(t)}(x)$ for every $x\in B_t$.
\item[(P2)] For every $t, t'\in\ct$ with $|t'|=|t|$ and
$t'\neq t$ and for every $x'\in B_{t'}$ we have
$|f_{\psi_1(t)}(x')- f_{\psi_2(t)}(x')|<\frac{1}{|t|+1}$.
\item[(P3)] $\overline{B}_{t^{\con}0}\cap
\overline{B}_{t^{\con}1}= \varnothing$,
$\overline{B}_{t^{\con}0}\cup \overline{B}_{t^{\con}1}\subseteq
B_t$ and $\rho-\mathrm{diam}(B_t)\leq \frac{1}{|t|+1}$.
\item[(P4)] $|B_t\cap \Gamma|=\aleph_1$.
\end{enumerate}
Assuming that the construction has been carried out, setting
$x_\sg=\bigcap_n B_{\sg|n}$ and
\[ V_\sg=\{ w\in \rr^X: |w(x_\sg)|>(q-p)/2\} \]
for all $\sg\in 2^\nn$, it is easy to check that $\psi_1$,
$\psi_2$ and $\{V_\sg:\sg\in 2^\nn\}$ satisfy all requirements of
Proposition \ref{embp2}.

We proceed to the construction. We set $\psi_1\big(
(\varnothing)\big)= \psi_2\big( (\varnothing)\big)=0$ and
$B_{(\varnothing)}=X$. Assume that for some $n\geq 1$ and for all
$t\in 2^{<n}$ the values $\psi_1(t), \psi_2(t)\in\nn$ and the open
balls $B_t$ have been constructed. As in Proposition \ref{embp1},
in order to determine $\psi_1(t), \psi_2(t)$ and $B_t$ for
every $t\in 2^n$ we shall follow a finite recursion.

Let $\{t_0\prec ...\prec t_{2^{n-1}-1}\}$ be the $\prec$-increasing
enumeration of $2^{n-1}$. For every $j\in \{0,...,2^n-1\}$
we choose an open ball $B^{-1}_j$ in $X$ such that
$\rho-\mathrm{diam}(B^{-1}_j)<\frac{1}{n+1}$ and $|B^{-1}_j\cap
\Gamma|=\aleph_1$. Moreover, the selection is done so
that for $j$ even we have $\overline{B^{-1}_j}\cap
\overline{B^{-1}_{j+1}}=\varnothing$ and $\overline{B^{-1}_j}\cup
\overline{B^{-1}_{j+1}}\subseteq B_{t_{j/2}}$. We set
$m_{-1}=\max\{\psi_1(t):t\in 2^{<n}\}$ and $l_{-1}=\max\{
\psi_2(t): t\in 2^{<n}\}$.

By finite recursion on $k\in \{0,...,2^n-1\}$, we will construct a
family $\{B^k_j: j=0,...,2^n-1\}$ of open balls of $X$ and a
pair $m_k, l_k$ in $\nn$ such that for every
$k\in\{0,...,2^n-1\}$ the following are satisfied.
\begin{enumerate}
\item[(a)] For every $j\in\{0,...,2^n-1\}$ we have $B^{k-1}_j\supseteq
B^k_j$. \item[(b)] $m_{k-1}< m_k$ and $l_{k-1}<l_k$. \item[(c)]
For every $x\in B^k_k$ we have $f_{m_k}(x)<p<q<f_{l_k}(x)$, while
for every $j\in\{0,..., 2^n-1\}$ with $j\neq k$ and every
$x'\in B^k_j$ we have $|f_{m_k}(x')-f_{l_k}(x')|<\frac{1}{n+1}$.
\item[(d)] For every $j\in\{0,...,2^n-1\}$ we have $|B^k_j\cap \Gamma|=\aleph_1$.
\end{enumerate}
We omit the above construction as it is similar to the one in
Proposition \ref{embp1}. We only notice that condition (k) is
replaced by
\begin{enumerate}
\item[(k')] $\xi_k=\max\{ \xi_j: j=0,...,2^n-1\}$.
\end{enumerate}
Now let $\{t_0\prec ...\prec t_{2^n-1}\}$ be the
$\prec$-increasing enumeration of $2^n$. We set $\psi_1(t_k)=m_k$,
$\psi_2(t_k)=l_k$ and $B_{t_k}=B^{2^n-1}_k$ for every
$k\in\{0,...,2^n-1\}$. It is easy to check that (P1)-(P4) are
satisfied. The proof of Proposition \ref{embp2} is completed.
\begin{rem}
We notice that Theorem \ref{metrnew} (respectively Theorem
\ref{embt2}) is valid for an analytic and metrizable (respectively
hereditarily separable) subspace $\ccc$ of $\kk$. In particular,
we have the following.
\begin{thm}
\label{reallynew} Let $\kk$ be a separable Rosenthal compact
and $\ccc$ an analytic subspace of $\kk$. Let $\{f_n\}_n$
be a countable dense subset of $\kk$ and $A\subseteq [\nn]$
witnessing the analyticity of $\ccc$.

If $\ccc$ is metrizable in the pointwise topology, consists of
bounded functions and it is norm non-separable, then
$2^{\leqslant\nn}$ canonically embeds into $\kk$ with respect to
$\{f_n\}_n$ and $\ccc$, such that its image is norm non-separable.

Respectively, if $\ccc$ is hereditarily separable and not
metrizable, then either $\splp$ or $\splm$ canonically embeds into
$\kk$ with respect to $\{f_n\}_n$ and $\ccc$.
\end{thm}
The additional information, provided by Theorem \ref{reallynew},
is that the canonical embedding of the corresponding prototype
is found with respect to the dense subset $\{f_n\}_n$
of $\kk$ witnessing the analyticity of $\ccc$, which is not
necessarily a subset of $\ccc$. The proof of Theorem
\ref{reallynew} follows the lines of Theorems \ref{metrnew}
and \ref{embt2}, using the arguments of the proof of Proposition
\ref{embp1}.
\end{rem}


\section{Non-$G_\delta$ points in analytic subspaces}

This section is devoted to the study of the structure of not first
countable analytic subspaces. The first subsection is devoted to
the presentation of an extension of a result of A. Krawczyk from
\cite{Kra}. The proof follows the same lines as in \cite{Kra}. In
the second one we show that $\alex$ canonically embeds into any
not first countable analytic subspace $\ccc$ of a separable
Rosenthal compact $\kk$ and with respect to any countable
subset $D$ of $\kk$, witnessing the analyticity of $\ccc$.

\subsection{Krawczyk trees}

We begin by introducing some pieces of notation and by recalling some
standard terminology. By $\Sigma$ we shall denote the set of all
non-empty strictly increasing finite sequences of $\nn$. We view
$\Sigma$ as a tree equipped with the (strict) partial order
$\sqsubset$ of extension. We view, however, every $t\in \Sigma$
not only as a finite increasing sequence, but also, as a finite
subset of $\nn$. Thus, for every $t,s\in \Sigma$ with $\max s<\min t$
we will frequently denote by $s\cup t$ the concatenation of $s$ and $t$. By
$[\Sigma]$ we denote the branches of $\Sigma$, i.e. the set
$\{\sg\in \nn^\nn: \sg|n\in\Sigma \ \forall n\geq 1\}$.
For every $t\in \Sigma$ by $\Sigma_t$ we denote the
set $\{s\in\Sigma:t\sqsubseteq s\}$.

For every $A, B\in [\nn]$ we write $A\subseteq^* B$ if the set
$A\setminus B$ is finite. If $\aaa\subseteq [\nn]$, then we set
$\aaa^*=\{ \nn\setminus A: A\in\aaa\}$. For a pair $\aaa, \bbb\subseteq [\nn]$
we say that $\aaa$ is \textit{$\bbb$-generated} if for every $A\in\aaa$
there exist $B_0, ..., B_k\in\bbb$ such that $A\subseteq B_0\cup ...\cup B_k$.
We say that $\aaa$ is \textit{countably $\bbb$-generated} if there
exists a sequence $(B_n)_n$ is $\bbb$ such that $\aaa$ is
$\{B_n:n\in\nn\}$-generated. An ideal $\iii$ on $\nn$ is said to
be \textit{bi-sequential} if for every $p\in\beta\nn$ with
$\iii\subseteq p^*$, the family $\iii$ is countably $p^*$-generated.
Finally, for every family $\fff$ of subsets of $\nn$ and every
$A\subseteq \nn$ we let $\fff[A]=\{ L\cap A: L\in\fff\}$ to be the
\textit{trace} of $\fff$ on $A$. Observe that if $\fff$ is hereditary,
then $\fff[A]=\fff\cap \ppp(A)=\{ L\in \fff: L\subseteq A\}$. The
following lemma is essentially Lemma 1 from \cite{Kra}.
\begin{lem}
\label{kral1} Let $\iii$ be a bi-sequential ideal. Let also
$\fff\subseteq \iii$ and $A\in [\nn]$. Assume that $\fff[A]$
is not countably $\iii$-generated. Then there exists a sequence
$(A_n)_n$ of pairwise disjoint infinite subsets of $A$ such that, setting
$\mathcal{A}= \{A_n:n\in\nn\}$, we have that $\iii[A]$
is $\mathcal{A}$-generated, while $\fff[A_n]$ is not
countably $\iii$-generated for every $n\in\nn$.
\end{lem}
\begin{proof}
(Sketch) It suffices to prove the lemma for $A=\nn$. We set
\[ \jjj=\{ C\subseteq \nn: \fff[C] \text{ is countably } \iii-\text{generated}\}. \]
Then $\jjj$ is an ideal, $\nn\notin \jjj$ and $\iii\subseteq \jjj$.
We select $p\in \beta\nn$ with $\jjj\subseteq p^*$. By the bi-sequentiality
of $\iii$, there exists a sequence $(D_n)_n$ in $p^*$ such that
$\iii$ is $\{D_n:n\in\nn\}$-generated. As $p^*$ is an ideal,
we may assume that $D_n\cap D_m=\varnothing$ if $n\neq m$. Define
$M=\{ n\in\nn: D_n\in\jjj\}$. By the fact that $\iii$ is
$\{D_n:n\in\nn\}$-generated and $\fff\subseteq \iii$,
we have that $\fff$ is $\{D_n:n\in\nn\}$-generated.
This observation and the fact that $(D_n)_n$ are disjoint
yield that the set $D=\bigcup_{n\in M} D_n$
belongs to $\jjj$. Moreover, the set $\nn\setminus M$ is infinite
(for if not we would get that $\nn\in p^*$). Let $\{k_0<k_1<...\}$
be the increasing enumeration of $\nn\setminus M$ and define
$A_0=D\cup D_{k_0}$ and $A_n=D_{k_n}$ for $n\geq 1$. It is easy to
see that the sequence $(A_n)_n$ is as desired.
\end{proof}
The main result of this subsection is the following theorem, which
corresponds to Lemma 2 in \cite{Kra}. We notice that it is one
of the basic ingredients in the proof of the embedding of $\alex$
in not first countable separable Rosenthal compacta.
\begin{thm}
\label{thmkra} Let $\iii$ be a bi-sequential ideal and $\fff\subseteq \iii$
analytic and hereditary. Assume that $\fff$ is not countably $\iii$-generated.
Then there exists a 1-1 map $\kappa:\Sigma\to \nn$ such that,
setting $\jjj_\fff=\{ \kappa^{-1}(L):L\in\fff\}$ and
$\jjj=\{\kappa^{-1}(M):M\in\iii\}$, the following are satisfied.
\begin{enumerate}
\item[(1)] For every $\sg\in[\Sigma]$, $\{\sg|n:n\geq 1\}\in \jjj_\fff$.
\item[(2)] (Domination property) For every $B\in \jjj$ and every
$n\geq 1$ there exist $t_0,...,t_k\in \Sigma$ with
$|t_0|=...=|t_k|=n$ and such that
$B\subseteq^* \Sigma_{t_0}\cup ...\cup \Sigma_{t_k}$.
\end{enumerate}
\end{thm}
It is easy to see that property (2) in Theorem \ref{thmkra} is
equivalent to say that $B$ is contained in a finitely splitting
subtree of $\Sigma$.
\begin{proof}
We fix a continuous map $\phi:\nn^\nn\to [\nn]$ with $\phi(\nn^\nn)=\fff$.
Recursively, we shall construct the following.
\begin{enumerate}
\item[(C1)] A family $(A_s)_{s\in\bt}$ of infinite subsets of $\nn$.
\item[(C2)] A family $(a_s)_{s\in \bt}$ of finite subsets of $\nn$.
\item[(C3)] A family $(U_s)_{s\in\bt}$ of basic clopen subsets of
$\nn^\nn$.
\end{enumerate}
The construction is done so that the following are satisfied.
\begin{enumerate}
\item[(P1)] $A_s\subseteq A_t$ if $t\sqsubset s$ and $A_s\cap A_t=\varnothing$
if $s$ and $t$ are incomparable.
\item[(P2)] For every $s\in\bt$, $|a_s|=|s|$ and $\max a_s\in A_s$ for
all $s\in\bt$ with $s\neq (\varnothing)$. Moreover $a_s\sqsubset a_t$
if and only if $s\sqsubset t$.
\item[(P3)] $U_s\subseteq U_t$ if $t\sqsubset s$ and $\mathrm{diam}(U_s)\leq
\frac{1}{2^{|s|}}$. Moreover, $U_s\cap U_t=\varnothing$ if $s$ and $t$
are incomparable.
\item[(P4)] For every $s\in\bt$, $\phi(U_s)[A_s]$ is not countably
$\iii$-generated.
\item[(P5)] For every $s\in\bt$ and every $\tau\in U_s$ we have that
$a_s\subseteq \phi(\tau)$.
\item[(P6)] For every $s\in\bt$, $\iii[\bigcup_n A_{s^{\con}n}]$ is
$\{A_{s^{\con}n}:n\in\nn\}$-generated.
\end{enumerate}
Assuming that the construction has been carried out we complete
the proof as follows. We define $\lambda:\bt\setminus
\{(\varnothing)\}\to\nn$ by $\lambda(s)=\max a_s$. By (P1) and
(P2) above, we see that $\lambda$ is 1-1. Let $\sg\in \nn^\nn$. We
claim that $\{\lambda(\sg|n):n\geq 1\}=\bigcup_{n} a_{\sg|n}
\in\fff$. To see this, by (P3), let $\tau$ be the unique element
of $\bigcap_n U_{\sg|n}$. Then, by (P5), we have that
$a_{\sg|n}\subseteq \phi(\tau)$ for all $n\in\nn$. Thus,
$\bigcup_{n} a_{\sg|n}\subseteq \phi(\tau)\in\fff$. As $\fff$ is
hereditary, our claim is proved. Now, let $B\subseteq
\bt\setminus\{(\varnothing)\}$ be such that $\{\lambda(t):t\in
B\}\in\iii$. We claim that $B$ must be dominated, i.e. for every
$n\geq 1$ there exist $s_0, ..., s_k\in \nn^n$ such that $B$ is
almost included in the set of the successors of the $s_i$'s in
$\bt$. If not, then we may find a subset $\{t_n\}_n$ of $B$, a
node $s$ of $\bt$ and a subset $\{m_n:n\in\nn\}$ of $\nn$ such
that $s^{\con} m_n\sqsubseteq t_n$ for every $n\in\nn$. Notice
that $\{\lambda(t_n):n\in\nn\}\in\iii$, as $\iii$ is hereditary.
Moreover, by the definition of $\lambda$ and by properties (P1)
and (P2) above, we see that $\lambda(t_n)\in A_{s^{\con}m_n}$ for
every $n\in\nn$. Hence $\{\lambda(t_n):n\in\nn\}\in \iii[\bigcup_n
A_{s^{\con}n}]$. This leads to a contradiction by properties (P1)
and (P6) above. We set $\kappa=\lambda|_{\Sigma}$. Clearly
$\kappa$ is as desired.

We proceed to the construction. We set $A_{(\varnothing)}=\nn$,
$a_{(\varnothing)}=\varnothing$ and $U_{(\varnothing)}=\nn^\nn$.
Assume that $A_s$, $a_s$ and $U_s$ have been constructed for some
$s\in\bt$. We set $\fff_s=\phi(U_s)$. By property (P4) above and
by Lemma \ref{kral1}, there exists a sequence $(A_n)_n$ of
pairwise disjoint infinite subsets of $A_s$ such that
$\fff_s[A_n]$ is not countably $\iii$-generated for every
$n\in\nn$, while $\iii[A_s]$ is $\{A_n:n\in\nn\}$-generated.
Recursively, we may select a subset $\{\tau_n: n\in\nn\}$ in $U_s$
such that for all $n\in\nn$ the following are satisfied.
\begin{enumerate}
\item[(i)] $\phi(\tau_n)\cap A_n$ is infinite.
\item[(ii)] $\phi(V_{\tau_n|k})[A_n]$ is not countably $\iii$-generated
for every $k\in\nn$, where $V_{\tau_n|k}=\{\sg\in \nn^\nn: \tau_n|k\sqsubset \sg\}$.
\end{enumerate}
This can be easily done, as $\phi(U_s)[A_n]=\fff_s[A_n]$ is not countably $\iii$-generated
for every $n\in\nn$. We may select $L=\{l_0<l_1<...\}\in[\nn]$ and
a sequence $(k_n)_n$ in $\nn$ such that, setting $\sg_n=\tau_{l_n}$ for
all $n$, the following are satisfied.
\begin{enumerate}
\item[(iii)] $V_{\sg_n|k_n}\cap V_{\sg_m|k_m}=\varnothing$ if
$n\neq m$. \item[(iv)] For every $n\in\nn$ we have $V_{\sg_n|k_n}\subseteq
U_s$. \item[(v)] For every $n\in\nn$ we have $\mathrm{diam}(V_{\sg_n|k_n})<
\frac{1}{2^{|s|+1}}$.
\end{enumerate}
Using the continuity of the map $\phi$, for every $n\in\nn$ we may also select
$k'_n, i_n\in\nn$ such that the following are satisfied.
\begin{enumerate}
\item[(vi)] $i_n\in \phi(\sg_n)\cap A_{l_n}$.
\item[(vii)] $\max a_s< i_n$ and $k_n<k'_n$.
\item[(viii)] For every $\tau\in V_{\sg_n|k'_n}$ we have $i_n\in \phi(\tau)$.
\end{enumerate}
For every $n\in\nn$ we set $a_{s^{\con}n}=a_s\cup\{ i_n\}$,
$A_{s^{\con}n}=A_{l_n}$ and $U_{s^{\con}n}=V_{\sg_n|k'_n}$.
It is easy to check that conditions (P1)-(P6) are satisfied.
The proof of the theorem is completed.
\end{proof}

\subsection{The embedding of $\alex$ in analytic subspaces}

The main result of this subsection is the following.
\begin{thm}
\label{agdt1} Let $\kk$ be a separable Rosenthal compact and
$\mathcal{C}$ an analytic subspace of $\kk$. Let
$\{f_n\}_n$ be a countable dense subset of $\kk$ and
$A\subseteq [\nn]$ analytic, witnessing the analyticity of
$\mathcal{C}$. Let also $f\in\mathcal{C}$ be a non-$G_\delta$
point of $\mathcal{C}$. Then there exists a canonical homeomorphic
embedding of $\alex$ into $\kk$ with respect to $\{f_n\}_n$ and
$\ccc$ which sends $0$ to $f$.
\end{thm}
For the proof we need to make some preliminary steps. Let $\kk$,
$\ccc$, $\{f_n\}_n$ and $f\in\mathcal{C}$ be as in Theorem \ref{agdt1}.
We may assume that $f_n\neq f$ for all $n\in\nn$. We let
\[ \iii_f=\{ L\in [\nn]: f\notin \overline{\{f_n\}}^p_{n\in L}\}.\]
Then $\iii_f$ is an analytic ideal on $\nn$ (see \cite{Kra}).
A fundamental property of $\iii_f$ is that it is bi-sequential.
This is due to R. Pol (see \cite{Pol3}). We notice that the
bi-sequentiality of $\iii_f$ can be also derived by the
non-effective proof of the Bourgain-Fremlin-Talagrand
theorem due to G. Debs (see \cite{De}, or \cite{AGR}).

Let also $A\subseteq [\nn]$ be as in Theorem \ref{agdt1}.
As we have already pointed out, we may assume that $A$ is
hereditary. We set
\[ \mathcal{F}=A\cap \iii_f. \]
Clearly $\mathcal{F}$ is an analytic and hereditary subset of $\iii_f$.
The assumption that $f$ is a non-$G_\delta$ point of $\mathcal{C}$
yields (and in fact is equivalent to say) that $\mathcal{F}$ is
not countably $\iii_f$-generated, i.e. there does exist a sequence
$(M_k)_k$ in $\iii_f$ such that for every $L\in \mathcal{F}$ there exists
$k\in\nn$ with $L\subseteq M_0\cup ... \cup M_k$. To see this,
assume on the contrary that such a sequence $(M_k)_k$ existed.
We set $N_k=M_0\cup ...\cup M_k$ for every $k\in\nn$. As $\iii_f$
is an ideal, we see that $N_k\in\iii_f$ for every $k$.
We set $F_k=\overline{\{f_n\}}^p_{n\in N_k}\cup \{f_k\}$.
The fact that $N_k\in\iii_f$ implies that $f\notin F_k$ for every $k\in\nn$.
Let $g\in \ccc\cap Acc(\kk)$ with $g\neq f$. By condition (2)
of Definition \ref{nad1}, there exists $L\in A$ with
$g\in \overline{\{f_n\}}^p_{n\in L}$. Hence, there exists $M\in [L]$
such that $g$ is the pointwise limit of the sequence $(f_n)_{n\in M}$.
As $A$ is hereditary, we see that $M\in \fff$, and so, there
exists $k_0\in \nn$ with $M\subseteq N_{k_0}$. This implies that
$g\in F_{k_0}$. It follows by the above discussion that
$\{f\}=\bigcap_k (\mathcal{C}\setminus F_k)$,
that is the point $f$ is $G_\delta$ in $\mathcal{C}$,
a contradiction.

Summarizing, we have that $\iii_f$ is bi-sequential, $\fff\subseteq \iii_f$
is analytic, hereditary and not countably $\iii_f$-generated. Thus,
we may apply Theorem \ref{thmkra} and we get the 1-1 map $\kappa:\Sigma\to
\nn$ as described in the theorem. Setting $f_t=f_{\kappa(t)}$ for every
$t\in\Sigma$ and invoking condition (1) of Definition \ref{nad1},
we get the following corollary.
\begin{cor}
\label{krac1} There exists a family $\{f_t\}_{t\in\Sigma} \subseteq
\{f_n\}_n$ such that the following are satisfied.
\begin{enumerate}
\item[(1)] For every $\sg\in [\Sigma]$, we have that
$f\notin \overline{\{f_{\sg|n}\}}^p_n$ and $Acc\big(\{f_{\sg|n}:n\in\nn\}\big)
\subseteq \ccc$.
\item[(2)] For every $B\subseteq \Sigma$ such that $f\notin\overline{\{f_t\}}^p_{t\in B}$
and every $n\geq 1$, there exist $t_0, ..., t_k\in \Sigma$
with $|t_0|=...=|t_k|=n$ and
such that $B\subseteq^* \Sigma_{t_0}\cup ... \cup \Sigma_{t_k}$.
\end{enumerate}
\end{cor}
We call the family $\{f_t\}_{t\in\Sigma}$ obtained in Corollary
\ref{krac1} as the \textit{Krawczyk tree} of $f$ with respect to
$\{f_n\}_n$ and $\mathcal{C}$. Let us isolate the following property
of the Krawczyk tree $\{f_t\}_{t\in\Sigma}$ that we will need later on.
\begin{enumerate}
\item[(P)] Let $i\in\Sigma$ and $(b_n)_n$ be a sequence in $\Sigma$
such that $\max i<\min b_n$ and $\max b_n<\min b_{n+1}$ for every $n\in\nn$.
Then, setting $s_n=i\cup b_n$ for all $n\in\nn$, the sequence $(f_{s_n})_n$
is pointwise convergent to $f$. To see this observe that, by property
(2) in Corollary \ref{krac1}, every subsequence of the sequence $(f_{s_n})_n$
accumulates to $f$. Hence the sequence $(f_{s_n})_n$ is pointwise
convergent to $f$.
\end{enumerate}
We will also need the following well-known consequence of the
bi-sequentiality of $\iii_f$. For the sake of completeness we
include a proof.
\begin{lem}
\label{poll1} Let $(A_l)_l$ be a sequence in $[\nn]$ such that
$\lim_{n\in A_l} f_n=f$ for every $l\in\nn$. Then there
exists $D\in[\nn]$ with $\lim_{n\in D} f_n=f$, $D\subseteq \bigcup_l A_l$
and such that $D\cap A_l\neq \varnothing$ for infinitely many $l\in\nn$.
\end{lem}
\begin{proof}
For every $k\in\nn$ we set $B_k=\bigcup_{l\geq k} A_l$. Then
$(B_k)_k$ is a decreasing sequence of infinite subsets of $\nn$.
We may select $p\in\beta\nn$ such that $\lim_{n\in p} f_n=f$ and
$B_k\in p$ for every $k\in\nn$. By the bi-sequentiality of $\iii_f$,
there exists a sequence $(C_m)_m$ of elements of $p$ converging to $f$.
We select a strictly increasing sequence $(l_k)_k$ in $\nn$ such that
$l_k\in B_k\cap C_0\cap ...\cap C_k$ for all $k\in\nn$ and we
set $D=\{l_k:k\in\nn\}$. It is easy to check that $D$ is as
desired.
\end{proof}
In the sequel we will apply K. Milliken's theorem \cite{Mil}. To
this end, we need to recall some pieces of notation. Given $b,
b'\in\Sigma$ we write $b<b'$ if $\max b<\min b'$. By $\mathbf{B}$
we denote the subset of $\Sigma^\nn$ consisting of all sequences
$(b_n)_n$ which are increasing, i.e. $b_n<b_{n+1}$ for every
$n\in\nn$. It is easy to see that $\mathbf{B}$ is a closed
subspace of $\Sigma^\nn$, where $\Sigma$ is equipped with the
discrete topology and $\Sigma^\nn$ with the product topology. For
every $\mathbf{b}=(b_n)_n\in \mathbf{B}$ we set
\[ \langle \mathbf{b}\rangle=\big\{ \bigcup_{n\in F} b_n: F\subseteq\nn
\text{ finite}\big\} \  \text{ and } \
[\mathbf{b}]=\{(c_n)_n\in\mathbf{B}: c_n\in \langle\mathbf{b}\rangle \
\forall n\}.\]
Let us point out that for every block sequence $\mathbf{b}$ the set
$\langle\mathbf{b}\rangle$ corresponds to an infinitely branching
subtree of $\Sigma$, denoted by $\mathcal{T}_{\mathbf{b}}$.
We also recall that the chains of $\mathcal{T}_{\mathbf{b}}$
are in one-to-one correspondence with the family $[\mathbf{b}]$
of all block subsequences of $\mathbf{b}$.
More precisely if $(t_n)_n$ is a chain of $\mathcal{T}_{\mathbf{b}}$,
then $(t_0, t_1\setminus t_0, ..., t_{n+1}\setminus t_n, ...)$
is the block subsequence of $\mathbf{b}$ corresponding to the
chain $(t_n)_n$. This observation was used by W. Henson to derive
an alternative proof of J. Stern's theorem (see \cite{Od}).
If $\beta=(b_0,...,b_k)$ with $b_0<...<b_k$ and
$\mathbf{d}\in\mathbf{B}$, then we set
\[ [\beta,\mathbf{d}]= \{ (c_n)_n\in\mathbf{B}: c_n=b_n \ \forall n\leq k
\text{ and } c_n\in \langle\mathbf{d}\rangle \ \forall n>k\}.\]
We will use the following consequence of Milliken's theorem.
\begin{thm}
\label{milth} For every $\mathbf{b}\in\mathbf{B}$ and every analytic subset
$A$ of $\mathbf{B}$ there exists $\mathbf{c}\in [\mathbf{b}]$
such that either $[\mathbf{c}]\subseteq A$, or $[\mathbf{c}]
\cap A=\varnothing$.
\end{thm}
For every $\mathbf{b}=(b_n)_n\in \mathbf{B}$ and every $n\in\nn$
we set $i_n=\bigcup_{i=0}^n b_i$. We define $C:\mathbf{B}\to \Sigma^\nn$
and $A:\mathbf{B}\to \Sigma^\nn$ by
\[  C\big( (b_n)_n\big)=(i_0, ..., i_n, ...) \ \text{ and } \
A\big( (b_n)_n\big)=(i_0\cup b_2, ..., i_{3n}\cup b_{3n+2}, ...).\]
We notice that for every $\mathbf{b}\in\mathbf{B}$ the sequence
$C(\mathbf{b})$ is a chain of $\Sigma$ while $A(\mathbf{b})$ is
an antichain of $\Sigma$ converging, in the sense of Definition
\ref{rd1}, to $\sg=\bigcup_n i_n\in [\Sigma]$.
We also notice that the functions $C$ and $A$ are continuous.
\begin{lem}
\label{mill1} Let $\{f_t\}_{t\in\Sigma}$ be a Krawczyk tree of $f$
with respect to $\{f_n\}_n$ and $\mathcal{C}$. Then there exists a
block sequence $\mathbf{b}=(b_n)_n$ such that for every
$\mathbf{c}\in [\mathbf{b}]$ the sequence $(f_t)_{t\in C(\mathbf{c})}$
is pointwise convergent to a function belonging to $\ccc$ and
different than $f$, while the sequence $(f_t)_{t\in A(\mathbf{c})}$
is pointwise convergent to $f$.
\end{lem}
\begin{proof}
Let
\[ C_1=\{ \mathbf{c}\in\mathbf{B}: \text{the sequence }
(f_t)_{t\in C(\mathbf{c})} \text{ is pointwise convergent}\}.\]
It is easy to see that $C_1$ is a co-analytic subset of
$\mathbf{B}$. By Theorem \ref{milth} and the sequential
compactness of $\kk$, we find $\mathbf{d}\in\mathbf{B}$ such
that $[\mathbf{d}]$ is a subset of $C_1$. As we have already
remarked, for every block sequence $\mathbf{c}$ the sequence
$C(\mathbf{c})$ is a chain of $\Sigma$. Hence, by Corollary
\ref{krac1}(1), we see that for every $\mathbf{c}\in [\mathbf{d}]$
the sequence $(f_t)_{t\in C(\mathbf{c})}$ must be pointwise
convergent to a function belonging to $\ccc$ and different
than $f$.

Now let
\[ C_2=\{ \mathbf{c}\in [\mathbf{d}]: \text{the sequence }
(f_t)_{t\in A(\mathbf{c})} \text{ is pointwise convergent to } f\}. \]
Again by Milliken's theorem, there exists $\mathbf{b}=(b_n)_n\in [\mathbf{d}]$
such that either $[\mathbf{b}]\subseteq C_2$, or $[\mathbf{b}]\cap
C_2=\varnothing$. We claim that $[\mathbf{b}]$ is subset of $C_2$.
It is enough to show that $[\mathbf{b}]\cap C_2\neq \varnothing$.
 To this end we argue as follows. Recall that for every
$l\in\nn$ we have set $i_l=b_0\cup ...\cup b_l$. Let
\[ A_l=\{i_l\cup b_m: m>l+1\} \subseteq \Sigma. \]
As the sequence $(b_n)_n$ is block, by property (P) above,
we see that the sequence $(f_t)_{t\in A_l}$ is pointwise
convergent to $f$. By Lemma \ref{poll1}, there exists
$D\subseteq \bigcup_{l} A_l$ such that the sequence $(f_t)_{t\in D}$
is pointwise convergent to $f$ and $D\cap A_l\neq\varnothing$
for infinitely many $l$. We may select
$L=\{l_0<l_1<...\}, M=\{m_0<m_1<...\}\in [\nn]$ such that
$l_n+1<m_n<l_{n+1}$ and $i_{l_n}\cup b_{m_n}\in D$ for all $n\in\nn$.
Now we define $\mathbf{c}=(c_n)_n\in [\mathbf{b}]$ as follows.
We set $c_0=i_{l_0}$, $c_1=b_{l_0+1}\cup ... \cup b_{m_0-1}$
and $c_2=b_{m_0}$. For every $n\in\nn$ with $n\geq 1$ let
$I_n=[m_{n-1}+1,l_n]$ and $J_n=[l_n,m_n-1]$ and set
\[ c_{3n}=\bigcup_{i\in I_n} b_i \ , \
c_{3n+1}=\bigcup_{i\in J_n} b_i \text{ and }
c_{3n+2}=b_{m_n}. \]
It is easy to see that $\mathbf{c}\in [\mathbf{b}]$ and
$A(\mathbf{c})=(i_{l_n}\cup b_{m_n})_n\subseteq D$.
Hence, the sequence $(f_t)_{t\in A(\mathbf{c})}$ is
pointwise convergent to $f$. It follows that $[\mathbf{b}]\cap
C_2 \neq\varnothing$ and the proof is completed.
\end{proof}
We are ready to proceed to the proof of Theorem \ref{agdt1}.
\begin{proof}[Proof of Theorem \ref{agdt1}]
Let $\mathbf{b}=(b_n)_n$ be the block sequence obtained
by Lemma \ref{mill1}. If $\beta=(b_{n_0},..., b_{n_k})$
with $n_0<...<n_k$ is a finite subsequence of $\mathbf{b}$,
then we let $\cup \beta=b_{n_0}\cup ...\cup b_{n_k}\in\Sigma$.
Recursively, we shall select a family $(\beta_s)_{s\in\ct}$
such that the following are satisfied.
\begin{enumerate}
\item[(C1)] For every $s\in\ct$, $\beta_s$ is a finite
subsequence of $\mathbf{b}$.
\item[(C2)] For every $s, s'\in \ct$ we have $s\sqsubset s'$
if and only if $\beta_s\sqsubset \beta_{s'}$.
\item[(C3)] For every $s\in\ct$ and every
$\mathbf{c}\in [\beta_{s^{\con}0}, \mathbf{b}]$ we have
$\cup \beta_{s^{\con}1}\in A(\mathbf{c})$.
\end{enumerate}
The construction proceeds as follows. We set
$\beta_{(\varnothing)}=(\varnothing)$. For every
$M=\{m_0<m_1<...\}\in [\nn]$, let $\mathbf{b}_M=(b_{m_n})_n$ be
the subsequence of $\mathbf{b}$ determined by $M$. Assume that for
some $s\in \ct$ the finite sequence $\beta_s$ has been defined. We
select $M=M_s\in [\nn]$ such that $\beta_s\sqsubset
\mathbf{b}_{M}$. The set $A(\mathbf{b}_{M})$ converges to the
unique branch of $\Sigma$ determined by the infinite chain
$C(\mathbf{b}_M)$. So, we may select a finite subsequence
$\beta_{s^{\con}1}$ with $\beta_s\sqsubset\beta_{s^{\con}1}$ and
such that $\cup \beta_{s^{\con}1}\in A(\mathbf{b}_{M})$. The
function $A:[\mathbf{b}]\to \Sigma^\nn$ is continuous. Hence,
there exists a finite subsequence $\beta_{s^{\con}0}$ of
$\mathbf{b}$ with $\beta_{s^{\con}0}\sqsubset \mathbf{b}_M$ and
such that condition (C3) above is satisfied. Finally, notice that
$\beta_{s^{\con}0}$ and $\beta_{s^{\con}1}$ are incomparable with
respect to the partial order $\sqsubset$ of extension.

One can also provide a recursive formula defining a family
$(\beta_s)_{s\in\ct}$ satisfying conditions (C1)-(C3) above.
In particular, set $\beta_{(\varnothing)}=(\varnothing)$,
$\beta_{(0)}=(b_0,b_1,b_2)$ and $\beta_{(1)}=(b_0,b_2)$.
Assume that $\beta_s$ has been defined for some $s\in\ct$.
Let $n_s=\max\{ n: b_n\in \beta_s\}$. If $s$ ends with $0$,
then we set
\[ \beta_{s^{\con}0}=\beta_s^{\con}(b_{n_s+1}, b_{n_s+2}, b_{n_s+3})
\ \text{ and } \ \beta_{s^{\con}1}= \beta_s^{\con}(b_{n_s+1},b_{n_s+3}).\]
If $s$ ends with $1$, then we set
\[ \beta_{s^{\con}0}=\beta_s^{\con}(b_{n_s+1}, b_{n_s+2},
b_{n_s+3}, b_{n_s+4}) \
\text{ and } \beta_{s^{\con}1}=\beta_s^{\con}(b_{n_s+1},b_{n_s+2}, b_{n_s+4}). \]
It is easy to see that, with the above choices,
conditions (C1)-(C3) are satisfied.

Having defined the family $(\beta_s)_{s\in\ct}$ for every $s\in \ct$
we let
\[ t_s=\cup \beta_s\in\Sigma \text{ and } h_s=f_{t_s}. \]
Clearly the family $\{h_s\}_{s\in\ct}$ is a dyadic subtree of the
Krawczyk tree $\{f_t\}_{t\in\Sigma}$ of $f$ with respect to
$\{f_n\}_n$ and $\mathcal{C}$. The basic properties of the family
$\{h_s\}_{s\in\ct}$ are summarized in the following claim.
\bigskip

\noindent \textsc{Claim 1.} \textit{The following hold.
\begin{enumerate}
\item[(1)] For every $\sg\in 2^\nn$ the sequence $(h_{\sg|n})_n$
is pointwise convergent to a function $g_\sg\in\mathcal{C}$ with
$g_\sg\neq f$.
\item[(2)] For every $P\subseteq 2^\nn$ perfect the function
$f$ belongs to the closure of the family $\{g_\sg:\sg\in P\}$.
\end{enumerate} }
\bigskip

\noindent \textit{Proof of the claim.} (1) Let $\sg\in 2^\nn$
and put $\mathbf{b}_{\sg}=\bigcup_n \beta_{\sg|n}\in [\mathbf{b}]$.
It is easy to see that the sequence $(t_{\sg|n})_n$
is a subsequence of the sequence $C(\mathbf{b}_\sg)$. So the
result follows by Lemma \ref{mill1}.\\
(2) Assume not. Then there exist $P\subseteq 2^\nn$ perfect and a
neighborhood $V$ of $f$ in $\rr^X$ such that $g_\sg\notin
\overline{V}$ for all $\sg\in P$. By part (1), for every $\sg\in
P$ there exists $n_\sg\in \nn$ such that $h_{\sg|n}\notin V$ for
all $n\geq n_\sg$. For every $n\in\nn$ let $P_n=\{\sg\in P:
n_\sg\leq n\}$. Then each $P_n$ is a closed subset of $P$ and
clearly $P=\bigcup_n P_n$. Thus, there exist $n_0\in\nn$ and
$Q\subseteq 2^\nn$ perfect with $Q\subseteq P_{n_0}$. It follows
that $h_{\sg|n}\notin V$ for all $\sg\in Q$ and $n\geq n_0$. Let
$\tau$ be the lexicographical minimum of $Q$. We may select a
sequence $(\sg_k)_k$ in $Q$ such that, setting $s_k=\tau\wedge
\sg_k$ for all $k\in\nn$, we have that $\sg_k\to \tau$, $\tau\prec
\sg_k$ and $|s_k|>n_0$. Notice that $s_k^{\con}0\sqsubset\tau$
while $s_k^{\con}1\sqsubset \sg_k$ and $|s_k^{\con}1|>n_0$. Hence,
by our assumptions on the set $Q$ and the definition of
$\{h_s\}_{s\in\ct}$, we get that
\begin{eqnarray}
\label{crucial} h_{s_k^{\con}1}=f_{t_{s_k^{\con}1}}\notin V \ \text{ for all } k\in\nn.
\end{eqnarray}
We are ready to derive the contradiction.
We set $\mathbf{b}_\tau=\bigcup_n \beta_{\tau|n}\in [\mathbf{b}]$.
As $\beta_{s_k^{\con}0}\sqsubset \mathbf{b}_\tau$, by property (C3)
in the above construction, we see that $t_{s_k^{\con}1}=\cup
\beta_{s_k^{\con}1}\in A(\mathbf{b}_\tau)$ for all $k\in\nn$.
By Lemma \ref{mill1}, the sequence $(f_t)_{t\in A(\mathbf{b}_\tau)}$
is pointwise convergent to the function $f$. It follows
that the sequence $(f_{t_{s_k^{\con}1}})_k$ is also
pointwise convergent to $f$, which clearly contradicts
(\ref{crucial}) above. The proof is completed. \hfill
$\lozenge$
\bigskip

\noindent We apply Theorem \ref{basis} to the family
$\{h_s\}_{s\in \ct}$ and we get a regular dyadic subtree
$T=(s_t)_{t\in\ct}$ of $\ct$ such that the family $\{h_{s_t}\}_{t\in\ct}$
is canonicalized. The main claim is the following.
\bigskip

\noindent \textsc{Claim 2.} \textit{The family $\{h_{s_t}\}_{t\in\ct}$
is equivalent to the canonical dense family of $\alex$. }
\bigskip

\noindent \textit{Proof of the claim.}  In order to prove
the claim we will isolate a property of the whole family
$\{h_s\}_{s\in\ct}$ (property (Q) below). Let $S$ be an
arbitrary regular dyadic subtree of $\ct$. Notice that
$g_\sg\in\overline{\{h_s\}}^p_{s\in S}$ for every
$\sg\in[\hat{S}]$. By property (2) in Claim 1,
we see that the function $f$ belongs to the pointwise closure of
$\{h_s\}_{s\in S}$ in $\rr^X$. By the Bourgain-Fremlin-Talagrand
theorem there exists $A\subseteq S$ such that the sequence
$(h_s)_{s\in A}$ is pointwise convergent to $f$. By property
(1) in Claim 1, we see that $A$ can be chosen to
be an antichain converging to some $\sg\in [\hat{S}]$.
As all these facts hold for every regular dyadic subtree
$S$ of $\ct$ we arrive to the following property of the
family $\{h_s\}_{s\in\ct}$.
\begin{enumerate}
\item[(Q)] For every regular dyadic subtree $S$ of $\ct$, there
exist two antichains $A_1, A_2$ of $S$ and $\sg_1,\sg_2\in
[\hat{S}]$ with $\sg_1\neq\sg_2$ such that $A_1$ converges to
$\sg_1$, $A_2$ converges to $\sg_2$ while both sequences
$(h_s)_{s\in A_1}$ and $(h_s)_{s\in A_2}$ are pointwise convergent
to $f$.
\end{enumerate}
Now let $T=(s_t)_{t\in\ct}$ be the regular dyadic subtree of $\ct$
such that the family $\{h_{s_t}\}_{t\in\ct}$ is canonicalized.
Invoking property (Q) above and referring to the description of
the families $\{d^i_t:t\in\ct\} \ (1\leq i\leq 7)$, we see that
$\{h_{s_t}\}_{t\in \ct}$ must be equivalent either to the
canonical dense family of $A(\ct)$ or the canonical dense family
of $\alex$. By property (1) in Claim 1, the first case
is impossible. It follows that $\{h_{s_t}\}_{t\in\ct}$
must be equivalent to the canonical dense family of $\alex$.
The claim is proved. \hfill $\lozenge$
\bigskip

Let $T=(s_t)_{t\in\ct}$ and $\{h_{s_t}\}_{t\in\ct}$ be as above.
Observe that for every $t\in \ct$ there exists a unique $n_t\in\nn$
with $h_{s_t}=f_{n_t}$. Thus, by passing to dyadic subtree of $T$
if necessary and invoking the minimality of the canonical
dense family of $\alex$, we get that the function
$\ct\ni t\mapsto n_t\in\nn$ is a canonical injection
and that the map
 \[ \alex\ni v_t \mapsto f_{n_t}\in \kk\]
is extended to homeomorphism $\Phi$ between $\alex$ and $\overline{\{
f_{n_t}\}}^p_{t\in\ct}$. That this homeomorphism sends $0$ to
$f$ is an immediate consequence of property (Q) in Claim 2 above.
Moreover, by Claim 1(1), we see that $\Phi(\delta_\sg)\in\mathcal{C}$
for every $\sg\in 2^\nn$. The proof is completed.
\end{proof}
By Theorem \ref{agdt1} and Proposition \ref{nap1}(1) we get the following corollary.
\begin{cor}
\label{gdt1} Let $\kk$ be a separable Rosenthal compact on a Polish
space $X$, $\{f_n\}_n$ a countable dense subset of $\kk$ and $f\in\kk$.
If $f$ is a non-$G_\delta$ point of $\kk$, then there exists a canonical
homeomorphic embedding of $\alex$ into $\kk$ with respect to $\{f_n\}_n$
which sends $0$ to $f$.
\end{cor}
After a first draft of the present paper, S. Todor\v{c}evi\'{c}
informed us (\cite{To3}) that he is aware of the above corollary
with a proof based on his approach in \cite{To1}.

We notice that if $\kk$ is a non-metrizable separable Rosenthal
compact on a Polish space $X$, then the constant function $0$ is
a non-$G_\delta$ point of $\kk-\kk$. Indeed, since $\kk$ is
non-metrizable, for every $D\subseteq X$ countable there exist
$f,g\in\kk$ with $f\neq g$ and such that $f|_D=g|_D$. This easily
yields that $0$ is a non-$G_\delta$ point of $\kk-\kk$. By Corollary
\ref{gdt1}, we see that there exists a homeomorphic embedding
of $\alex$ into $\kk-\kk$ with $0$ as the unique non-$G_\delta$ point
of its image. This fact can be lifted to the class of analytic
subspaces, as follows.
\begin{cor}
\label{kmk} Let $\kk$ be a separable Rosenthal compact and $\ccc$
an analytic subspace of $\kk$ which is non-metrizable. Let also
$D=\{f_n\}_n$ be a countable dense subset of $\kk$ witnessing
the analyticity of $\ccc$. Then there exists a family
$\{f_t\}_{t\in\ct}\subseteq D-D$, equivalent to the canonical
dense family of $\alex$, with $Acc\big(f_t:t\in\ct\}\big)\subseteq
\ccc-\ccc$ and such that the constant function $0$
is the unique non-$G_\delta$ point of
$\overline{\{f_t\}}^p_{t\in\ct}$.
\end{cor}
\begin{proof}
Let $\{g_n\}_n$ be an enumeration of the set
$D-D$, which is dense in $\kk-\kk$. It is easy to see that
$\ccc-\ccc$ is analytic subspace of $\kk-\kk$,
witnessed by the sequence $\{g_n\}_n$. Moreover, by the fact that
$\ccc$ is non-metrizable, we get that the constant function
$0$ belongs to $\ccc-\ccc$ and it is a non-$G_\delta$
point of it. By Theorem \ref{agdt1},
the result follows.
\end{proof}


\section{Connections with Banach space Theory}

This section is devoted to applications, motivated by the results
obtained in \cite{ADKbanach}, of the embedding of $\alex$ in
analytic subspaces of separable Rosenthal compacta containing $0$
as a non-$G_\delta$ point. The first one concerns the existence of
unconditional families. The second deals with spreading and level
unconditional tree bases.


\subsection{Existence of unconditional families}

We recall that a family $\{x_i\}_{i\in I}$ in a Banach space $X$
is said to be 1-unconditional, if for every $F\subseteq G\subseteq
I$ and every $(a_i)_{i\in G}\in \rr^G$ we have
\[ \Big\| \sum_{i\in F} a_i x_i\Big\| \leq
\Big\| \sum_{i\in G} a_i x_i\Big\|. \] We will need the following
reformulation of Theorem 4 in \cite{ADKbanach}, where we also
refer the reader for a proof.
\begin{thm}
\label{unct1} Let $X$ be a Polish space and $\{f_\sg:\sg\in
2^\nn\}$ be a bounded family of real-valued functions on $X$ which
is pointwise discrete and having the constant function $0$ as the
unique accumulation point in $\rr^X$. Assume that the map
$\Phi:2^\nn\times X\to\rr$ defined by $\Phi(\sg,x)=f_\sg(x)$ is
Borel. Then there exists a perfect subset $P$ of $2^\nn$ such that
the family $\{f_\sg:\sg\in P\}$ is 1-unconditional in the supremum
norm.
\end{thm}
In \cite{ADKbanach} it is shown that if $X$ is a separable Banach
space not containing $\ell_1$ and with non-separable dual, then
$X^{**}$ contains an $1$-unconditional family of the size of the
continuum. This result can be lifted to the frame of separable
Rosenthal compacta, as follows.
\begin{thm}
\label{tC} Let $\kk$ be a separable Rosenthal compact on a Polish
space $X$. Let also $\ccc$ be an analytic subspace of $\kk$
consisting of bounded functions.
\begin{enumerate}
\item[(a)] If $\ccc$ contains the function $0$ as a non-$G_\delta$
point, then there exists a family $\{f_\sg:\sg\in 2^\nn\}$ in
$\ccc$ which is 1-unconditional in the supremum norm, pointwise
discrete and having $0$ as unique accumulation point. \item[(b)]
If $\ccc$ is non-metrizable, then there exists a family
$\{f_\sg-g_\sg:\sg\in 2^\nn\}$, where $f_\sg, g_\sg\in\ccc$ for
all $\sg\in 2^\nn$, which is 1-unconditional in the supremum norm.
\end{enumerate}
\end{thm}
\begin{proof}
(a) Let $D=\{f_n\}_n$ be a countable dense subset of $\kk$
witnessing the analyticity of $\ccc$. As $0$ is a non-$G_\delta$
point of $\ccc$, by Theorem \ref{agdt1} there exists a family
$\{f_t\}_{t\in\ct}\subseteq D$, equivalent to the canonical dense
family of $\alex$, with $Acc\big(\{f_t:t\in\ct\}\big)\subseteq
\ccc$ and such that the constant function $0$ is the unique
non-$G_\delta$ point of $\overline{\{f_t\}}^p_{t\in\ct}$. For
every $\sg\in 2^\nn$ let $f_\sg$ be the pointwise limit of the
sequence $(f_{\sg|n})_n$. Clearly the family $\{f_\sg:\sg\in
2^\nn\}$ is pointwise discrete, having $0$ as the unique
accumulation point. Moreover, it is easy to see that the map
$\Phi:2^\nn\times X\to \rr$ defined by $\Phi(\sg,x)=f_\sg(x)$ is
Borel. By Theorem \ref{unct1}, the result follows.\\
(b) It follows by Corollary \ref{kmk} and Theorem \ref{unct1}.
\end{proof}
Actually we can strengthen the properties of the family
$\{f_\sg:\sg\in 2^\nn\}$ obtained by part (a) of Theorem \ref{tC}
as follows.
\begin{thm}
\label{biorthogonal} Let $\kk$ be a separable Rosenthal compact
on a Polish space $X$ and $\ccc$ be an analytic subspace of $\kk$
consisting of bounded functions. Assume that $\ccc$ contains the
function $0$ as a non-$G_\delta$ point. Then there exist a family
$\{(g_\sg,x_\sg):\sg\in 2^\nn\}\subseteq \ccc\times X$ and $\ee>0$
satisfying $|g_\sg(x_\sg)|>\ee$, $g_\sg(x_\tau)=0$ if $\sg\neq \tau$
and such that the family $\{g_\sg:\sg\in 2^\nn\}$ is 1-unconditional
in the supremum norm and having $0$ as unique accumulating point.
\end{thm}
\begin{proof}
Let $\{f_\sg:\sg\in 2^\nn\}\subseteq \ccc$ be the family obtained
by Theorem \ref{tC}(a). We notice that, by the proof of Theorem
\ref{tC}, we also have that the map $\Phi:2^\nn\times X\to \rr$
defined by $\Phi(\sg,x)=f_\sg(x)$ is Borel. Using this and by
passing to a perfect subset of $2^\nn$ if necessary, we may find
$\ee>0$ such that $\|f_\sg\|_\infty>\ee$ for all $\sg\in 2^\nn$.
Define $N\subseteq 2^\nn\times X$ by
\[ (\sg,z)\in N \Leftrightarrow |f_\sg(z)|>\ee.\]
As the map $\Phi$ is Borel, we see that the set $N$ is Borel.
Moreover, by the choice of $\ee$, we have that for every
$\sg\in 2^\nn$ the section $N_\sg=\{z:(\sg,z)\in N\}$ of $N$
at $\sg$ is non-empty. By the Yankov-Von Neumann Uniformization
theorem (see \cite{Kechris}, Theorem 18.1), there exists
a map
\[ 2^\nn\ni \sg\mapsto z_\sg\in X\]
which is measurable with respect to the $\sg$-algebra generated
by the analytic sets and such that $(\sg,z_\sg)\in N$
for every $\sg\in 2^\nn$. Invoking the classical fact
that analytic sets have the Baire property, by Theorem
8.38 in \cite{Kechris} and by passing to a further perfect
subset of $2^\nn$ if necessary, we may assume that the
map $\sg\mapsto z_\sg$ is actually continuous.

For every $m\in\nn$ define $A_m\subseteq 2^\nn\times 2^\nn$ by
\[ (\sg,\tau)\in A_m \Leftrightarrow |f_\tau(z_\sg)|>\frac{1}{m+1}.\]
Notice that the set $A_m$ is Borel. Since the family $\{f_\sg:\sg\in 2^\nn\}$
accumulates to $0$, we get that for every $\sg\in 2^\nn$
the section $(A_m)_\sg=\{\tau:(\sg,\tau)\in A_m\}$ of $A_m$
at $\sg$ is finite, hence meager in $2^\nn$. By the Kuratowski-Ulam
theorem (see \cite{Kechris}, Theorem 8.41), we have that the
set $A_m$ is meager in $2^\nn\times 2^\nn$. Hence so is
the set
\[ A=\bigcup_{m\in\nn} A_m.\]
By a result of J. Mycielski (see \cite{Kechris}, Theorem 19.1)
there exists $P\subseteq 2^\nn$ perfect such that for every
$\sg,\tau\in P$ with $\sg\neq \tau$ we have that $(\sg,\tau)\notin A$
and $(\tau,\sg)\notin A$. This implies that $f_\tau(z_\sg)=0$
and $f_\sg(z_\tau)=0$. We fix a homeomorphism $h:2^\nn\to P$
and we set $g_\sg=f_{h(\sg)}$ and $x_\sg=z_{h(\sg)}$ for every
$\sg\in 2^\nn$. Clearly the family $\{(g_\sg,x_\sg):\sg\in 2^\nn\}$
is as desired.
\end{proof}
The proof of the corresponding result in \cite{ADKbanach} is based
on Ramsey and Banach space tools, avoiding the embedding of
$\alex$ into $(B_{X^{**}}, w^*)$.

We recall that a Banach space $X$ is said to be representable if
$X$ isomorphic to a subspace of $\ell_\infty(\nn)$ which is
analytic in the weak* topology (see \cite{GT}, \cite{GL} and
\cite{AGR}). We close this subsection with the following.
\begin{thm}
\label{uncp1} Let $X$ be a non-separable representable Banach
space. Then $X^{*}$ contains an unconditional family of size
$|X^{*}|$.
\end{thm}
\begin{proof}
Identify $X$ with its isomorphic copy in $\ell_\infty(\nn)$. Then
$B_X$ is an analytic subset of $(B_{\ell_\infty},w^*)$. Let
$f:\nn^\nn\to B_X$ be an onto continuous map. Let $\{x_n\}_n$ be a
norm dense subset of $\ell_1(\nn)$. Viewing $\ell_1$ as a subspace
of $\ell_\infty^*$, we define $f_n:\nn^\nn\to\rr$ by $f_n=x_n\circ
f$. Then $\{f_n\}_n$ is a uniformly bounded sequence of continuous
real-valued functions on $\nn^\nn$. Notice that
$\overline{\{f_n\}}^p_n=\{ x^{*}\circ f: x^*\in B_{X^*}\}$ which
can be naturally identified with $\{x^*|_{B_X}:x^*\in B_{X^*}\}$.
By the non-effective version of Debs' theorem (see \cite{AGR}) we
have the following alternatives.
\bigskip

\noindent \textbf{Case 1.} There exist an increasing sequence
$(n_k)_k$, a continuous map $\phi:2^\nn\to\nn^\nn$ and real
numbers $a<b$ such that for every $\sg\in 2^\nn$ and every
$k\in\nn$, if $\sg(k)=0$ then $f_{n_k}\big(\phi(\sg)\big)<a$,
while if $\sg(k)=1$, then $f_{n_k}\big(\phi(\sg)\big)>b$. In this
case, for every $p\in\beta\nn$ we set
\[ g_p=p-\lim_{k} f_{n_k}. \]
Then $g_p=x^*_p|_{B_X}$ for some $x^*_p\in X^*$. We claim that
$\{x^*_p:p\in\beta\nn\}$ is equivalent to the natural basis of
$\ell_1(2^{\mathfrak{c}})$. To see this, observe that $g_p\big(
\phi(\sg)\big)\leq a$ if and only if $\{k:\sg(k)=0\}\in p$ and
$g_p\big(\phi(\sg)\big)\geq b$ if and only if $\{k:\sg(k)=1\}\in
p$. Setting $A_p=[g_p\leq a]$ and $B_p=[g_p\geq b]$ for all
$p\in\beta\nn$, we see that the family $(A_p,B_p)_{p\in\beta\nn}$
is an independent family of disjoint pairs. By Rosenthal's
criterion, the family $\{g_p:p\in\beta\nn\}$ is equivalent to
$\ell_1(2^{\mathfrak{c}})$. Thus, so is $\{x^*_p:p\in \beta\nn\}$.
\bigskip

\noindent \textbf{Case 2.} The sequence $\{f_n\}_n$ is relatively
compact in $\mathcal{B}_1(\nn^\nn)$. In this case, as $X$ is
non-separable, $0\in \overline{\{f_n\}}^p_n$ is a non-$G_\delta$
point. Thus, by Theorem \ref{tC}(a), there exists an
1-unconditional family in $X^{*}$ of the size of the continuum.
\end{proof}
It can be also shown that every representable Banach space has a
separable quotient (see Theorem 15 in \cite{ADKbanach}). For
further applications of the existence of unconditional families we
refer the reader to \cite{ADKbanach}.


\subsection{Spreading and level unconditional tree bases}

We start with the following definition.
\begin{defn}
\label{dtreebases} Let $X$ be a Banach space.
\begin{enumerate}
\item[(1)] A tree basis is a bounded family $\{x_t\}_{t\in\ct}$ in
$X$ which is Schauder basic when it is enumerated according to the
canonical bijection $\phi_0$ between $\ct$ and $\nn$. \item[(2)] A
tree basis $\{x_t\}_{t\in\ct}$ is said to be spreading if there
exists $(\ee_n)_n\downarrow 0$ such that for every $n,m \in\nn$
with $n<m$, every $0 \leq d< 2^n$ and every pair
$\{s_i\}_{i=0}^d\subseteq 2^n$ and $\{t_i\}_{i=0}^d\subseteq 2^m$
with $s_i\sqsubset t_i$ for all $i\in\{0,...,d\}$, we have that
$\|T\|\cdot \|T^{-1}\|<1+\ee_n$ where
$T:\mathrm{span}\{x_{s_i}:i=0,...,d\}\to
\mathrm{span}\{x_{t_i}:i=0,...,d\}$ is the natural 1-1 and onto
linear operator. \item[(3)] A tree basis $\{x_t\}_{t\in\ct}$ is
said to be level unconditional if there exists
$(\ee_n)_n\downarrow 0$ such that for every $n\in\nn$, the family
$\{x_t: t\in 2^n\}$ is $(1+\ee_n)$-unconditional.
\end{enumerate}
\end{defn}
In \cite{ADKbanach} the existence of spreading and level
unconditional tree bases was established for every separable
Banach space $X$ not containing $\ell_1$ and with non-separable
dual.

This result can be extended in the frame of separable Rosenthal
compacta, as follows.
\begin{thm}
\label{nst1} Let $\kk$ be a uniformly bounded separable Rosenthal
compact on a compact metrizable space $X$ and having a countable
dense subset $D$ of continuous functions. Let also $(\ee_n)_n$ be
a decreasing sequence of positive reals with $\ee_n\to 0$. Assume
that the constant function $0$ is a non-$G_\delta$ point of $\kk$.
Then there exists a family $\{u_t\}_{t\in\ct} \subseteq
\mathrm{conv}(D)$ equivalent to the canonical dense family of
$\alex$ such that, setting $g_\sg=\lim_n u_{\sg|n}$ for all
$\sg\in 2^\nn$, the following are satisfied.
\begin{enumerate}
\item[(1)] The function $0$ is the unique non-$G_\delta$ point of
$\overline{\{u_t\}}^p_{t\in\ct}$. \item[(2)] The family
$\{u_t\}_{t\in\ct}$ is a tree basis with respect to the supremum
norm. \item[(3)] The family $\{g_\sg:\sg\in 2^\nn\}$ is a subset
of $\kk$ and 1-unconditional. \item[(4)] For every $n\in\nn$, if
$\{t_0\prec ...\prec t_{2^n-1}\}$ is the $\prec$-increasing
enumeration of $2^n$, then for every $\{\sg_0,...,\sg_{2^n-1}\}
\subseteq 2^\nn$ with $t_i\sqsubset \sg_i$ for all
$i\in\{0,...,2^n-1\}$ we have that $(g_{\sg_i})_{i=0}^{2^n-1}$ is
$(1+\ee_n)$-equivalent to $(u_{t_i})_{i=0}^{2^n-1}$.
\end{enumerate}
\end{thm}
The proof of the above result is a slight modification of Theorem
17 in \cite{ADKbanach}, where we also refer the reader for more
information.

We close this subsection with the following result whose proof is
based on Stegall's construction \cite{St}.
\begin{thm}
\label{newtreebases} Let $X$ be a Banach space such that $X^*$ is
separable and $X^{**}$ is non-separable. Let also $\ee>0$. Then
there exists a family $\{u_t\}_{t\in\ct} \subseteq B_X$ such that
the following are satisfied.
\begin{enumerate}
\item[(i)] The family $\{u_t\}_{t\in\ct}$ is equivalent to the
canonical dense family of $2^{\leqslant\nn}$. \item[(ii)] For
every $\sg\in 2^\nn$, if $y^{**}_\sg$ is the weak* limit of
$(u_{\sg|n})_n$, then there exists $y^{***}_\sg\in X^{***}$ with
$\|y^{***}_\sg\|\leq 1+\ee$ and such that
$y^{***}_\sg(y^{**}_\sg)=1$ while $y^{***}_\sg(y^{**}_\tau)=0$ for
all $\tau\neq \sg$. \item[(iii)] For every $n\in\nn$, if
$\{t_0\prec...\prec t_{2^n-1}\}$ is the $\prec$-increasing
enumeration of $2^n$, then for every
$\{\sg_0,...,\sg_{2^n-1}\}\subseteq 2^\nn$ with
$t_i\sqsubset\sg_i$ for all $i\in\{0,...,2^n-1\}$, we have that
$(y^{**}_{\sg_i})_{i=0}^{2^n-1}$ is $(1+\frac{1}{n})$-equivalent
to $(u_{t_i})_{i=0}^{2^n-1}$.
\end{enumerate}
\end{thm}
\begin{proof}
Since $X^*$ is separable, we have that $(B_{X^{**}},w^*)$ is
compact metrizable. Fix a compatible metric $\rho$ for
$(B_{X^{**}},w^*)$. Using Stegall's construction \cite{St}, we get
the following.
\begin{enumerate}
\item[(C1)] A family $\{x^*_t\}_{t\in\ct}\subseteq X^*$, and
\item[(C2)] a family $\{B_t\}_{t\in\ct}$ of open subsets of
$(B_{X^{**}},w^*)$
\end{enumerate}
such that for all $t\in\ct$ the following are satisfied.
\begin{enumerate}
\item[(P1)] $1<\|x^*_t\|<1+\ee$. \item[(P2)]
$\overline{B}_{t^{\con}0}\cap
\overline{B}_{t^{\con}1}=\varnothing$,
$\overline{B}_{t^{\con}0}\cup \overline{B}_{t^{\con}1}\subseteq
B_t$ and $\rho-\mathrm{diam}(B_t)\leq \frac{1}{|t|+1}$.
\item[(P3)] For all $x^{**}\in B_t$,
$|x^{**}(x^*_t)-1|<\frac{1}{|t|+1}$. \item[(P4)] For all $t'\neq
t$ with $|t|=|t'|$ and for all $x^{**}\in B_{t'}$,
$|x^{**}(x^*_t)|<\frac{1}{|t|+1}$.
\end{enumerate}
By property (P2), for every $\sg\in 2^\nn$ we have that $\bigcap_n
B_{\sg|n}=\{x^{**}\}$. Moreover, the map $2^\nn\ni \sg\mapsto
x^{**}_\sg\in (B_{X^{**}},w^*)$ is a homeomorphic embedding. By
Goldstine's theorem, for every $t\in \ct$ we choose $x_t\in
B_t\cap X$. Notice that $w^*-\lim_{n} x_{\sg|n}=x^{**}_\sg$ for
all $\sg\in 2^\nn$. For every $\sg\in 2^\nn$ we choose
$x^{***}_\sg\in \bigcap_n\overline{\{x^*_{\sg|k}:k\geq n\}}^{w*}$.
By (P3) we see that $x^{***}_\sg(x^{**}_\sg)=1$ while, by (P4),
$x^{***}_\sg(x^{**}_\tau)=0$ for all $\tau\neq\sg$. Moreover, we
have that
\[ \sup\{ |\lambda_i|:i=0,...,n\}\leq (1+\ee) \Big\|\sum_{i=0}^n
\lambda_i x^{**}_{\sg_i} \Big\| \] for every $n\in\nn$, every
$\{\sg_0,...,\sg_n\}\subseteq 2^\nn$ and every
$(\lambda_i)_{i=0}^n \in \rr^{n+1}$. Arguing as in the proof of
Theorem 17 in \cite{ADKbanach}, we may construct a family
$\{u_t\}_{t\in\ct}\subseteq \mathrm{conv}\{ x_t:t\in\ct\}$ and a
regular dyadic subtree $S=(s_t)_{t\in\ct}$ of $\ct$ such that the
following are satisfied.
\begin{enumerate}
\item[(1)] For all $\sg\in 2^\nn$, the sequence $(u_{\sg|n})_n$ is
weak* convergent to $y^{**}_\sg$, where
\[ y^{**}_\sg=\lim_n x_{s_{\sg|n}}. \]
\item[(2)] For every $n\in\nn$, if $\{t_0\prec...\prec
t_{2^n-1}\}$ is the $\prec$-increasing enumeration of $2^n$, then
for every $\{\sg_0,...,\sg_{2^n-1}\} \subseteq 2^\nn$ with
$t_i\sqsubset \sg_i$ for all $i\in\{0,...,2^n-1\}$ we have that
$(y^{**}_{\sg_i})_{i=0}^{2^n-1}$ is $(1+\frac{1}{n})$-equivalent
to $(u_{t_i})_{i=0}^{2^n-1}$.
\end{enumerate}
For all $\sg\in 2^\nn$, let $\bar{\sg}=\bigcup_n s_{\sg|n}\in
2^\nn$. Setting $y^{***}_\sg=x^{***}_{\bar{\sg}}$ for all $\sg\in
2^\nn$, we see that properties (ii) and (iii) in the statement of
the theorem are satisfied. Finally, by passing to a regular dyadic
subtree if necessary, we also have that the family
$\{u_t\}_{t\in\ct}$ is equivalent to the canonical dense family of
$2^{\leqslant\nn}$, i.e. property (i) is satisfied. The proof is
completed.
\end{proof}
\begin{rem}
(1) We do not know if the family $\{u_t\}_{t\in\ct}$ obtained in
Theorem \ref{newtreebases} can be chosen to be Schauder basic or
an FDD. It seems also to be unknown whether for every Banach space
$X$ with $X^*$ separable and $X^{**}$ non-separable, there exists
a subspace $Y$
of $X$ with a Schauder basis such that $Y^{**}$ is non-separable. \\
(2) The family $\{y^{**}_\sg:\sg\in 2^\nn\}$ obtained in Theorem
\ref{newtreebases} cannot be chosen to be unconditional, as the
examples of non-separable HI spaces show (see \cite{AAT},
\cite{AT}). However, all these second dual, non-separable HI
spaces have quotients with separable kernel which contains
unconditional families of the cardinality of the continuum. The
following problem is motivated by the previous observation.
\end{rem}
\noindent \textit{Problem.} Let $X$ be a separable Banach space
with $X^{**}$ non-separable. Does there exist a quotient $Y$ of
$X^{**}$ containing an unconditional family of size $|X^{**}|$?


\end{document}